\documentclass[11pt]{amsart}
\usepackage{graphicx}
\usepackage{amssymb}
\usepackage{amsfonts}
\usepackage{hyperref}
\usepackage{mathrsfs}
\usepackage{xcolor}
\usepackage[margin=2.7cm]{geometry}

  [1995/08/01 v0.1 Double stroke roman fonts]

\def\ds@whichfont{dsrom}
\relax

\DeclareMathAlphabet{\mathds}{U}{\ds@whichfont}{m}{n}

\swapnumbers
\newtheorem{theorem}{Theorem}[section]
\newtheorem{lemma}[theorem]{Lemma}
\newtheorem{corollary}[theorem]{Corollary}
\newtheorem{proposition}[theorem]{Proposition}
\theoremstyle{definition}

\newtheorem{assumption}[theorem]{Assumption}

\newtheorem{remark}[theorem]{Remark}

\numberwithin{equation}{section}
\theoremstyle{plain}
\newtheorem{thm}{Theorem}[section]
\numberwithin{equation}{section} 
\numberwithin{figure}{section} 
\theoremstyle{plain}
\newtheorem{cor}[thm]{Corollary} 
\theoremstyle{plain}
\theoremstyle{remark}
\newtheorem*{acknowledgement*}{Acknowledgement}

\newcommand{\cA}{{\mathcal A}}
\newcommand{\cB}{{\mathcal B}}
\newcommand{\cC}{{\mathcal C}}
\newcommand{\cD}{{\mathcal D}}
\newcommand{\cE}{{\mathcal E}}
\newcommand{\cF}{{\mathcal F}}
\newcommand{\cG}{{\mathcal G}}

\newcommand{\cK}{{\mathcal K}}
\newcommand{\cL}{{\mathcal L}}

\newcommand{\cX}{{\mathcal X}}

\newcommand{\te}{{\theta}}

\newcommand{\Om}{{\Omega}}
\newcommand{\om}{{\omega}}
\newcommand{\ve}{{\varepsilon}}
\newcommand{\del}{{\delta}}
\newcommand{\Del}{{\Delta}}
\newcommand{\gam}{{\gamma}}

\newcommand{\sig}{{\sigma}}
\newcommand{\al}{{\alpha}}
\newcommand{\be}{{\beta}}
\newcommand{\ka}{{\kappa}}
\newcommand{\la}{{\lambda}}

\newcommand{\bbE}{{\mathbb E}}

\newcommand{\bbN}{{\mathbb N}}
\newcommand{\bbP}{{\mathbb P}}
\newcommand{\bbR}{{\mathbb R}}

\newcommand{\bbZ}{{\mathbb Z}}
\newcommand{\bbI}{{\mathbb I}}

\begin{document}
\title[]{Spectral methods for limit theorems for random expanding transformations}
 \vskip 0.1cm
 \author{Yeor Hafouta \\
\vskip 0.1cm
Department  of Mathematics\\
The  University of Florida}
\email{yeor.hafuta@ufl.edu}

\maketitle
\markboth{Yeor Hafouta}{Limit theorems
}
\renewcommand{\theequation}{\arabic{section}.\arabic{equation}}
\pagenumbering{arabic}

\begin{abstract} 
We extend the  spectral method for proving limit theorems to random non-uniformly expanding dynamical systems. This yields the
 CLT  and moderate deviations principles (MDP). We show that as the amount of non-uniformity decreases the CLT rates and the speed in the MDP  become closer to the optimal ones.
For smooth systems the rates are effective. Compared to recent progress on the subject \cite{Adv} we are able to consider much more general maps, Gibbs measures and observables. However,  our main results are new even in the setup of \cite{Adv}.  
 \end{abstract}




\section{Introduction}\label{Sec 1}
An important discovery made in the last century is that autonomous hyperbolic dynamical systems could exhibit stochastic behaviour. One of the most celebrated results in this direction is the fact that appropriately normalized Birkhoff sums could satisfy the central limit theorem (CLT). Since then many other probabilistic limit theorems have been obtained for autonomous systems.
However, many systems appearing in nature are non-autonomous  due to an interaction with the outside world. 
Such systems can be better described by compositions of
different maps $T_j$, rather than by repeated application of exactly the same transformation. Namley, the $j$-th iterate of the system is given by 
$
T_{j-1}\circ\cdots\circ  T_{1}\circ T_0.
$
Yet, many powerful  tools  developed for studying autonomous systems  are unavailable in non-autonomous settings, so very often new ideas
are needed to handle the non-stationary case.


One notable example of non-autonomous systems are random dynamical systems.
Random transformations emerge in a natural way as a model
for description of a physical system whose evolution mechanism depends on
time in a stationary way. This leads to the study of actions of compositions
of different maps $T_j$ chosen at random from a typical sequence of transformations. To fix the notation, we have an underlying probability space
$(\Om,\cF,\bbP)$ and a probability preserving map $\te:\Om\to\Om$ so that the $n$-th iteration the system on a fiber $\om\in\Om$ is given by 
$$
T_\om^n:=T_{\te^{n-1}\om}\circ\cdots\circ T_{\te\om}\circ T_\om.
$$
This setup was discussed already in Ulam and von Neumann \cite{UN} and in Kakutani \cite{Ka} in connection with random ergodic theorems. 
Since then the ergodic theory of random dynamical systems has attracted a lot of attention, see, for instance \cite{Arnold98, Cong97, Crauel2002, Kifer86, KiferLiu, LiuQian95}. We refer to  the introduction of \cite[Chapter 5]{KiferLiu} for a historical discussion and applications to, for instance, statistical physics, economy and meteorology.  


Probabilistic limit theorems (aka statistical properties)  of random  dynamical systems have attracted a lot of attention in literature in the past two decades. For instance, decay of correlations were obtained in \cite{ABR,Baldi, Buzzi}. The central limit theorem for  iterations of random hyperbolic transformations  chosen
independently with identical distribution (iid) was considered in  \cite{ANV, ALS09}. In this case the orbits in the state space
form a Markov chain (\cite{Kifer86}) and limit theorems
 are obtained relying on stationary methods which involve the spectral gap of an appropriate annealed (averaged) transfer operator. Another approach for iid transformation is based on construction of random Young towers \cite{ABR,Baldi,Su2, Su1}, which exist only in particular situations.

 Beyond iterations of random iid maps, limit theorems were mostly obtained for quite general classes of random uniformly expanding transformations  and for some classes of random uniformly hyperbolic maps, see, for instance \cite{DavorNonlin, DavorCMP, DavorTAMS, DavH AHP, HK, Nonlin, YH YT, Kifer 1998, Ste20, STE,STSU}. In  fact, in the uniform case there are also a few results for compositions of non-random sequences of maps. In particular the results in \cite{Bhak} cover certain types of random uniformly hyperbolic maps, and the results in \cite{DolHaf} cover most of the results for random uniformly expanding maps. 
 
For random non-uniformly expanding maps the situation is more complicated. 
In \cite{Cogburn, Kifer 1998} an inducing approach for random transformations or Markov chains in random environment was presented. However, verifying the conditions of these results beyond the uniform case requires some integrability conditions on the first return time to relatively complicated sets. 
Another   approach \cite{DHS,DS}  shows that for general ergodic random maps random Birkhoff sums $S_n^\om f=\sum_{j=0}^{n-1}f_{\te^j\om}\circ T_\om^j$
obey many limit theorems for expanding on average maps $T_\om$ under a scaling condition $\text{esssup}_{\om\in\Omega}(K(\om)\|f_\om\|)<\infty$ for an appropriate tempered random variable (i.e. $K(\te^n\om)$ grows subexponentially fasy).
However, $K(\om)$ essentially comes from Oseledets theorem and it is not computable, which makes the scaling condition hard to verify. 
To the best of our knowledge, the only paper where some explicit examples of non-uniformly expanding maps and random functions $f_\om$ satisfying the CLT is \cite{Adv} (part of the results are based on verifying the conditions in \cite{Kifer 1998}). The main examples in \cite{Adv} were certain classes of full branch (partially) expanding maps, considered as stochastic processes with respect to random Gibbs measure whose potential has sufficiently small (local) oscillation. For instance, this includes small piecewise perturbations of random piecewise expanding affine maps with full branches. In this case the random Gibbs measure is the unique absolutely continuous equivariant random measure. The results were obtained for iterates of random weakly dependent maps, which is necessary in general in order to get limit theorems (see the example in \cite[Appedix A]{DHS}).

In this paper we prove  limit theorems for much more general random weakly dependent expanding maps and  measures.  
Under some mixing conditions on the random environment we first prove (see Theorem \ref{CLT}) a quenched CLT for random Birkhoff sums of the form
\begin{eqnarray*}
S_n^\om f=\sum_{j=0}^{n-1}f_{\te^j\om}\circ T_{\te^{j-1}\om}\cdots\circ\cdots T_{\te\om}\circ T_\om
 \end{eqnarray*}
 where $f_\om$ is a random H\"older continuous function and $T_\om$ are locally expanding maps. The randomness comes by substituting random points  distributed according to random equilibrium states (Gibbs measures). For $C^2$ expanding maps on a compact connected Riemannian manifold this includes but not limited to the unique family  of random absolutely continuous equivariant measures (a.c.e.m.).
 Except for the weak dependence conditions, our conditions only involve integrabilty assumptions on random variables like the norm of $f_\om$, the maximal amount of expansion of $T_\om$, the reciprocal of the minimal amount of expansion and similar ones.
 In fact, we prove the following ``continuity" of the rates by showing that the rates in the self normalized CLT get closer
to the optimal $O(n^{-1/2})$ rates as the regularity of the above random variables increases (see Theorem \ref{BE1}). The general optimal rates $O(n^{-1/2})$ are not expected because of the sizes of individual summands $f_{\te^j\om}\circ T_\om^j$ and the non uniform mixing time. 
In the special case of $C^2$ expanding maps and a.c.e.m. we prove effectivity of the rates, which means that the rates have the form $K_b(\om)n^{-(1/2-1/b)}$, with $K_b\in L^b$ and $b\to\infty$ as the amount of regularity of the random variables increases (see Theorem \ref{BE2}). 
 As a byproduct of our methods we can also prove moderate deviations with close to optimal speed (Theorem \ref{MDP}). All the results are new even in the setup of \cite{Adv}, while the effective ones (Theorem \ref{BE2}) seems to be new also for uniformly expanding maps. 


Before discussing our methods recall that by Levi's continuity theorem in order to show that a sequence of random variables $W_n$ converges in law towards the standard normal distribution it is enough to show that  $\mathbb E[e^{itW_n}]\to e^{-\frac12t^2}$ for all $t\in\bbR$, namely that the corresponding Fourier transforms converge towards the Fourier transform of the  standard normal law. In general, CLT rates follow by quantifying the above convergence and  using  a quantified Fourier inversion formula (Esseen's inequality \cite[Ch. XVI.3]{Feller}).

The spectral method is
one of the most powerful tools  to prove  limit theorems for homogeneous Markov chains and autonomous dynamical systems in the past seven decades. This method is the mechanism that makes Fourier methods effective in these settings.
Our results will be obtained by showing how  the spectral method works for random non uniformly expanding transformations.
For autonomous systems this method is very often referred  as the Nagaev-Guivarch method (\cite{GH,Na}), and it is based on a spectral gap for an appropriate Markov or transfer operator $L$  together with a perturabtive argument \cite{Kato} which extends the spectral gap to the complex operators $L_t(g)=L(ge^{it f})$ whose iterates control the characteristic function of the Birkhoff sum $S_nf=\sum_{j=0}^{n-1}f\circ T^j$. The spectral method has been applied for many  stationary systems, and we refer to \cite{HH} for a general framework (see also \cite{G,HP}). In recent years ``spectral" methods have been extended to uniformly expanding  or hyperbolic random and sequential systems and inhomogenious Markov chains. Namely,  sequential (random) complex Perron-Frobenius theorems were established in \cite{DolHaf,DavorCMP, DavorTAMS,DavH AHP, HK, YH YT, Nonlin}. In the non-autonomous case there is a sequence of operators instead of a single one, and this tool replaces the spectral gap of the complex operators  and it is the mechanism behind the optimal CLT rates in \cite{DolHaf,HK,Nonlin,DavH AHP, YH YT} and  other results. This method does not work effectively without uniform expansion of the random dynamics. This is the main difficulty to overcome in this paper, where already effective real Perron-Frobenius  theorems are obtained for the first time (and they several other applications \cite{New Davor1, New Davor2, YH LLT 24}, see also \cite[Section 9]{YH 23}). 
All of this is described in detail in the next paragraph. 
\vskip0.0435cm

\paragraph{\textbf{Outline of the proof}.}
 We first prove effective polynomial   Perron-Frobenius  theorems (convergence rates)  for the iterates of the random transfer operators of $T_\om$, see Theorems \ref{RPF Poly} and \ref{OSC1}. This is already enough for the CLT and several  other results. Similarly to many other papers  (e.g. \cite{Liver} for deterministic maps and \cite{Buzzi, HK, Kifer 1996,  Kifer Thermo, MSU, Varandas} for random ones) the rates are obtained using Birkhoff's method \cite{Bir} of contraction of projective metrics associated with real cones. However, effectiveness of the rates follows from a combination of an inducing approach and mixing conditions on the base maps which allow us to obtain quantitative estimates on the amount of contraction after $n$ iterates.
The second step is an extension to the non-uniform case of the random complex
Perron-Frobenius theorem
(i.e. the perturbative approach). However, in view of the non-uniform and sub-exponential decay there are several obstacles. To overcome this we employ a double inducing scheme. We first induce on a level set of an appropriate random variable. Then we take a purely sequential approach and  join together several of the blocks generated by the visiting times to this level set. This will exponentialize the convergence rates.  After that, using a change of cones and a complex conic perturbative argument  (using the theory of \cite{Dub, Rug}),  we are able to prove a   complex ``sequential" Perron-Frobenius theorem for the array  of transfer operators generated by finitely many of the above blocks.
This theorem has uniform estimates but it is non-uniform in the parameter space $t$.
Using ideas in \cite{HK}, this  yields non-effective close to $O(n^{-1/2})$ rates in the self-normalized CLT, i.e. when normalizing by $\sqrt{\text{Var}(S_n^\om f)}$. 
To get effective rates and to pass to the deterministic normalization $n^{1/2}$ the third step is composed of obtaining effective convergence rates towards the limit $\sigma^2=\lim_{n\to\infty}n^{-1}\text{Var}(S_n^\om f)$, as well as effective growth rates of the visiting times to the level set.
The fourth and last step is composed of approximating the sum $S_n^\om f$ by corresponding Birkhoff sums with respect to the system generated by the  blocks.

\section{Preliminaries and main results}\label{Sec 2}
\subsection{Limit theorems}
Let us begin with the random environment. Let $(\Omega,\mathcal F,\mathbb P,\te)$  be the shift system corresponding to a stationary sequence of random variables $X=(X_k)_{k\in\mathbb Z}$ taking values on some measurable space $\Om_0$. Then $\Omega=\Omega_0^\bbZ$,\, $\te$ is the left shift and $\om=(\om_k)_{k\in\bbZ}$ and $X$ induce the same law on $\Omega$. Henceforth we abbreviate $L^p=L^p(\Omega,\cF,\bbP)$ for $p>0$.

Now we present some mixing (weak dependence) coefficients of $(X_k)$. Let $(\Om_1,\mathcal F_1,\mathbb P_1)$ be a probability space on which the sequence is defined. For $A\subset \bbZ$, let $\sigma\{X_i: i\in A\}$ be the $\sigma$-algebra generated by  $\{X_i: i\in A\}$.
Recall that $\alpha(m)$ and $\psi_U(m)$ are the smallest numbers  such that
\begin{eqnarray*}
&|\mathbb P_1(A\cap B)-\mathbb P_1(A)\bbP_1(B)|\leq \alpha(m)\,\,\,\text{ and }\\   
&\bbP_1(A\cap B)-\bbP_1(A)\bbP_1(B)\leq \bbP_1(A)\bbP_1(B)\psi_U(m)
\end{eqnarray*}
for all sets $A,B$ with $A\in\sigma\{X_0,...,X_k\}$ and $B\in\sigma\{X_{k+m}, X_{k+m+1},...\}$ for some $k\in\mathbb N$. For every $p>1$ we consider the following mixing condition.
\begin{assumption}\label{Mix A}
Either  $\lim_{n\to\infty}\psi_U(n)=0$ or $\alpha(n)=O(n^{-p})$. 
\end{assumption}

To simplify the presentation in this section we consider only two classes of maps. 
First we present a class of smooth expanding maps. Let $M$ be a compact connected Riemannian manifold and let $T:\Omega\times M\to \Om\times M$ be such that for $\mathbb P$-a.a. $\om$ the map $T_\om=T(\om,\cdot):M\to M$ is of class $C^2$, it is surjective and $DT_\om$ is invertible. 
We assume that
 $\beta(\om):=\|(DT_{\om})^{-1}\|\leq 1$ almost surely but $\mathbb P(\beta(\om)=1)<1$.

Let us also recall  the definition of a  random subshift of finite type (SFT).
Let $d_\om$ be a random variable and set $\cA_\om=\{1,2,...,d_\om\}$. Let $A^\om$ be a measurable\footnote{Namely, for every $k,\ell\in\bbN$ the map $\om\mapsto A_{k,\ell}^\om\bbI(k\leq d_\om, \ell\leq d_{\te\om})$ is measurable.} family of  $d_\om\times d_{\te\om}$ matrices with $0-1$ entries. We assume that there is a random variable $m(\om)$ such that $\bbP$-a.s. the matrix $A^{\te^{m(\om)-1}\om}\cdots A^{\te\om}\cdot A^{\om}$ has only positive entries. 
Define 
$$
\cE_\om=\{x=(x_i)_{i=0}^\infty:x_i\in \cA_{\te^i\om}, A^{\te^i\om}_{x_i,x_{i+1}}=1,\,\,\forall i\}. 
$$
Then $\cE_\om$ is a 
closed embedded subset of the compact space 
$M=\bar\bbN\times \bar\bbN\cdots$
 which is the
infinite product of the one-point compactification $\bar \bbN=\bbN\cup\{\infty\}$  of $\bbN$ 
with the metric on $M$ given by
\begin{eqnarray*}
&d(x,y)=\sum_{i=0}^{\infty}e^{-i}\left|1/x_i-1/y_i\right|.    
\end{eqnarray*}
Let $T_\om:\cE_\om\to\cE_{\te\om}$ be the left shift defined by 
$
T_\om(x)=(x_{i+1})_{i=0}^\infty,\,\, x=(x_i)_{i=0}^\infty.
$

Next, let us fix some $p>2$ which  will measure the amount of non-uniformity.
\begin{assumption}\label{Mom A}
We have $\beta(\om)^{-1},\|D(T_{\om})\|,d_\om, m(\om) \in L^p$.
\end{assumption}

Let us also fix a  $\al\in(0,1]$ and denote by $v_\al(g)$ the corresponding H\"older constant of a function $g$ on $M$. Set $\|g\|_\al=\|g\|_\infty+v_\al(g)$.
For both classes of maps, let $\phi:\Omega\times M\to \bbR$ be a measurable function such that  $\|\phi(\om,\cdot)\|_\al\in L^{p}$. Let $\mu$ be the  Gibbs measure corresponding to $\phi$. That is (see \cite{Kifer Thermo}), if $\tau(\om,x)=(\te\om, T_\om x)$ and $\mu=\int\mu_\om d\mathbb P(\om)$ then  $\mu$ is
the unique $\tau$-invariant probability measure maximizing in the variational principle so that
\begin{equation}\label{Equib}
\pi_{\tau}(\phi;\mu):=h_{\mu}(\tau)+\int \phi  d\mu=\sup_{\nu}\pi_{\tau}(\phi;\nu)    
\end{equation}
where the last supremum is taken over all $\tau$-invariant probability measures $\nu$, and so
$\mu$ is the unique equilibrium state for $\phi$. 

Next we introduce our approximation conditions. Let $\cF_{r}$ be the sub-$\sig$-algebra generated by the projection on the coordiantes $\om_j, |j|\leq r$. For $Y\in L^p$  denote $\mathfrak{a}_p(r;Y)=\|Y-\bbE[Y|\cF_r]\|_{L^p}$. For  $B\in\cF$ denote $\mathfrak d(B,\cF_r):=\inf\{\bbP(B\Delta A): A\in\cF_r\}$.
\begin{assumption}\label{Approx A}
For $Y(\om)=\|\phi_\om\|_\al, \beta(\om), \|DT_\om\|,d_\om$   we have $\mathfrak{a}_p(r;Y)=O(r^{-p})$.  Moreover, for subshifts,  for every level set $L_M=\{\om: m(\om)\leq M\}$ we have $\mathfrak d(L_M,\cF_r)=O(r^{-p})$.
\end{assumption}
 Note that $\mathfrak d(L_M,\cF_r)=0$ for all $r\geq M$ if $A^\om$ depends only on $\om_0$.
Next, let us  take a measurable function $f:\Omega\times M\to \bbR$ 
such that
$\|f_\om\|_{\alpha}\in L^{p_1}, p_1>2$, where $f_\om=f(\om,\cdot)$. Suppose  $\tilde f(\om,x)=f(\om,x)-\mu_\om(f_\om)$ is not an $L^2(\Omega\times M,\mu)$ coboundary with respect to $\tau$.
For $n\in\bbN$, set 
$$
 S_n^\om f=\sum_{j=0}^{n-1}f_{\te^j\om}\circ T_{\te^{j-1}\om}\cdots\circ\cdots T_{\te\om}\circ T_\om, \,\,\, \hat S_n^\om f=n^{-1/2}(S_n^\om f-\mu_\om(S_n^\om f)).   
 $$
We consider $S_n^\om f$  as random variables on the probability space $(M,\text{Borel},\mu_\om)$.
 Our first result is the central limit theorem (CLT).
\begin{theorem}\label{CLT}
Under Assumptions \ref{Mix A}--\ref{Approx A}, there exists  $\varepsilon(p)>0$ with $\lim_{p\to\infty}\varepsilon(p)=0$ such that  when $p_1\geq 2+\varepsilon(p)$,   for $\bbP$-a.a. $\om$  and all  $t\in\bbR$,
 $$
\lim_{n\to \infty}\mu_\om\{\sigma^{-1}\hat S_n^\om f\leq t\}=\frac{1}{\sqrt{2\pi}}\int_{-\infty}^t e^{-\frac12x^2}dx:=\Phi(t)
 $$
 where the number $\sigma>0$ is the a.s. limit $\sigma=\lim_{n\to\infty}\|\hat S_n^\om\|_{L^2(\mu_\om)}$.
\end{theorem}
By keeping track of the constants in the proof  one can show that we can take $\varepsilon(p)=\frac{c}{\sqrt{p}}$ for some absolute constant $c$. Now we discuss rates in the self normalized CLT.
\begin{theorem}\label{BE1}
Let Assumptions \ref{Mix A}--\ref{Approx A} hold  and suppose $p_1=p$.
  Let $\bar S_n^\om f=S_n^\om f-\mu_\om(S_n^\om f)$ and $\Sigma_{\om,n}=\|\bar S_n^\om f\|_{L^2(\mu_\om)}$. Then for $\bbP$-a.a. $\om$,
  $$
\sup_{t\in\mathbb R}\left|\mu_\om\{\bar S_n^\om f\leq \Sigma_{\om,n}t\}-\Phi(t)\right|=O(n^{-(\frac{1}2-\delta(p))})
  $$
 where $\delta(p)\to 0$ as $p\to\infty$. A precise formula for $\delta(p)$ can be recovered from the proof.
\end{theorem}
We can also get close to optimal moderate deviations principles.
\begin{theorem}\label{MDP}
Under the conditions of Theorem \ref{BE1} we have the following.
 Let $(a_n)$ be a sequence such that 
$
\lim_{n\to\infty}a_n=\infty\,\,\text{ and }\,\,a_n=o(n^{\frac12-\delta(p)}). 
$
Denote $W_n^\om=\frac{\bar S_n^\om f}{n^{1/2}a_n}$. Then, for $\bbP$-a.a. $\om$
for every Borel-measurable set $\Gamma\subset\bbR$,
$$
-\inf_{x\in\Gamma^o}\frac 12x^2\leq \liminf_{n\to\infty}\frac{1}{a_n^2}\ln \mu_\om\{W_n^\om\in \Gamma\}\leq \limsup_{n\to\infty}\frac{1}{a_n^2}\ln \mu_\om\{W_n^\om\in \Gamma\}\leq -\inf_{x\in\overline\Gamma}\frac 12x^2
$$
where $\overline{\Gamma}$ is the closure of $\Gamma$ and $\Gamma^o$ is its interior.
\end{theorem}
Theorems \ref{CLT}--\ref{MDP}  hold in greater generality, see Section \ref{Maps sec}.
Random subshfits model other systems like random expanding Markov interval maps  or finite state  Markov chains in random environments (see \cite[Section 3.2]{YH 23}). Thus Theorems \ref{CLT}--\ref{BE1} also hold for these processes.

For $C^2$ expanding maps 
let
$\phi(\om,x)=-\ln|\text{Jac}(D_x T_\om)|
$ 
and suppose $\|\phi_\om\|_\al\in L^{p}$. In this case (see \cite[Theorem 2.2]{Kifer Thermo}) $\mu_\om$ is the unique random absolutely continuous measure with $(T_{\om})_*\mu_{\om}=\mu_{\te\om}$, $\mathbb P$-a.s.  Let $f_\om$ be as above. The following result is an effective version of Theorem \ref{BE1}.

\begin{theorem}\label{BE2}
For  $C^2$ expanding maps, 
suppose that both $\om\to T_\om$ and $\om\to f_\om$ depend only on $\om_0$. Under Assumptions \ref{Mix A} and \ref{Mom A} and when $p_1=p$
there  is $\delta(p)>0$ with $\lim_{p\to\infty}\delta(p)=0$
and   $K_{p}(\cdot)\in L^{1/\delta(p)}$
such that for $\mathbb P$-a.a. $\om$,
$$
\sup_{t\in\mathbb R}\left|\mu_\om\{\sigma^{-1}\hat S_n^\om f\leq t\}-\Phi(t)\right|\leq K_{p}(\om)n^{-(\frac{1}2-\delta(p))}.
$$
A precise formula for $\delta(p)$  can be recovered from the proof.
\end{theorem}


\subsection{Effective random real Ruelle-Perron-Frobenius rates}

Let $L_\om$ be the transfer operator of $T_\om$ with respect to $\mu_\om$ and $\mu_{\te\om}$, namely
\begin{equation}\label{Dual}
 \int_{\cE_\om} g\cdot (f\circ T_\om)d\mu_\om=\int_{\cE_{\te\om}}(L_\om g)fd\mu_{\te\om}.   
\end{equation}
 for all bounded measurable functions $f,g$ on $M$.
Set $L_\om^n=L_{\te^{n-1}\om}\circ\cdots\circ L_{\te\om}\circ L_{\om}$.
The following result is an important step in the proof of the Theorems \ref{CLT}--\ref{BE2}.
\begin{theorem}\label{RPF Poly}
Let Assumptions \ref{Mix A}-\ref{Approx A} be in force. Suppose  $\|\phi_\om\|_\al\in L^p$.

(i) There exists  $\beta(p)>0$ with $\lim_{p\to\infty}\beta(p)=\infty$ and  $R_{p}\in L^{\beta(p)}$ such that
$$
\sup_{\|g\|_\al\leq 1}\|L_\om^ng-\mu_\om(g)\|_{\infty}\leq R_{p}(\om)n^{-\beta(p)}.
$$

(ii) There are random variables $A_{p}$ such that 
$$
\sup_{\|g\|_\al\leq 1}\|L_\om^ng-\mu_\om(g)\|_{\alpha}\leq R_{p}(\om)A_{p}(\om)n^{-\beta(p)}
$$
and for every $c$ large enough  the first visiting time $m_1(\om)$ to the set $\{\om: A_{p}(\om)\leq c\}$ satisfies $m_1(\om)\in L^{\beta(p)}$.
Moreover, $A_{p}=e^{Y(\om)}$ where $Y=Y_{p}\in L^{\beta(p)}$ and $\|Y-\bbE[Y|\cF_r]\|_{L^{\beta(p)}}=O(r^{-\beta(p)})$.
\end{theorem}
By keeping track of the constants in the proof one can show that for $p>4$ we can take $\beta(p)=c\sqrt p$ for some absolute constant $c$. 

Theorem \ref{RPF Poly} is a particular case of Theorem \ref{OSC1} which 
 holds true for more general expanding maps. Note that by invoking an appropriate version of Oseledets theorem we get exponential decay rates, but with random variables $R=R_p$ which are only tempered. However, it is unlikely that without additional regularity this could be used to prove even the CLT. Thus, the regularity properties of $R_p$ and $A_p$ are the important part of Theorem \ref{RPF Poly}.
We can also prove stretched exponential version of Theorems \ref{OSC1} and \ref{RPF Poly}   which can yield effective rates of order $O(e^{-cn^{1/2-\varepsilon}}), \varepsilon>0$. Since this is not needed for our main results we only refer to a previous version of this manuscript \cite{YH 23} (see Theorem 2.18 there).




\section{Auxiliary results-growth of products of stationary mixing sequences}\label{Section Aux}
Let $g:\Om\to[0,1]$ be a measurable function such that $\bbE[g]=\int g(\om)d\bbP(\om)<1$.
Let us define $g_{\om,n}=\prod_{j=0}^{n-1}g(\te^j\om)$, where $n\in\bbN$.
In this section we obtaine estimates of the form 
\begin{equation}\label{g g }
 \bbE[g_{\om,n}]=O(n^{-d}), d>0.  
\end{equation}
Then we will use a simple fact which for the sake of convenience is formulated as a lemma.

\begin{lemma}\label{Simple}
Let $q\geq 1$ and suppose  \eqref{g g } holds with $g^q$ instead of $g$. Let $b_n$ be a positive sequence such that 
$
\sum_{n\geq 1}u_n b_n^q<\infty.
$
Then  $A(\om):=\sup\{b_n g_{\om,n}: \,n\in\bbN\}\in L^q$. 
\end{lemma}
Note that 
$
g_{\om,n}\leq A(\om)b_n^{-1}
$
and so the lemma means that decay rates of the $q$-th moment of $g_{\om,n}$ as $n\to\infty$  determine the regularity in the a.s. growth rates.
\begin{proof}[Proof of Lemma \ref{Simple}]
As $(A(\om))^q\leq \sum_{n\geq 1}b_n^q g_{\om,n}^q$,  
$
\bbE[(A(\om))^q]\leq \sum_{n\geq 1}b_n^q\bbE[g_{\om,n}]\leq \sum_{n\geq 1} b_n^qu_n.
$
\end{proof}

Next, let $\mathcal F_{n,m}\subset\cF$ be the $\sigma$-algebra generated by the projection on the coordinates $\om_j, n\leq j\leq m$ .
Let us first recall the following result which follows from  \cite[Theorem A.5]{Hall Hyde} by induction.
\begin{lemma}\label{Basic alpha cor}
 Let $U_1,U_2,....,U_\ell$ be random variables with values in $[0,1]$ such that $U_i,$ is measurable with respect to $\cF_{n_i,m_i}$  where $n_i\leq m_{i}<n_{i+1}$ for all $i$. Then 
 \begin{eqnarray*}
&\left|\bbE\left[\prod_{j=1}^{\ell}U_j\right]-\prod_{j=1}^\ell\bbE[U_j]\right|\leq 4\sum_{i=1}^{\ell-1}\al(n_{i+1}-m_i).     
 \end{eqnarray*}
\end{lemma}

Next, let
$
\beta_q(r)=\|g-\bbE[g|\cF_r]\|_{L^q(\Om,\cF,\bbP)}.
$
\begin{lemma}\label{alpha exp lemm}
If $\beta_1(r)+\al(r)=O(r^{-M})$ for some $M>0$ then for every $0<x<1$ we have
$
\bbE[g_{\om,n}]=O(n^{1-x(M+1)}). 
$
\end{lemma}
\begin{proof}
Let $n\in\bbN$ and let us take $r\leq \frac n3$. Since $0\leq g(\om)\leq 1$ we have 
$
\bbE[g_{\om,n}]\leq \bbE\big[\prod_{j=1}^{[\frac{n-1}{3r}]}g(\te^{3rj}\om)\big].
$
Moreover, 
$$
\left|\prod_{j=1}^{[\frac{n-1}{3r}]}g(\te^{3rj}\om)-\prod_{j=1}^{[\frac{n-1}{3r}]}\bbE[g(\te^{3rj}\om)|\cF_{3rj-r,3rj+r}]\right|
\leq \sum_{j=1}^{[\frac{n-1}{3r}]}|g(\te^{3rj}\om)-\bbE[g(\te^{3rj}\om)|\cF_{3rj-r,3rj+r}]|    
$$
where we used that $|\prod_{j=1}^d a_j-\prod_ {j=1}^db_j|\leq \sum_{j=1}^d|a_j-b_j|$ for all numbers $a_j,b_j\in[0,1]$.
Thus 
\begin{eqnarray*}
 & \bbE[g_{\om,n}]\leq \left[\frac{n-1}{3r}\right]\beta_1(r)+\bbE\left[\prod_{j=1}^{[\frac{n-1}{3r}]}\bbE[g(\te^{3rj}\om)|\cF_{3rj-r,3rj+r}]\right].  
\end{eqnarray*}
Hence, by Lemma \ref{Basic alpha cor}, 
$
\bbE[g_{\om,n}]\leq \left[\frac{n-1}{3r}\right](\beta_1(r)+\al(r))+\left(\bbE[g]\right)^{[\frac{n-1}{3r}]}.
$
Now the proof  follows by taking $r=[n^{x}]$ for $0<x<1$.
\end{proof}

Next, we prove some expectation estimates which involve the coefficients $\psi_U$. Recall the following   result (see \cite[Lemma 60]{Adv}):
\begin{lemma}\label{psi Lemm 2}
 Let $I_i=[a_i,b_i], i=1,2,...,d$ be intervals in the positive integers so that $I_j$ is to the left of $I_{j+1}$ and the distance between them is at least $L$. Let $Y_1,Y_2,...,Y_d$ be nonnegative bounded random variables so that $Y_i$ is measurable with respect to $\cF_{a_i,b_i}$. Then 
 \begin{eqnarray*}
&\bbE\left[\prod_{i=1}^{d}Y_i\right]\leq\left(1+\psi_U(L)\right)^{d-1}\prod_{i=1}^{d}\bbE[Y_i].    
 \end{eqnarray*}
 \end{lemma}

Using Lemma \ref{psi Lemm 2}  we can prove the following result.
\begin{lemma}\label{g exp lemm}
Suppose that 
\begin{eqnarray}\label{psi g cond}  
&\limsup_{k\to\infty}\psi_U(k)<\frac{1}{\bbE_\bbP[g]}-1
\end{eqnarray}
and $\beta_1(r)\leq Cr^{-A}$ for some $C>0$ and $A>1$. Then for every $\ve\in(0,1)$ there is a constant $C_\ve>0$ such that  
$
\bbE_\bbP[g_{\om,n}]\leq  C_\ve n^{1-\ve A}
$
 for all $n\in\bbN$.
\end{lemma}
\begin{proof}

Let us write 
$$
g_{\om,n}=\prod_{k=0}^{n-1}\bbE[g(\te^k\om)|\cF_{k-r,k+r}]+
\sum_{j=0}^{n-1}g_{\om,j-1}(g(\te^j\om)-\bbE[g(\te^j\om)|\cF_{j-r,j+r}])\prod_{k=j+1}^{n-1}\bbE[g(\te^k\om)|\cF_{k-r,k+r}].
$$
Using that $\bbE[g(\te^k\om)|\cF_{k-r,k+r}]$ also take values in $[0,1]$ we see that 
\begin{eqnarray*}
&\left|\bbE[g_{\om,n}]-\bbE\left[\prod_{k=0}^{n-1}\bbE[g(\te^k\om)|\cF_{k-r,k+r}]\right]\right|\leq n\beta_1(r).   
\end{eqnarray*}
Now, if $r<[n/3]$ then since $0\leq g(\om)\leq 1$,
\begin{eqnarray*}
&
\bbE\left[\prod_{k=0}^{n-1}\bbE[g(\te^k\om)|\cF_{k-r,k+r}]\right]\leq \bbE\left[\prod_{k=0}^{[(n-1)/3r]}\bbE[g(\te^{kr}\om)|\cF_{k-r,k+r}]\right]\leq\left((1+\psi_U(r))\bbE[g]\right)^{[\frac{n-1}{3r}]}.
\end{eqnarray*}
Note that, as opposed to part (i), the term $\beta_\infty(r)$ does not appear since we replaced $g$ by $\bbE[g|\cF_r]$.
Using \eqref{psi g cond}  we conclude that there is a constant $a>0$ such that if $r$  large enough then
$$
\bbE[g_{\om,n}]\leq n\beta_1(r)+e^{-an/r}.
$$
The result follows by taking $r=[n^\ve]$ for a given $\ve\in(0,1)$.
\end{proof}




Finally, the following result will be instrumental in the proofs of Theorem \ref{OSC1}.
\begin{lemma}\label{A Lemma}
Let $A\subset\Om$ be a measurable set with positive probability and the following property: for every $r\in\bbN$ there are  sets $A_r\in\cF_r$ and $B_r\in\cF$ such that
$$
\bbP(B_r)\leq Cr^{-A},\,\, A_r\subset A\cup B_r,\,\,
p_0:=\lim_{r\to\infty}\bbP(A_r)>0
$$
 for some $a>3$ and $C>0$.
Assume also that either
\begin{equation}\label{limsup ass}
(1+\limsup_{k\to\infty}\psi_U(k))(1-p_0(1-e^{-c}))<1 
\end{equation}
for some constant $c>0$, or 
\begin{equation}\label{al u b r}
\al(r)=O(r^{1-A}), \text{ as }r\to\infty.   
\end{equation}
Let $\ell(\om)$ be a random variable taking values in $\bbN$.
\vskip0.1cm
(i) For every $\varepsilon\in(0,1)$ there is a constant $c_1>0$ which depends only on $C,a,\eta$ and $\varepsilon$ such that for all $k,n\in\bbN$,
\begin{equation}\label{Exp approx}
 \bbE_{\bbP}\left[e^{-c\sum_{j=\ell(\om)}^{n}\bbI_A\circ \te^{kj}}\right]\leq c_1n^{2-\ve a}+\bbP(\ell(\om)\geq [n/2]).   
\end{equation}
\vskip0.1cm
(ii) Let $\be>0$ and $d\geq 1$ be such that $\be d<a-3$. If also  
 $\bbP(\ell\geq n)=O(n^{-\be d-1-\ve_0})$ for some $\ve_0>0$ then there exists a random variable $K\in L^d(\Om,\cF,\bbP)$ such that  
 $\bbP$-a.s. for all $n\in\bbN$,
\begin{equation}\label{Sup approx1}
e^{-c\sum_{j=\ell(\om)}^{n}\bbI_A(\te^{kj}\om)}\leq K(\om)n^{-\be}.    \end{equation}
\vskip0.1cm
\end{lemma}

\begin{proof}
(i) For all $s,n\in\bbN$  set $Y_{s,n}(\om)=e^{-c\sum_{j=s}^{n}\bbI_A(\te^{kj}\om)}$.
Then, since $Y_{s,n}\leq 1$ we see that
\begin{eqnarray*}
&\bbE_{\bbP}[Y_{\ell(\om),n}]=\sum_{s=1}^\infty\bbE_\bbP[\bbI(\ell=2)Y_{s,n}]\leq \sum_{s=1}^{[n/2]}\bbE_{\bbP}[\bbI(\ell=s)Y_{s,n}]+\bbP(\ell>[n/2])
\\
&\leq
\frac12 n\max_{1\leq s\leq [n/2]}\bbE_{\bbP}[\bbI(\ell=s)Y_{s,n}]+\bbP(\ell>[n/2])\leq
 n\max_{1\leq s\leq [n/2]}\bbE_{\bbP}[Y_{s,n}]+\bbP(\ell>[n/2]).
\end{eqnarray*}
Next, fix some $1\leq s\leq [n/2]$. 
Since the sequence   $(\bbI_A\circ \te^{kj})_{j}$ is stationary and $s\leq [n/2]$
we have 
$$
\bbE_{\bbP}[Y_{s,n}]=\bbE_{\bbP}[Y_{0,n-s}]\leq \bbE_{\bbP}[Y_{0,[n/2]}].
$$
Now, let $r<\frac 18 n$ and 
set $B_{r,n}=\bigcap_{j=1}^{n}\te^{-kj}(\Om\setminus B_r).
$
Then since $\bbI_{A_r}-\bbI_{B_r}\leq \bbI_{A}$ we have
$$
\bbE_{\bbP}[Y_{0,[n/2]}]=\bbE_{\bbP}\big[e^{-c\sum_{j=0}^{[n/2]}\bbI_A\circ \te^{kj}}\big]
\leq 
\bbE_{\bbP}\big[e^{-c\sum_{j=0}^{[n/2]}\bbI_{A_r}\circ \te^{kj}}\bbI_{B_{r,[n/2]}}\big]
+([n/2]+1)\bbP(B_r):=I_1+I_2.
$$
Note that by the assumption of the Lemma we have $I_2\leq ([n/2]+1)b_r$.
Let us now estimate $I_1$. By omitting the appropriate (negative) terms in the exponents we see that 
\begin{eqnarray*}
&I_1=\bbE_{\bbP}\big[e^{-c\sum_{j=0}^{[n/2]}\bbI_{A_r}\circ \te^{kj}}\big]\leq 
\bbE_{\bbP}\big[e^{-c\sum_{j=0}^{[\frac{n}{8r}]}\bbI_{A_r}\circ \te^{4rkj}}\big].
\end{eqnarray*}
Now, notice that $\bbI_{A_r}\circ \te^{4rkj}$ is measurable with respect to $\sig\{X_{s}: |s-4rkj|\leq r\}$. Thus, by applying either Lemma \ref{psi Lemm 2} or Lemma \ref{Basic alpha cor} we see that 
\begin{eqnarray*}
 &\bbE_{\bbP}\big[e^{-c\sum_{j=0}^{[\frac{n}{8r}]}\bbI_{A_r}\circ \te^{4rkj}}\big]\leq E_{r,n}:=\min\left(\left((1+\psi_U(2r)\bbE[e^{-c\bbI_{A_r}}]\right)^{[\frac{n}{8r}]},
\left(\bbE[e^{-c\bbI_{A_r}}]\right)^{[\frac{n}{8r}]}+[\frac{n}{8r}]\al(2r)
\right)
\\
&\leq \min\left(\left((1+\psi_U(2r)\bbE[e^{-c\bbI_{A_r}}]\right)^{[\frac{n}{8r}]},
\left(\bbE[e^{-c\bbI_{A_r}}]\right)^{[\frac{n}{8r}]}+nC_0r^{1-A}
\right).   
\end{eqnarray*}
for some constant $C_0$.
Now,  the estimate
\eqref{Exp approx} follows by taking $r=[n^\ve]$  and using \eqref{limsup ass} and that $\bbE_{\bbP}[e^{-cA_r}]=1-\bbP(A_r)(1-e^{-c})$.

(ii) 
To prove \eqref{Sup approx1} we  define 
$
Y(\om)=\sup_{n\geq 1}(n^{\be}e^{-c\sum_{j=0}^{n}\bbI_A(\te^{kj}\om)}).
$
Then, $Y^d\leq \sum_{n\geq 1}n^{\be d}e^{-c\sum_{j=0}^{n}\bbI_A\circ \te^{kj}}$ and so by applying \eqref{Exp approx} with $\ve$ close enough to $1$ so that $\be d+2-a\ve<-1$ we see that there is a constant $C_0$ such that 
$
\bbE_\bbP[Y^d]\leq C_0\sum_{n\geq 1}n^{\be d}(n^{2-a\ve}+n^{-\be d-1-\ve_0})<\infty.
$
  \end{proof}

\section{Effective random Ruelle-Perron-Frobenius theorems for exapanding maps}
We consider here a  class of random expanding maps that includes the ones in Section \ref{Sec 2} (see Sections \ref{SFT} and \ref{Smooth Sec}). 
\subsection{Random expanding maps: basic definitions and main assumptions}\label{Maps sec}
Let  $(\cX,d)$ be a bounded  metric space normalized in size so that 
$\text{diam}(\cX)\leq 1$, and  let $\cB$ be its the Borel $\sig$-algebra. Let us take a 
set $\cE\subset\Om\times \cX$ measurable with respect to the product $\sig$-algebra $\cF\times\cB$ such that the fibers $\cE_\om=\{x\in \cX:\,(\om,x)\in\cE\},\,\om\in\Om$ are random closed sets.
Next, let
$
\{T_\om: \cE_\om\to \cE_{\te\om},\, \om\in\Om\}
$
be a collection of  surjective maps between the metric spaces 
$\cE_\om$
and $\cE_{\te\om}$ so that
the map $(\om,x)\to T_\om x$ on $\cE$ is measurable with respect to the restriction $\cG$ of $\cF\times\cB$ to $\cE$. 
For every $\om\in\Om$ and $n\in\bbN$  consider the $n$-th step iterate $T_\om^n$ given by
\begin{equation}\label{T om n}
T_\om^n=T_{\te^{n-1}\om}\circ\cdots\circ T_{\te\om}\circ T_\om: \cE_\om\to\cE_{\te^n\om}.
\end{equation}

Let us take a  $\cG$-measurable function $\phi:\cE\to \bbR$ and define $\phi_\om:\cE_\om\to \bbR$ by $\phi_\om(x)=\phi(\om,x)$.
Henceforth,  functions of this form will be referred to as random functions. Let us fix some  $\al\in(0,1]$ and suppose that for $\bbP$-a.e. $\om$ the function $\phi_\om$ is H\"older continuous with exponent $\al$. 
The role of the random function $\phi_\om$ and its H\"older continuity is that under the right assumptions we will always consider the random Gibbs measure $\mu_\om$ corresponding to $\phi_\om$ (see in Section \ref{RPF Prem}).  

\subsubsection{Random expansion, the random potential and H\"older norms}
Let us first describe our local expansions assumptions of the maps $T_\om$. We assume that there are random variables $\xi_\om \in(0,1]$ and
$\gamma_\om>0$  such that, $\bbP$-a.s.
for every $x,x'\in\cE_{\te\om}$ with $d(x,x')\leq \xi_{\te\om}$ we can write 
\begin{eqnarray}
&T_\om^{-1}\{x\}=\{y_i=y_{i,\om}(x): i<k\},\,\,\,\,T_\om^{-1}\{x'\}=\{y_i'=y_{i,\om}(x'): i<k\}\,\,\,\text{ and}\label{Pair1.0}\\ & d(y_i,y_i')\leq \min((\gam_\om)^{-1}d(x,x'),\xi_\om)\label{Pair2.0}  
\end{eqnarray}
for all  $1\leq i<k=k(\om,x)$ (where either $k\in\bbN$ or $k=\infty$). To simplify\footnote{Note that we can always decrease $\xi_\om$ and force it to be smaller than $1$, but when we can take $\xi_\om=1$ then our setup allows maps with infinite degrees.} the presentation and proofs we suppose that either $\xi_\om<1$ for $\bbP$-.a.a. $\om$ or $\xi_\om=1$ for $\bbP$-a.a. $\om$. 
When $\xi_\om<1$ we also assume that there is a finite random variable $D_\om\geq 1$ such that 
$$
\deg (T_\om)=\sup\{|T_\om^{-1}(x)|: x\in\cE_{\te\om}\}\leq D_\om.
$$
In particular, in this case $k(\om,x)$ above is always finite. When $\xi_\om=1$ but $\deg (T_\om)=\infty$ we also assume that there is a random variable $D_\om\geq 1$ such that $\|\ell_\om\|_\infty e^{-\|\phi_\om\|_\infty}\leq D_\om$, where  $\ell_\om(x)=\sum_{y: T_\om y=x}e^{\phi_\om(y)}$.
Then in both cases we have  
$
\ell_\om\leq e^{\|\phi_\om\|_\infty}D_\om. 
$

Next, we discuss covering assumptions. When $\xi_\om<1$ we suppose that there is a positive integer valued random variable $m(\om)$ with the property that $\bbP$-a.s.
$$
T_\om^{m(\om)}(B_\om(x,\xi_\om))=\cE_{\te^{m(\om)}\om}
$$
for every $x\in\cE_\om$,  where $B_\om(x,\xi)$ is the ball of radius $\xi$ around $x$ inside $\cE_\om$. Notice that since $T_\om$'s are surjective, it follows that 
\begin{equation}\label{Cover}
T_\om^{n}(B_\om(x,\xi_\om))=\cE_{\te^{n(\om)}\om}
\end{equation}
for all $n\geq m(\om)$. Henceforth, when $\xi_\om=1$ we  set $m(\om)=0$.

Next, for all $\beta\in(0,1]$, a H\"older continuous function $f:\cE_\om\to\bbR$ with exponent $\be$ and $0<r\leq 1$ 
let  $\|f\|_\infty=\sup\{|f(x)|: x\in\cE_\om\}$ and 
$
v_{\be,r}(f)=\sup\left\{\frac{|f(x)-f(x')|}{d^\be(x,x')}: x,x'\in\cE_\om,\,0<d(x,x')\leq r\right\}.
$
Let
$\|f\|_{\be,r}=\|f\|_\infty+v_{\be,r}(f)$ be the corresponding H\"older norm 
and denote $v_{\be}=v_{\be,1}$ and $\|f\|_\be=\|f\|_{\be,1}$. We will constantly use the following simple result.
\begin{lemma}\label{Norm comp}
 For every function $f:\cE_\om\to\bbR$ we have 
 $$
v_{\be,\xi_\om}(f)\leq v_{\be}(f)\leq \max(2\|f\|_\infty\xi_\om^{-\be}, v_{\be,\xi_\om}(f))\,\text{ and  }\,\,
\|f\|_{\be,\xi_\om}\leq \|f\|_{\be}\leq 3\xi_\om^{-\al}\|f\|_{\be,\xi_\om}.
 $$
\end{lemma}
\begin{proof}
The first inequality $v_{\be,\xi_\om}(f)\leq v_{\be}(f)$ holds since $\xi_\om\leq 1$. The second inequality holds since when $d(x,y)>\xi_\om$ then 
$
|f(x)-f(y)|\leq 2\|f\|_\infty\leq 2\xi_\om^{-\be}d^\be(x,y).
$
\end{proof}

Next, let $\al\in(0,1]$ be such that
$\|\phi_\om\|_\al<\infty$ for $\bbP$-a.a. $\om$.
Let us take a random variable $H_{\om}\geq 1$ with
\begin{equation}\label{H prop}
\|\phi_\om\|_\al\leq H_\om,\,\bbP\text{-a.s.} 
\end{equation} 
Note that $\om\mapsto\|\phi_\om\|_\al$ is
 measurable in  (see \cite[Lemma 5.1.3]{HK}),
but our proofs work the same if we impose restrictions on certain approximation coefficients  of  upper bounds $H_\om$. This will come in handy in applications to, for instance, linear response.

Next, let  $N(\om)$ be any measurable upper bound on the $\al$ H\"older constant of the map $T_\om$, namely for $\bbP$-a.a. $\om$ we have
\begin{eqnarray}\label{N def}
&\sup\left\{\frac{d(T_\om x, T_\om x')}{d^\al(x,x')}: x,x'\in\cE_\om, x\not=x'\right\}\leq N(\om).    
\end{eqnarray}
We can always take $N(\om)$ to be the actual H\"older constant since it is clearly measurable in $\om$, but we will only need to impose approximation conditions on some measurable upper bound.  

\subsubsection{Recap:  construction of Gibbs measures}\label{RPF Prem}
In this section we  recall the construction of Gibbs measures and introduce the cocycle of transfer operators correspdoning to these measures. 

Let $\phi_\om$  be a random $\al$-H\"older continuous function, like in the previous section.
Let us consider the operator $\cL_\om$ which maps a function $g$ on $\cE_\om$ to a function $\cL_\om g$ on $\cE_{\te\om}$ given by
$$
\cL_\om g(x)=\sum_{T_\om y=x}e^{\phi_\om(y)}g(y).
$$

For the maps $T_\om$ and potentials $\phi_\om$ considered in this paper  we have the following result  which appeared in literature in various setups  (see \cite{HK, Kifer 1996, Kifer Thermo, MSU}). There is a triplet consisting of a positive random variable $\la_\om$, a strictly positive random H\"older continuous function $h_\om$ and a random probability measure $\nu_\om$ on $\cE_\om$ such  that $\bbP$-a.s. we have
$$
\nu_\om(h_\om)=1, \,\cL_\om h_\om=\la_\om h_{\te\om}, \,(\cL_\om)^*\nu_{\te\om}=\la_\om \nu_\om.
$$
Moreover, with $\cL_\om^n=\cL_{\te^{n-1}\om}\circ\cdots\circ\cL_{\te\om}\circ \cL_{\om}$ and $\la_{\om,n}=\prod_{j=0}^{n-1}\la_{\te^j\om}$ for every H\"older continuous function $g$ with exponent $\al$,
\begin{equation}\label{RPF0}
 \left\|(\la_{\te^{-n}\om,n})^{-1}\cL_{\te^{-n}\om,n}g-\nu_{\te^{-n}\om}(g)h_\om\right\|_\infty \leq C(\om)\|g\|_\al\del^n   
\end{equation}
for some random variable $C(\om)$ and a constant $\del\in(0,1)$.  The probability measure $\mu_\om=h_\om d\nu_\om$ is called the random Gibbs measure associated with $\phi_\om$.  
They are equivariant (i.e.$(T_\om)_*\mu_\om=\mu_{\te\om}$, $\bbP$-a.s.) random equilibrium states (as in\eqref{Equib}) and they satisfy the Gibbs property (see \cite{Kifer Thermo, MSU}). For non-singular maps when $\phi_\om=-\ln\text{Jac}(T_\om)$   then $\mu_\om$ is the unique random absolutely continuous 
 equivariant measure.

\subsection{Approximation, mixing and moments assumptions}

\begin{assumption}[Moment conditions]\label{Mom Ass}
For some $p>2,\, q,q_0>1$ and $b_2>1$ such that $q_0<q$ and $qq_0>p$ we have 
$$
\ln D_\om\in L^{qq_0}(\Om,\cF,\bbP), N(\om)\in L^{b_2}(\Om,\cF,\bbP), H_\om\in L^p(\Om,\cF,\bbP)
$$
where for a random variable $Y_\om$ and $a\geq 1$ we write $Y_\om\in L^a$ if $\om\mapsto Y_\om\in L^a(\Om,\cF,\bbP)$. 
\end{assumption}
Next,  for every $1\leq p\leq \infty$ we consider the following approximation coefficients
\begin{equation}\label{approx coef}
\beta_{p}(r)=\|\gamma_{\om}^{-\al}-\bbE[\gamma_{\om}^{-\al}|\cF_r]\|_{L^p(\Om,\cF,\bbP)}, \,\,\,\,\,h_p(r)=\|H_\om-\bbE[H_\om|\cF_r]\|_{L^p(\Om,\cF,\bbP)},
\end{equation}
$$
d_p(r)=\|\ln D_\om-\bbE[\ln D_\om|\cF_r]\|_{L^p(\Om,\cF,\bbP)},\,\,\,\,\,
n_p(r)=\|N(\om)-\bbE[N(\om)|\cF_r]\|_{L^p(\Om,\cF,\bbP)}.
$$
 When $H_\om, D_\om, \gamma_\om$ and $N(\om)$  are measurable with respect to $\cF_{-r_0,r_0}$ for some $r_0$ then all the coefficients vanish for $r>r_0$. However, 
our results  hold when $\beta_p(r), h_p(r), d_p(r), n_p(r)$ and $ b_p(r)$ decay fast enough as $r\to\infty$ for appropriate  $p$'s. 

Next, for every $u,b_0,\eta_1>0$, $a_0\in (0,1)$, and  $M>1$  consider one of the following conditions.
\begin{assumption}[Interplay between expansion and mixing]\label{Ass A u}
One of the following hold.
\vskip0.1cm
(i) $(\gamma_{\te^j\om})_{j\geq 0}$ is an iid sequence and $\bbE[\gamma_\om^{-\al u}]<1$;
\vskip0.05cm
or
\vskip0.05cm
(ii) $\gamma_\om\geq 1$, $\bbP(\gamma_\om=1)<1$ and 
\begin{eqnarray}\label{psi limsup}
&\limsup_{k\to\infty}\psi_U(k)<\min\left(\frac{1}{\bbE[\gamma_\om^{-\al u}]}, \frac 1{a_0}\right)-1;    
\end{eqnarray}
\vskip0.01cm
or 
\vskip0.01cm
(iii) $\gamma_\om\geq 1$, $\bbP(\gamma_\om=1)<1$ and 
\begin{eqnarray}\label{al 1}
&\al(r)=O(r^{-(M-1)}).    
\end{eqnarray}
\end{assumption}
Clearly \eqref{psi limsup}  holds for all $u$ and $a_0$ if $\lim_{k\to\infty}\psi_U(k)=0$, but we only require  \eqref{psi limsup}  with some $a_0$ which depends explicitly on the tails of  $D_\om, H_\om, \|\phi_\om\|_\al$ and $m(\om)$. 
The iid case corresponds to iid  $(X_j)_{j\in\bbZ}$ and maps $T_\om$ which depend only on $\om_0$ or when  $\bbE[\tilde\gamma_\om^{-\al u}]<1$, where 
$
\tilde\gamma_\om=\inf\{\gamma_{\om'}: \om'_0=\om_0\}.
$
In this case  the maps are expanding on average, and  in the circumstances of parts (ii) and (iii)  the maps are  weakly expanding, but   $T_\om$ is properly expanding with positive probability.  


Our first requirement from $m(\cdot)$ in \eqref{Cover} is as follows.
\begin{assumption}\label{Inner aprpox}
 For all $M_0\in\bbN$  and every $r\in\bbN$  there are sets $A_r=A_{r,M_0}$ measurable with respect to  $\cF_r$ and $B_r=B_{r,M_0}\in\cF$ such that, with $L_{M_0}=\{\om: m(\om)\leq M_0\}$ we have
\begin{equation}\label{A r req}
 A_{r}\subset L_{M_0}\cup B_{r}, \,\,\lim_{r\to\infty}\bbP(A_r)=\bbP(L_{M_0})   
\end{equation}
and $\bbP(B_{r})=O(r^{-M})$ for some $M>0$.
\end{assumption}
\begin{remark}
(i) Clearly $\lim_{r\to\infty}\bbP(A_r\Delta L_{M_0})=0$.
Since $\cF=\bigcup_{j=r}^\infty\cF_r$ there are always sets $A_{r}\in\cF_r$ and $B_r$ like in \eqref{A r req}  such that 
$\lim_{r\to\infty}\bbP(B_{r})=0$, so our assumption concerns the decay rates of $\bbP(B_{r})$. 
If  $\bbP(A_r\Delta L_{M_0})=O(r^{-M})$ then Assumption \ref{Inner aprpox} holds with $B_r=A_r\setminus L_{M_0}$, 
but this assumption has weaker requirements. 
\vskip0.1cm
(ii) When $\om\mapsto T_\om$ and  $\om\mapsto\xi_\om$ depend only on finitely many coordinates $\om_{j}$ then $L_{M_0}\in\cF_{-m_0,m_0}$ for some $m_0\in\bbN$. In this case we can take $B_r=\emptyset$ for all $r\geq m_0$. However, even when $T_\om$ depends only on $\om_0$ in our application to $C^2$ expanding maps in Section \ref{Smooth Sec} it seems unlikely that $L_{M_0}\in\cF_{-m_0,m_0}$ for some $m_0$, and instead we will control the decay rates of $\bbP(B_r)$. 
\end{remark}

Consider also  the following type of polynomial tails and approximation assumptions.
\begin{assumption}\label{Poly Ass}
With $p>2,\, q,q_0,b_2>1$ and  like in Assumption \ref{Mom Ass}
we have the following:

\vskip0.1cm
(i) $\bbP(m(\om)\geq n)=O(n^{-\beta d-1-\ve_0})$ for some $\beta,\ve_0>0$ and $d\geq q$ such that
$
\beta d+\ve_0>\max(\frac{p_0}{q_0}+p_0-1, v);
$
\vskip0.1cm
(ii)
either $\lim_{r\to\infty}\beta_\infty(r)=0$ or $\beta_{\tilde u}(r)=O(r^{-A})$ for some $A>2\tilde u+1$;
\vskip0.1cm
(iii) 
$
d_1(r)+h_{\tilde p}(r)+n_{b_2}(r)+\min(\beta_\infty(r),c_r\beta_{u_0}(r))=O(r^{-M})
$
for some $M>0$, where $c_r=r^{2-\frac{\ve A-1}{v_0}}$
 and 
$\ve>0$ satisfies $\ve A>2\tilde u+1$.     
\vskip0.1cm
Here $p_0, u,\tilde u, \tilde p, p_0, b,v, u_0,v_0>1$ satisfy
\begin{equation}\label{para}
\frac1{p_0}+\frac{1}{q_0}=\frac{1}q,\,\,\,\,\,\, \frac{1}{qq_0}=\frac{1}{p}+\frac{1}v,\,\,\,\,\, \frac{1}b=\frac{1}{p}+\frac{1}{u},\,\,\,\,\,\, 
\frac{1}{p}=\frac{1}{\tilde p}+\frac{1}{\tilde u},\,\,\,\,\frac{1}{\tilde u}\geq \frac1{u_0}+\frac1{v_0}.
\end{equation}
\end{assumption}

\subsection{Effective random Ruelle-Perron-Frobenius theorems}

Let $L_\om$ be the operator which maps a function $g$ on $\cE_\om$ to a function $L_\om g$ on $\cE_{\te\om}$ 
given by 
\begin{equation}\label{TO NORM}
  L_\om g=\frac{\cL_\om(gh_\om)}{\la_\om h_{\te\om}}.  
\end{equation}
Then  $(L_\om)^*\mu_{\te\om}=\mu_\om$ and $L_\om \textbf{1}=\textbf{1}$, where we denote by $\textbf{1}$ the function which takes the constant value $1$.
Moreover, $L_\om g$ is the transfer operator of $T_\om$ with respect to $\mu_\om$ and $\mu_{\te\om}$, namely for all bounded measurable functions $g:\cE_\om\to\bbR$ and $f:\cE_{\te\om}\to\bbR$ we have 
\begin{equation}\label{Dual}
 \int_{\cE_\om} g\cdot (f\circ T_\om)d\mu_\om=\int_{\cE_{\te\om}}(L_\om g)fd\mu_{\te\om}.   
\end{equation}
Define $L_\om^n=L_{\te^{n-1}\om}\circ\cdots\circ L_{\te\om}\circ L_{\om}$.
It follows from \eqref{RPF0} that
\begin{equation}\label{RPF int 1}
\|L_{\om}^ng-\mu_{\om}(g)\|_{\infty}\leq \|g\|_\al A(\om)B(\te^n\om)\del^n   
\end{equation}
for some random variables $A$ and $B$.
Note that under certain assumptions one can also invoke Oseledets theorem, which together with \eqref{RPF int 1} yields that
\begin{equation}\label{RPF int 2}
\|L_\om^n-\mu_\om\|_\al\leq K(\om)\del^n
\end{equation}
for some tempered  random variable $K(\om)$. However, in order to use either \eqref{RPF int 1} or \eqref{RPF int 2}  to prove limit theorems more some regularity of $A(\om)$ and $B(\om)$ or $K(\om)$ is needed. On the other hand, the exponential decay rate can be relaxed. 

The following result is a more general form of Theorem \ref{RPF Poly}.
\begin{theorem}\label{OSC1}
Let  Assumptions \ref{Mom Ass}, \ref{Inner aprpox}  and \ref{Poly Ass} be in force, where in Assumption \ref{Inner aprpox} we suppose that $\bbP(B_r)=O(r^{-M})$. 
When  $\beta_\infty(r)\to 0$ we set $a=M$ while when  $\beta_{\tilde u}(r)=O(r^{-A})$ we set 
$
a=\min\left(M,1+\frac{1-\ve A}{\tilde u}\right)
$
where $M,\ve,\tilde u,A$ come from the above assumptions. 
Suppose that $A$  and $M$ are large enough so that $a>\beta d+3$.
\vskip0.2cm
(i)There is a constant $a_0>0$ which can be recovered from the proof such that 
if either  part (i)  from Assumption  \ref{Ass A u}  holds, part (ii) of Assumption  \ref{Ass A u}  holds with  $u$ and $a_0$ or  \eqref{al 1} holds with the above $M$ then there is a unique RPF triplet like in \eqref{RPF0}. 
\vskip0.1cm

(ii) Let $t$ be given by $\frac 1t=\frac 1q+\frac 1d$. There is a random variable $R(\om)\in L^t(\Om,\cF,\bbP)$ such that for $\bbP$-a.a. $\om$,  all $n\in\bbN$ and every H\"older continuous function $g:\cE_\om\to\bbR$,
\begin{equation}\label{RPF2}
\left\|L_\om^n g-\mu_\om(g)\right\|_{\infty}\leq V_\om\|g\|_\al R(\om)n^{-\be}    
\end{equation}
where 
$
V_\om=12\xi_\om^{-1}\left(1+\frac{4\gamma_{\te^{-1}\om}^{\al}}{H_{\te^{-1}\om}}\right).
$

(iii)
Let  $M_\om$ be some measurable upper bound on the minimal number of balls of radius $\xi_\om$ needed to cover $\cE_\om$. 
Then for $\bbP$-a.a. $\om$ and all $n\geq 1$ we have
$$
\left\|L_\om^n-\mu_\om\right\|_{\al}\leq R_1(\te^n\om)(M_{\te^n\om}\xi_{\te^n\om}^{-1})V_\om R_2(\om)n^{-\beta}
$$
for $R_2(\cdot)\in L^{t/2}(\Om,\cF,\bbP)$ and $R_1(\om)$ has the following property: for all $C$ large enough, the first visiting time $n_1(\om)$ to the set $A=\{R_1(\om)\leq C\}$ satisfies 
$
\bbP(n_1(\om)\geq k)=O(k^{-a}),
$
where $a$ is like in Theorem \ref{CLT}. Furthermore, $R_1(\om)=e^{Y(\om)}$, 
where $Y\in L^{b_1}$ and for every $\varepsilon\in(0,1)$ we have $\|Y-\bbE[Y|\cF_r]\|_{L^{\tilde b}}=O(r^{1+\frac{1-\varepsilon A}{\tilde u}}+r^{-M})$.

\end{theorem}

\begin{remark}\label{Change remark}
 If $\om\mapsto M_\om\xi_\om^{-1}\in L^\ell(\Om,\cF,\bbP)$ then $E(\om):=\sup_{n\geq 1}(n^{-\ve}M_{\te^n\om}\xi_{\te^n\om}^{-1})\in L^\ell(\Om,\cF,\bbP)$ for all $\ve>1/\ell$. Indeed, 
$
\bbE[(E(\om))^\ell]\leq \sum_{n\geq 1}n^{-\ve \ell}\bbE[(M_{\te^n\om}\xi_{\te^n\om}^{-1})^\ell]=\|M_\om \xi_\om^{-1}\|_{L^\ell}^\ell\sum_{n\geq 1}n^{-\ve \ell}<\infty.
$
Therefore, in Theorem \ref{OSC1} (iii) we can replace $M_{\te^n\om}\xi_{\te^n\om}^{-1}R_2(\om)$ with $A(\om)=R_2(\om)E(\om)$ and $n^{-\be}$ with $n^{-\be+\ve}.$ Note that if $v$ is given by $\frac 1v=\frac{2}{t}+\frac{1}\ell$ then $A(\om)\in L^v(\Om,\cF,\bbP)$.
\end{remark}
The proof of Theorem \ref{OSC1} appears in Section \ref{OSC 12 SEC}. This section includes many preparations, and the proof is finalized in Section \ref{OSC 12}.
The following result is a direct consequence of Theorem \ref{OSC1} (ii) and the fact that $\mu_\om(f\cdot (g\circ T_\om^n))=\mu_{\te^n\om}((L_\om^n f)g)$ (which follows from iterating \eqref{Dual}).
\begin{corollary}[Effective polynomial decay of correlations]\label{DEC}
In the circumstances of Theorem \ref{OSC1} (ii), 
$\bbP$-a.s. for every $n\in\bbN$, $f:\cE_\om\to\bbR$ and $g:\cE_{\te^n\om}\to\bbR$ with $\mu_\om(f)=0$ we have 
$$
\left|\mu_\om(f\cdot(g\circ T_\om^n))\right|\leq \|g\|_{L^1(\mu_{\te^n\om})}\|f\|_\al R(\om)n^{-\be}.
$$
\end{corollary}


\subsection{The CLT}
Let $f:\cE\to \bbR$ be a measurable function  such that $f_\om$ is H\"older continuous with exponent $\al$ and $\mu_\om(f_\om)=0$ for $\bbP$ a.e. $\om$. Let $S_n^\om f=\sum_{j=0}^{n-1}f_{\te^j\om}\circ T_\om^j$,  $n\in\bbN$.

\begin{theorem}[CLT]\label{CLT0}
Let the conditions of Theorem \ref{OSC1} be in force with $\beta>1$.
\vskip0.2cm
(I) There is a constant $a_0>0$ which can be recovered from the proof such that 
if either  part (i)  from Assumption  \ref{Ass A u}  holds, part (ii) of Assumption  \ref{Ass A u}  holds with  $u$ and $a_0$ or  \eqref{al 1} holds with the above $M$ then there is a unique RPF triplet like in \eqref{RPF0}. 
\vskip0.1cm

(II) Let $t>1$ be defined by $\frac 1t=\frac 1q+\frac 1d$ and 
let $s>1$ be defined by $\frac{1}s+\frac 1t=\frac 12$.  Suppose
 $\|f_\om\|_{\al}\in L^{s}$. Then:

\vskip0.1cm
(1) There is a number $\sig\geq0$ such that for $\bbP$-a.a. $\om$ we have 
\begin{eqnarray*}
&
\lim_{n\to\infty}\frac1n\text{Var}_{\mu_\om}(S_n^\om f)=\sig^2.
\end{eqnarray*}
Furthermore, $\sig=0$ iff $\bbP$-a.a. $\om$ we have $ f(\om,x)=q(\om,x)-q(\te\om, T_\om x)$, $\mu_\om$-a.s.  for some $q\in L^2(\mu)$, where  $\mu=\int_{\Om} \mu_\om d\bbP(\om)$.
(2) For $\bbP$-a.a. $\om$ the sequence   $(S_n^\om f)$ obeys the CLT: if $\sigma>0$ then for every real $z$ we have
\begin{eqnarray*}
&
\lim_{n\to\infty}\mu_\om\{x: n^{-1/2}S_n^\om f(x)\leq z\}=\frac{1}{\sqrt{2\pi}\sig}\int_{-\infty}^{z}e^{-\frac{y^2}{2\sig^2}}dy.
\end{eqnarray*}
\vskip0.01cm
(3) Set $\tau(\om,x)=(\te\om, T_\om x)$.
Suppose $\sig>0$ and set
 $\zeta(t)=\left(2t\log\log t\right)^{1/2}$ and 
\begin{eqnarray*}
&\eta_n(t)=\left(\zeta(\sig^2 n)\right)^{-1}\left(\sum_{j=0}^{k-1}f\circ\tau^j+(nt-k)f\circ\tau^k\right)\,\,\text{ if }\,\,t\in\left[\frac{k}{n}, \frac{k+1}n\right),\,\, 0\leq k<n.    
\end{eqnarray*}
Then the functional Law of iterated logarithm holds: for $\mu$-a.s. the sequence of functions $\{\eta_n(\cdot), n\geq 3/\sig^2\}$ is relatively compact in $C[0,1]$, 
and the set of limit points as $n\to\infty$ coincides with the set of absolutely continuous functions $x\in C[0,1]$ so that $\int_{0}^1(\dot{x}(t))^2dt\leq 1$. 
\end{theorem}
Relying on Theorem \ref{RPF Poly}, the proof of Theorem \ref{CLT} follows from \cite[Theorem 2.3]{Kifer 1998} similarly to \cite[Theorem 32]{Adv}.

\subsection{Proof of Theorem \ref{RPF Poly} for random subshifts of finite type}\label{SFT}
For subshifts  \eqref{Pair1.0} and \eqref{Pair2.0} hold with $\xi_\om=e^{-1}$ and $\gamma_\om=e$. Moreover, \eqref{Cover} holds for every $n\geq m(\om)$ with $m(\om)$ from the definition of the random subshift. Since the degree of $T_\om$ is $d_\om$ we can take $D_\om=d_\om$. Moreover, we can take $N(\om)=e$ since $T_\om$ expands distances by at most $e$. Furthermore,  Assumption \ref{Inner aprpox} in this case is a weaker form of our assumption that $\mathfrak d(L_M,\cF_r)=O(r^{-p})$. Thus, Theorem \ref{RPF Poly} for random subshifts follows from Theorem \ref{OSC1} by choosing appropriate parameters.  
\qed

\subsection{Proof of Theorem \ref{RPF Poly} for $C^2$ expanding maps}\label{Smooth Sec}
In what follows we will explain how to verify all the assumptions of Theorem \ref{OSC1} for $C^2$ expanding maps in Section \ref{Sec 2}.
After this is done Theorem \ref{RPF Poly} for $C^2$  maps will follow from Theorem \ref{OSC1} by choosing appropriate parameters.

Let $\beta(\om)=\|(DT_\om)^{-1}\|$, $\gamma_\om=(\beta(\om))^{-1}$ and denote $\mathfrak a_p(r)=\|\beta-\mathbb E[\beta|\cF_r]\|_{L^{p}(\Omega,\cF,\bbP)}$. 
Applying  Lemmata \eqref{alpha exp lemm} and \ref{g exp lemm} with $g=\beta^u$ we get the following result.
\begin{lemma}\label{beta mom}
Let $u\geq 1$. Suppose that either $\lim_{n\to\infty}\psi_U(n)=0$ or $\al(n)+\mathfrak a_1(n)\leq Cn^{-A}$ for some $C>0$ and $A>u+1$. Then for every $\ve\in(0,1)$ there is $C_{\ve,u}>0$ such that for all $n\in\bbN$, 
$$
\left\|g_{\om,n}\right\|_{L^u(\Om,\cF,\bbP)}\leq  C_{\ve,u} n^{(1-\ve A)/u}.
$$    
\end{lemma}
Next, set 
\begin{eqnarray}\label{Z def}
&Z_\om=\sum_{j=1}^\infty\prod_{i=1}^{j}\gamma_{\te^{-i}\om}^{-1}\,\,\,\text{ and }\,\,\,\,\xi_\om=C\min(1,Z_\om^{-1})
\end{eqnarray}
where $C=\frac12\min(1,\rho_M)$ and $\rho_M$ is the injectivity radius of $M$. 
The following result is an immediate consequence of the previous lemma.
\begin{corollary}\label{Cor Smooth}
 In the circumstances of the above lemma we have 
 $Z_\om\in L^u(\Om,\cF,\bbP)$. Moreover
$$
\|Z_\om-\bbE[Z_\om|\cF_{-2r,2r}]\|_{L^{u/2}}=O(a_u(r)+r^{1-(1-\ve A)/u}).
$$
\end{corollary}

Next, it was shown in \cite[Section 4]{Kifer Thermo} that \eqref{Pair1.0}
 and \eqref{Pair2.0} hold with $\gamma_\om$ as above and $\xi_\om$.  Moreover, it was shown that the degree of $T_\om$ does not exceed $C_0\left(\|DT_\om\|Z_\om\right)^{\text{dim }M}$ for some constant $C_0>0$. Thus we can take 
 $$
D_\om=C_0\left(\|DT_\om\|Z_\om\right)^{\text{dim }M}.
 $$
Furthermore, it was shown that we can cover $M$ by 
$
M_\om:=C_1(Z_\om)^{\text{dim M}}
$
balls with radius $\xi_\om$, where $C_1$ is some constant.
Additionally, as an upper bound for all H\"older exponents we can take 
$
N(\om)=\|DT_\om\|_\infty
$
or any other measurable upper bound of $\|DT_\om\|_\infty$.
Let $\phi_\om$ be a random $\al$-H\"older continuous function such that $\int \ln(1+v_\al(\phi_\om))d\bbP(\om)<\infty$. By Assumption \ref{Mom A} we have $\ln(\gamma_\om)\in L^1(\Om,\cF,\bbP)$, $\ln(1+\|D T_\om\|)\in L^1(\Om,\cF,\bbP)$ and $\bbE[\ln \gamma_\om]<0$. Thus
by \cite[Theorem 2.2]{Kifer Thermo} there is a triplet $\la_\om,h_\om$ and $\nu_\om$ as described in Section \ref{RPF Prem}.

Next, since $\bbP$-a.a. $\om$ we have $g_{\om,n}:=\prod_{j=0}^{n-1}\gamma_{\te^j\om}\to \infty$ as $n\to\infty$,  
 by \cite[Proposition 8.19]{MSU} we see that there is a constant $R>0$ such that \eqref{Cover} holds with 
\begin{equation}\label{m def smooth case}
m(\om)=\min\left\{n: \xi_\om^{-1} g_{\om,n}\leq R\right\}.    
\end{equation}
By Corollary \ref{Cor Smooth} and the above formulas for $D_\om,\xi_\om$ and $M_\om$,
what is left to do  is to show that $m(\cdot)$ has the properties described in Assumptions \ref{Inner aprpox} and  \ref{Poly Ass}.
 The rest of the assumptions in Theorem \ref{OSC1}, are just regularity conditions imposed on  the random variables $v_\al(\phi_\om), \|\phi_\om\|_\infty, \ln D_\om$ and $N(\om)$ defined above (or on some  upper bounds).

Let us now  estimate the tails of  $m(\cdot)$, as required in Assumption \ref{Poly Ass}.
\begin{lemma}
In the circumstances of Lemma \ref{beta mom} and when $u\geq 2$ then $\forall \varepsilon>0$,
$$
\bbP(m(\om)>k)=O(k^{(1-\varepsilon A)/u}).
$$
\end{lemma}
\begin{proof}
Denote $g_{\om,k}=\prod_{j=0}^{k-1}\gamma_{\sigma^{j}\om}$.
 By the definition of $m(\om)$ and the Markov inequality,
$$
\bbP(m(\om)>k)\leq \bbP\left(\xi_\om^{-1}g_{\om,k}\geq R\right)\leq R^{-1}\bbE\left[\xi_\om^{-1}g_{\om,k}\right]
\leq R^{-1}C^{-1}\bbE\left[(1+Z_\om) g_{\om,k}\right]
$$ 
where we have used that $\xi_\om^{-1}=C^{-1}\max(1,Z_\om)\leq C^{-1}(1+Z_\om)$.
Thus by the H\"older inequality and Lemma \ref{beta mom},
$
\bbP(m(\om)>k)\leq R^{-1}C^{-1}(1+\|Z_\om\|_{L^p})k^{(1-\varepsilon A)/u}.
$
Finally, note that $\|Z_\om\|_{L^p}$ in view of Corollary \ref{Cor Smooth} and since $p\leq u$.
\end{proof}
Finally, we discuss conditions under which Assumption \ref{Inner aprpox} is in force.
\begin{lemma}
Let $u\geq 1$ be such that $Z_\om\in L^q(\Om,\cF,\bbP)$. Let $p$ be the conjugate exponent of $q$. 
If
$
b_r=\|g_{\om}-\bbE[g_\om|\cF_r]\|_{L^p}+\|Z_\om-\bbE[Z_\om|\cF_r]\|_{L^q}\to 0\,\,\text{ as }r\to\infty
$
then for every $M$ and $r\in\bbN$ there are sets $A_r\in\cF_r$ and $B_r\in\cF$ such that  $\bbP(B_r)=O(\sqrt{b_r})$ and
$$
A_r\subset \{\om: m(\om)\leq M\}\cup B_r,\,\,\,\lim_{r\to\infty}\bbP(A_r)=\bbP(m(\om)\leq M).
$$
\end{lemma}
\begin{proof}
Set $g_{\om,n}=\prod_{i=0}^{n-1}g_{\te^i\om}=\prod_{i=0}^{n-1}\gamma_{\te^i\om}^{-1}$. Then by the definition \eqref{m def smooth case} of $m(\om)$ we have
$$
L_M:=\{\om: m(\om)\leq M\}=\left\{\om: \xi_\om^{-1}\min_{j<M}g_{\om,n}\leq R\right\}.
$$
Let 
$
a(r)=\left\|\xi_\om^{-1}\min_{j<M}g_{\om,j}-\bbE[\xi_\om^{-1}\min_{j<M}g_{\om,j}|\cF_r]\right\|_{L^1}.
$
Then, 
$$
A_r:=\left\{\om: \bbE[\xi_\om^{-1}\min_{j<M}g_{\om,j}|\cF_r]\leq R-\sqrt{a(r)}\right\}\subset L_M\cup B_r
$$
where 
$
B_r=\left\{\left|\xi_\om^{-1}\min_{j<M}g_{\om,j}-\bbE[\xi_\om^{-1}\min_{j<M}g_{\om,j}|\cF_r]\right|>\sqrt{a(r)}\right\}. 
$
Now, by the Markov inequality,  
$
\bbP(B_r)\leq \sqrt{a(r)}.
$
Thus, all that is left to do is to show that $a(r)\to 0$ as $r\to\infty$ in an appropriate rate.

First, since $\gamma_\om\geq 1$ we have $\min_{j<M}g_{\om,j}=g_{\om,M}$. Now, using that 
$
|\prod_{i=0}^{M-1}a_i-\prod_{i=0}^{M-1}b_i|\leq\sum_{i=0}^{M-1}|a_i-b_i|
$ 
for all numbers $a_i,b_i\in[0,1]$, we see that
\begin{eqnarray*}
&\left\|g_{\om,M}-\prod_{j=0}^{M-1}\bbE[g_{\te^j\om}|\cF_{j-r,j+r}]\right\|_{L^p}\leq M\|g_{\om}-\bbE[g_\om|\cF_r]\|_{L^p}.    
\end{eqnarray*}
Thus
$
a(r)\leq MC^{-1}\|g_{\om}-\bbE[g_\om|\cF_r]\|_{L^p}(1+\|Z_\om\|_{L^q})+C^{-1}\|Z_\om-\bbE[Z_\om|\cF_r]\|_{L^q}
$
where we have used that $\xi_\om^{-1}=C^{-1}\max(1,Z_\om)$.
\end{proof}

\section{Preparations for the proof of  Theorem \ref{OSC1}: effective cones contractions}\label{OSC 12 SEC}

\subsection{The reversed covering time $j_\om$}
Let  $j_\om=\min\{n\in\bbN: m(\te^{-n}\om)\leq n\}$. 
\begin{lemma}
The random variable $j_\om$ is finite $\bbP$-a.s.    
\end{lemma}
\begin{proof}
Since $\te$ preserves $\bbP$ for all $n$ we have
$$
\bbP(m(\te^{-n}\om)\leq n)=\bbP(m(\om)\leq n)\to 1\text{ as }n\to\infty.
$$
Thus, the union of the sets $U_n=\{\om:m(\te^{-n}\om)\leq n \}, n\geq 1$ has probability $1$. Therefore, for $\bbP$-a.a. $\om$ there exists an $n\in\bbN$ such that $\om\in U_n$.
\end{proof}

Next, notice that by the definition of $j_\om$ we have $m(\te^{-j_\om}\om)\leq j_\om$. Thus by  \eqref{Cover} and  since all the maps are $T_\om$ are surjective, $\bbP$-a.s. for all $x\in \cE_{\te^{-j_\om}\om}$,
\begin{equation}\label{j cover}
T_{\te^{-j_\om}\om}^{j_\om}(B_{\te^{-j_\om}\om}(x,\xi_{\te^{-j_\om}\om}))=\cE_\om    
\end{equation}

In the course of the proofs of the main results we will need the following two lemmata.
\begin{lemma}\label{tails lemma}
For every integer $k\geq 0$ we have 
$
\bbP(j_\om>k )\leq \bbP(m(\om)>k).
$
\end{lemma}
\begin{proof}
$
\bbP(j_\om> k)=\bbP(m(\te^{-n}\om)>n,\, \forall n\leq k)\leq \bbP(m(\te^{-k}\om)>k)=\bbP(m(\om)>k).
$
\end{proof}

\begin{lemma}\label{Inner approx j}
Under Assumption \ref{Inner aprpox}  
for every $J\in\bbN$ large enough we have the following.  
For every $r\in\bbN$ there is a set $\tilde A_{r}$ measurable with respect to  $\cF_r$ such that with $R_J=\{\om:j_\om\leq J\}$,
$$
\tilde A_r\subset R_J\cup \tilde B_r\,\,\text{ and }\,\,\lim_{r\to\infty}\bbP(\tilde A_r)=\bbP(R_{J})
$$
for some measurable set $\tilde B_r$ such that either $\bbP(\tilde B_r)=O(r^{-M})$.
\end{lemma}
This lemma means that Assumption \ref{Inner aprpox} is in force  for $j_\om$ if it holds for $m(\om)$.
\begin{proof}
Using the definition of $j_\om$ we see that 
 $
R_{J}=\{\om: j_\om\leq J\}=\bigcup_{k=1}^{J}\{\om: m(\te^{-k}\om)\leq k\}=\bigcup_{k=1}^{J}\bbI_{\te^{k}L_k}
 $
 where $L_k=\{\om: m(\om)\leq k\}$. Since Assumption \ref{Inner aprpox} is in force, for each $1\leq k\leq J$ and all $r\in\bbN$ there are sets $A_{k,r}\in\cF_r$ and $B_{k,r}\in\cF$ such that $A_{k,r}\subset L_k\cup B_{k,r}$,
 $\lim_{r\to\infty}\bbP(A_{k,r})=\bbP(L_k)$ and
 either $\bbP(B_{k,r})=O(r^{-M})$ for every $1\leq k\leq J$ or $\bbP(B_{k,r})=O(e^{-br^a})$ for every $1\leq k\leq J$. Next, set
 $
\bar A_r=\bigcup_{k=1}^{J}\te^{k}A_{k,r}\,\,\,\text{ and }\,\,\,
\bar B_r=\bigcup_{k=1}^{J}\te^kB_{k,r}.
 $
 Note that $\te^kA_{k,r}\in\cF_{-r-k,r+k}\subset\cF_{-r-J,r+J}$ and so, $\bar A_{r}\in \cF_{-r-J, r+J}$. Thus for $r>J$ we can take $\tilde A_{r}=\bar A_{r-J}$, while when $r\leq J$ we take  $\tilde A_r=\emptyset$.
\end{proof}

\subsection{Random equivariant cones and their basic properties}
Define 
\begin{eqnarray*}
  &  Q_\om=Q_\om(H)=\sum_{j=1}^{\infty}H_{\te^{-j}\om}\prod_{k=1}^j\gamma_{\te^{-k}\om}^{-\al}.
\end{eqnarray*}
Then $Q_\om$ is finite $\bbP$-a.s. since $\ln H_\om\in L^1$ (because of Assumption \ref{Mom Ass}).
Next,  let us fix some $s>2$ and consider the cone of real valued functions on $\cE_\om$ given by 
$$
\cC_\om=\cC_{\om,s}=\{g:\cE_\om\to [0,\infty): g(x)\leq g(x')e^{sQ_\om d^\al(x,x')}\text{ if } d(x,x')\leq \xi_\om\}.
$$
Notice that a positive function $g:\cE_\om\to\bbR$ belongs to $\cC_{\om,s}$ if and only if $v_{\al,\xi_\om}(\ln g)\leq sQ_\om$. 

Next, combining Lemma \ref{Norm comp} and \cite[Lemma 5.5.4]{HK} we get the following result.

\begin{lemma}\label{regeneration}
 For every function $g:\cE_\om\to\bbR$ such that $\|g\|_{\al}<\infty$ there are functions $g_1,g_2, g_3, g_4\in \cC_{\om,s}\cup (-\cC_{\om,s})$ such that 
 $g=g_1+g_2+g_3+g_4$ and 
 $$
\|g_2\|_{\al}+\|g_2\|_{\al}+\|g_3\|_{\al}+\|g_4\|_{\al}\leq U_\om\|g\|_\al
 $$
 where $U_\om=12\xi_\om^{-1}(1+\frac{4}{sQ_\om})$.
\end{lemma}

We will also need the following result which appeared as  \cite[Lemma 5.7.3]{HK}.
\begin{lemma}\label{Lemma 5.7.3}
For $\bbP$-a.a. $\om$
we have $\cL_\om\cC_{\om,s}\subset \cC_{\te\om,s_{\te\om}'}\subset\cC_{\te\om,s}$
where 
$s'_{\te\om}=\frac{sQ_\om+H_\om}{Q_\om+H_\om}<s$.
\end{lemma}

\begin{corollary}
For $\bbP$-a.a. $\om $ we have
 $h_\om\in\cC_\om$.  
\end{corollary}
\begin{proof}
By \eqref{RPF0}, possibly along a subsequence, we have 
 \begin{equation}\label{h lim}
 h_\om=\lim_{n\to\infty}\frac{\cL_{\te^{-n}\om}^n\textbf{1}}{\la_{\te^{-n}\om,n}}    
 \end{equation}
 uniformly over $\cE_\om$.
 Now the lemma follows since the cone $\cC_{\om,s}$ contains all constant functions.
\end{proof}






\subsection{Some useful estimates}
\begin{lemma}
For $\bbP$-a.a. $\om$ we have     
\begin{equation}\label{Lam Bound}
\big(D_\om e^{\|\phi_\om\|_\infty}\big)^{-1} \leq \la_\om\leq D_\om e^{\|\phi_\om\|_\infty}.
\end{equation}
\end{lemma}

\begin{proof}
Note that $\la_\om=\nu_{\te\om}(\cL_\om \textbf{1})$ and that $\cL_\om \textbf{1}\geq e^{-\|\phi_\om\|_\infty}$. Therefore,
\begin{equation}\label{Lam Bound0}
e^{-\|\phi_\om\|_\infty} \leq \la_\om\leq \|\cL_\om \textbf{1}\|_\infty.
\end{equation}
Now \eqref{Lam Bound} follows
using that $\|\cL_\om \textbf{1}\|_\infty\leq e^{\|\phi_\om\|_\infty}D_\om$ and that $D_\om\geq 1$.
\end{proof}

\begin{lemma}\label{h Bound lemm}
If $\xi_\om<1$ then 
\begin{equation}\label{h est 1 0}
\big(B_{j_\om}(\om)\big)^{-1}\leq h_\om \leq B_{j_\om}(\om)   
\end{equation}
where 
$
B_j(\om)=e^{Q_{\te^{-j}\om}}D_{\te^{-j}\om,j}e^{\sum_{k=-j}^{-1}\|\phi_{\te^k\om}\|_\infty}
$
and $D_{\om,k}=\prod_{j=0}^{k-1}D_{\te^j\om}$ for every $\om\in\Om$ and $k\in\bbN$.
\end{lemma}

\begin{proof}
 Aguing like in the proof of \cite[Lemma 3.8]{MSU} for all $n\geq j_\om$ we have 
\begin{equation}
\big(B_{j_\om}(\om)\big)^{-1}\leq \frac{\cL_{\te^{-n}\om}^n\textbf{1}}{\nu_\om(\cL_{\te^{-n}\om}^n\textbf{1})} \leq B_{j_\om}(\om)   
\end{equation}
and \eqref{h est 1 0}  follows using \eqref{h lim} since the above inequalities are preserved after taking the limit, and taking into account that 
$
\la_{\te^{-n}\om,n}=\nu_\om(\cL_{\te^{-n}\om}^n \textbf{1}).
$
\end{proof}


When $\xi_\om=1$ we get simpler estimates.
\begin{lemma}\label{Lem 5.9}
If $\xi_\om=1$ for $\bbP$-a.a. $\om$ then 
\begin{equation}\label{h est 2}
e^{-sQ_\om}\leq  h_\om\leq e^{sQ_\om}.   
\end{equation}
\end{lemma}
\begin{proof}
Using that $\xi_\om=1$, $h_\om\in\cC_{\om,s}$ and $\min_{x\in\cE_\om} h_\om(x)\leq \nu_\om(h_\om)=1\leq \|h_\om\|_\infty$ we have
$
1\leq \|h_\om\|_\infty\leq (\min_{x\in\cE_\om}h_\om(x))e^{sQ_\om}\leq e^{sQ_\om}.
$ 
\end{proof}

\subsection{The normalized potential}\label{Sec norm pot}
Let 
$$
\tilde \phi_\om=\phi_\om+\ln h_\om-\ln (h_{\te\om}\circ T_\om)+\ln \la_\om.
$$
Then $L_\om$ in \eqref{TO NORM} is the transfer operator generated by the map $T_\om$ and the random potential $\tilde\phi_\om$. Because our goal is to apply cone contraction arguments with $L_\om$ and not $\cL_\om$, 
our proofs will not take advantage of the cone $\cC_{\om,s}$. Instead we will need to work with a similar cone $\tilde\cC_{\om,s}$ corresponding to the normalized potential $\tilde \phi_\om$. In order to get effective estimates we need to obtain some explicit upper bounds related to the new potential $\tilde\phi_\om$. 
\begin{lemma}\label{v tilde phi}
For $\bbP$-a.a. $\om$ we have
$$
v_{\al,\xi_\om}(\tilde \phi_\om)\leq \tilde H_\om:=H_\om+sQ_\om+sQ_{\te\om}N(\om).
$$
\end{lemma}
\begin{proof}
The lemma follows from the definitions of $\tilde\phi_\om$ and $N(\om)$ and the fact that $h_\om\in\cC_{\om,s}$ which means that $v_{\om,\xi_\om}(\ln h_\om)\leq sQ_\om$.    
\end{proof}

Next, notice that for all $n\in\bbN$,
$
 S_n^\om\tilde \phi=S_n^\om\phi+\ln h_\om-\ln(h_{\te^n\om}\circ T_\om^n
 )+\sum_{j=0}^{n-1}\ln \la_{\te^j\om}.
$
Using  \eqref{Lam Bound}, Lemma \ref{h Bound lemm} (or Lemma \ref{Lem 5.9}) and the inequality $\|\phi_\om\|_\infty\leq H_\om$ we get the following result.
\begin{corollary}\label{cor h bound}

(i) $\bbP$-a.s. for every $n\in\bbN$ we have 
$
\|S_n^\om\tilde \phi\|_\infty\leq R_n(\om)
$
where
\begin{eqnarray*}
&R_n(\om)=3\sum_{j=-J_n(\om)}^{n}(H_{\te^j\om}+\ln D_{\te^j\om})
+s(Q_{\te^{-j_\om}\om}+Q_{\te^{(n-j_{_{\te^n\om}})}\om}),\,\, \,J_n(\om)=\max(j_\om,j_{\te^n\om}).
\end{eqnarray*}
In particular, if $\xi_\om<1$ then for every $\om$ and $M,J\in\bbN$ such that
$J\geq \max(j_\om, j_{\te^M\om})$ we have 
\begin{eqnarray*}
&\|S_M^\om\tilde\phi\|_\infty\leq R_{J,M}:=
3\sum_{k=-J}^{M}(H_{\te^k\om}+\ln D_{\te^k\om})+
2s\max_{-J\leq \ell\leq M}Q_{\te^{\ell}\om}.
\end{eqnarray*}

(ii)  If $\xi_\om=1$ then for all $M\in\bbN$ we have 
\begin{eqnarray*}
&
\|S_M^\om\tilde\phi\|_\infty\leq  2\sum_{k=0}^{M-1}(H_{\te^k\om}+\ln D_{\te^k\om})+s(Q_\om+Q_{\te^M\om})\leq R_{0,M}(\om).
\end{eqnarray*}
\end{corollary}

\subsection{Approximation estimates}\label{Sec approx est}
In this section we will show that all the random variables that will appear in the expressions bounding the cone-contraction rates of the iterates $L_\om^n$ can be well approximated by random variables measurable with respect to $\cF_r$, as $r\to\infty$. For every $p>0$ and $r\in\bbN$ let $n_p(r), q_p(r)$ and $h_p(r)$ be like in \eqref{approx coef}. Define also
$
\tilde h_p(r)=\|\tilde H_{\om}-\bbE[\tilde H_{\om}|\cF_r]\|_{L^p(\Om,\cF,\bbP)}.
$
Note that  all the above quantities are decreasing in $r$ and increasing in $p$.

The following result follows from the definition of $\tilde H_\om$ and the H\"older inequality. 
\begin{lemma}\label{tilde H approx}
Let $\tilde p,b,b_2\geq 1$ be such that $\frac{1}{\tilde p}=\frac{1}{b}+\frac1{b_2}$. 
Suppose that
$Q_\om\in L^{b}$
and $N(\om)\in L^{b_2}$.  Then $H_\om\in L^{\tilde p}$ and 
\begin{equation}\label{tilde approx}
\tilde h_{\tilde p}(r)\leq h_{\tilde p}(r)+sq_{\tilde p}(r)+s\|Q_\om\|_{L^{b}}n_{b_2}(r)+sq_{b}(r-1)\|N(\om)\|_{L^{b_2}}.
\end{equation}
\end{lemma}

To estimate the term $q_{b}(r)$ and to show that $Q_\om\in L^{b}$
we will use the following result.
\begin{lemma}\label{L1}
Let $1<p<\infty$ be such that $H_\om\in L^p$ and  $u$ be any number greater or equal than the conjugate exponent of $p$.
\vskip0.2cm
(i) When $\gamma_\om\geq 1$ we have the following.
 Suppose that  
\begin{equation}\label{upper psi cond}
 \limsup_{k\to\infty}\psi_U(k)<(\bbE[\gamma_\om^{-\al u}])^{-1}-1.   
\end{equation}
For all $k\geq1$ denote 
$
\beta_{k}(r)=\left\|\gamma_\om^{-\al }-\bbE[\gamma_\om^{-\al }|\cF_r]\right\|_{L^k(\Om,\cF,\bbP)}.
$
Assume also that one of the following conditions hold:
\vskip0.1cm
(1)  $\beta_{\infty}(r)\to 0$ as $r\to \infty$. 
\vskip0.1cm
(2)  $\beta_u(r)=O(r^{-A})$ for some $A>2u+1$.
\vskip0.1cm
Let $b>1$ be given by  $1/b=1/p+1/u$ and let $u_0,v_0>u$ and $v_0\in(1,\infty]$ be such that $\frac 1u\geq \frac 1{u_0}+\frac{1}{v_0}$.
Then  $Q_\om \in L^b$ 
and for some constant $C_0>0$,
\begin{equation}\label{Q cond approx1}
q_b(2r)\leq C_0\left(\zeta_r+h_p(r)+\min\left(\beta_\infty(r), c_r\beta_{u_0}(r)\right)\right)
\end{equation}
where under (1) we have $\zeta_n=e^{-an/u}$ for some $a>0$ and under (2) we have $\zeta_n=n^{1+(1-\ve A)/u}$ for  $\ve\in(0,1)$ such that $2+(1-\ve A)/u<0$.  Moreover, when $u_0=u$ (i.e. $v_0=\infty$) then $c_r=r^2$, while when $v_0<\infty$ then $c_r=c$ is a constant under (1), while under (2) we have $c_r=r^{2+\frac{1-\ve A}{v_0}}$. 
\end{lemma}

\begin{proof}
 Let us first show that in both situations described in part (i) of the lemma we have
\begin{eqnarray}\label{Q cond}
 & \left\|Q_\om-\sum_{j=1}^{n}H_{\te^{-j}\om}\prod_{k=1}^j\gamma_{\te^{-k}\om}^{-\al}\right\|_{L^b}\leq C\|H\|_{L^p}\sum_{k>n}\del_k 
\end{eqnarray}
where under (1) we have $\del_k=e^{-ak/u}$ for some $a>0$ and under (2) we have $\del_k=O(k^{(1-\ve A)/u})$ for an arbitrary $\ve\in(0,1)$.
Note that once \eqref{Q cond} is proven by taking $n=1$ we get that $Q_\om\in L^b$.
In order to prove \eqref{Q cond}, for all $j$ by the H\"older inequality  we have 
\begin{eqnarray}\label{RightAfter}
& \bbE\left[H_{\te^{-j}\om}\prod_{k=1}^j\gamma_{\te^{-k}\om}^{-\al}\right]\leq \|H\|_{L^p}\left\|\prod_{k=1}^j\gamma_{\te^{-k}\om}^{-\al}\right\|_{L^u}.   \end{eqnarray}
Next, recall that for every $d\geq1$, random variable $X\in L^d$ and a sub-$\sig$-algebra $\cG$, \, $\|X-\bbE[X|\cG]\|_{L^d}$ minimizes $\|X-Y\|_{L^d}$ for $Y\in L^d(\cG)$. Thus, with $g(\om)=\gamma_{\om}^{-\al u}$ for all $q\geq1$ we have,
\begin{equation}\label{comp approx}
 \|g(\om)-\bbE[g(\om)|\cF_r]\|_{L^q(\Om,\cF,\bbP)}\leq \|g(\om)-(\bbE[\gamma_{\om}^{-\al}|\cF_r])^u\|_{L^q(\Om,\cF,\bbP)}\leq C_{u,p}\beta_{q}(r)  
\end{equation}
for some constant $C_{u,p}$ which depends only on $u$ and $p$,
where we have taken into account that $\gamma_{\om}\geq1$.
Now the proof of \eqref{Q cond} in the circumstances of the first part follows from Lemma \ref{g exp lemm} applied with the above $g(\om)$, which yields that 
\begin{eqnarray}\label{gamm prod exp}
&\left\|\prod_{k=1}^j\gamma_{\te^{-k}\om}^{-\al}\right\|_{L^u}^u\leq C\del_j^u.
\end{eqnarray}
Next, write
\begin{eqnarray*}
& \sum_{j=1}^r H_{\te^{-j}\om}\prod_{k=1}^{j}\gamma_{\te^{-k}\om}^{-\al}=\sum_{j=1}^r\bbE[H_{\te^{-j}\om}|\cF_{-j-r,-j+r}]\prod_{k=1}^{j}\bbE[\gamma_{\te^{-k}\om}^{-\al}|\cF_{-k-r,-k+r}]
\\&+   
\sum_{j=1}^r(H_{\te^{-j}\om}-\bbE[H_{\te^{-j}\om}|\cF_{-j-r,-j+r}])\prod_{k=1}^{j}\gamma_{\te^{-k}\om}^{-\al}
\\
&+\sum_{j=1}^r\bbE[H_{\te^{-j}\om}|\cF_{-j-r,-j+r}]\left(\prod_{k=1}^{j}\gamma_{\te^{-k}\om}^{-\al}-\prod_{k=1}^j\bbE[\gamma_{\te^{-k}\om}^{-\al}|\cF_{-k-r,-k+r}]\right).
\end{eqnarray*}

Let us denote $\Del_{\om,k,r}=\gamma_{\te^{-k}\om}^{-\al}-\bbE[\gamma_{\te^{-k}\om}^{-\al}|\cF_{-k-r, -k+r}]$. Then using the triangle and the H\"older inequalities together with \eqref{gamm prod exp} we see that
\begin{eqnarray*}
&\left\|\sum_{j=1}^r H_{\te^{-j}\om}\prod_{k=1}^{j}\gamma_{\te^{-k}\om}^{-\al}-\sum_{j=1}^r\bbE[H_{\te^{-j}\om}|\cF_{-j-r,-j+r}]\prod_{k=1}^{j}\bbE[\gamma_{\te^{-k}\om}^{-\al}|\cF_{-k-r,-k+r}]\right\|_{L^b}\\
&
\leq Ch_p(r)\sum_{j=1}^r\del_j+\|H_\om\|_{L^p}\sum_{j=1}^r\left\|\prod_{k=1}^j\gamma_{\te^{-k}\om}^{-\al}-\prod_{k=1}^j\bbE[\gamma_{\te^{-k}\om}^{-\al}|\cF_{-k-r,-k+r}])\right\|_{L^u}
\\
&\leq Ch_p(r)\sum_{j=1}^r \del_j+\|H_\om\|_{L^p}\sum_{j=1}^r\sum_{s=1}^j\left\|\left(\prod_{k=1}^{s-1}\gamma_{\te^{-k}\om}^{-\al}\right)\Del_{\om,s,r}
\left(\prod_{k=s+1}^{j}\bbE[\gamma_{\te^{-k}\om}^{-\al}|\cF_{-k-r,k-+r}]\right)\right\|_{L^u}.
\end{eqnarray*}
Now, arguing like in the proof of Lemma \ref{g exp lemm} we get that
\begin{eqnarray}\label{Arguing}
&\left\|\left(\prod_{k=1}^{s-1}\gamma_{\te^{-k}\om}^{-\al}\right)\left(\prod_{k=s+1}^{j}\bbE[\gamma_{\te^{-k}\om}^{-\al}|\cF_{-k-r,k-+r}\right)\right\|_{L^u}\leq C\del_{j-1}    
 \end{eqnarray}
where $\del_j$ is like in \eqref{gamm prod exp}. 
We conclude that there is  a constant $C''>0$ such that for all $1\leq s\leq j$,
\begin{eqnarray*}
&
\left\|\left(\prod_{k=1}^{s-1}\gamma_{\te^{-k}\om}^{-\al}\right)\Del_{\om,s,r}\left(\prod_{k=s+1}^{j}\bbE[\gamma_{\te^{-k}\om}^{-\al}|\cF_{-k-r,k-+r}]\right)\right\|_{L^u}\leq C''\del_{j-1}\beta_\infty(r).
\end{eqnarray*}
On the other hand, 
let $u_0,v_0>1$ be such that $\frac{1}{u_0}+\frac{1}{v_0}\leq \frac 1u$ (possibly $v_0=\infty$). 
Since $\gamma_\om\geq 1$, if 
 $v_0>u$ then by \eqref{Arguing} we have
\begin{eqnarray}\label{Arguing1}
&\left\|\left(\prod_{k=1}^{s-1}\gamma_{\te^{-k}\om}^{-\al}\right)\left(\prod_{k=s+1}^{j}\bbE[\gamma_{\te^{-k}\om}^{-\al}|\cF_{-k-r,k-+r}\right)\right\|_{L^{v_0}}\leq (\del_{j-1})^{u/v_0}.    
 \end{eqnarray}
Since $\gamma_\om\geq 1$ and $u_0\geq u$, by the H\"older inequality we see that
\begin{eqnarray*}
&
\left\|\left(\prod_{k=1}^{s-1}\gamma_{\te^{-k}\om}^{-\al}\right)\Del_{\om,s,r}\left(\prod_{k=s+1}^{j}\bbE[\gamma_{\te^{-k}\om}^{-\al}|\cF_{-k-r,k-+r}]\right)\right\|_{L^{u}}\leq (\del_{j-1})^{u/v_0}\beta_{u_0}(r).
\end{eqnarray*}
Combining  the above estimates we conclude that there is a constant $C_0>0$ such that 
\begin{eqnarray}
&\,\,\,\,\,\,\,\,\left\|\sum_{j=1}^r H_{\te^{-j}\om}\prod_{k=1}^{j}\gamma_{\te^{-k}\om}^{-\al}-\sum_{j=1}^r\bbE[H_{\te^{-j}\om}|\cF_{-j-r,-j+r}]\prod_{k=1}^{j}\bbE[\gamma_{\te^{-k}\om}^{-\al}|\cF_{-k-r,-k+r}]\right\|_{L^p}   \label{Approx}
\\
&\leq C_0\left(h_p(r)\sum_{j=1}^r \del_j+\|H_\om\|_{L^p}+\min\left(\beta_\infty(r)\sum_{j=1}^{r}j\del_j, \beta_{u_0}(r)\sum_{j=1}^r j(\del_{j-1})^{u/v_0}\right)\right)\nonumber.
\end{eqnarray}
In the case when $(\gamma_{\te^j\om})_{j=0}^\infty$ are iid it is easy to see that
\begin{eqnarray*}
&
\left\|\sum_{j=1}^r H_{\te^{-j}\om}\prod_{k=1}^{j}\gamma_{\te^{-k}\om}^{-\al}-\sum_{j=1}^r\bbE[H_{\te^{-j}\om}|\cF_r]\prod_{k=1}^{j}\gamma_{\te^{-k}\om}^{-\al}\right\|_{L^b}\leq
h_p(r)\sum_{j=1}^re^{-aj}\leq C_ah_p(r)
\end{eqnarray*}
for some constant $C_a>0$.
Combining the above estimates 
 and taking into account the minimization  properties of conditional expectations discussed right after \eqref{RightAfter} we conclude that 
 \begin{equation}\label{Q cond approx proof}
q_b(2r)\leq C_0\left(\sum_{k=r}^\infty \del_k+h_p(r)\sum_{j=1}^r \del_j+\min\left(\beta_\infty(r)\sum_{j=1}^{r}j\del_j, \beta_{u_0}(r)\sum_{j=1}^r j(\del_{j-1})^{u/v_0}\right)\right).
\end{equation}
 In the iid case the above estimate holds with $\del_j=e^{-aj}$, noting that in this case $\beta_u(r)=\beta_\infty(r)=0$, so the above right hand side  is of order $\sum_{k=r}^\infty \del_k+e^{-ar}\sum_{j=1}^r\del_j+h_p(r)$.

The lemma now follows since in both cases (1) and (2)  the series $\sum_{j=1}^{r}j\del_j$ converges.
Moreover, in case (1)  we have 
$\sum_{k=r}^\infty \del_k=O(e^{-ar/u})$ and in case (2) we have 
$\sum_{k=r}^\infty \del_k=O(r^{1+(1-\ve A)/u})$. Furthermore, when $v_0<\infty$ in case (1)  the series $\sum_{j\geq 1}j(\del_{j-1})^{u/v_0}$ converges, while in case (2) we have $\sum_{j=1}^rj(\del_{j-1})^{u/v_0}=O(r^{2-\frac{1-\ve A}{v_0}})$.
\end{proof}


Next, let us define
\begin{eqnarray*}
&    \tilde Q_\om=Q_\om(\tilde H)=\sum_{j=1}^{\infty}\tilde H_{\te^{-j}\om}\prod_{k=1}^j\gamma_{\te^{-k}\om}^{-\al}.
\end{eqnarray*}
For $p\geq 1$ and $r\in\bbN$ set
$
\tilde q_p(r)=\|\tilde Q_\om-\bbE[\tilde Q_\om|\cF_r]\|_{L^p(\Om,\cF,\bbP)}.
$
Applying Lemma \ref{L1} with $\tilde H_\om$ instead of $H_\om$ and using Lemma \ref{tilde H approx}
we get the following result.
\begin{lemma}\label{tilde Q approx}
Let $p,u,\tilde u,\tilde p, \tilde b, b,b_2\geq 1$ be such that $\tilde u\geq u$,
$\frac{1}{b}=\frac1{p}+\frac1{u}$, $\frac{1}{\tilde p}=\frac1{b}+\frac{1}{b_2}$ and $\frac{1}{\tilde b}=\frac1{\tilde p}+\frac1{\tilde u}$. 
Let also $u_0, v_0>1$ be such that $\frac 1{\tilde u}\geq \frac{1}{u_0}+\frac1{v_0}$.
Let \eqref{upper psi cond} hold and suppose  $H_\om\in L^p$ and $N(\om)\in L^{b_2}$. Then $\tilde Q_\om\in L^{b_1}$ and all the estimates in Lemma \ref{L1} hold  with $\tilde Q_\om$ instead of $Q_\om$, $\tilde u$ instead of $u$, $\tilde b$ instead of $b$, $\tilde h_{\tilde p}(r)$ instead of $h_p(r)$ and $\tilde q_{\tilde b}(r)$ instead of $q_b(r)$.
\end{lemma}

\subsection{Equivariant cones for the normalized transfer operators and projective diameter estimates}\label{Sec cones}
Fix some $s>2$ and define the cone 
$$
\tilde\cC_{\om}=\tilde \cC_{\om,s}=\{g:\cE_\om\to [0,\infty): g(x)\leq g(x')e^{s\tilde Q_\om d^\al(x,x')} \text{ if }d(x,x')\leq \xi_\om\}.
$$
Then by applying \cite[Lemma 5.7.3]{HK} (i.e. Lemma \ref{Lemma 5.7.3}) with the potential $\tilde\phi$ instead of $\phi$ we get 
\begin{equation}\label{Inclusion 1}
 L_\om\tilde \cC_\om\subset \tilde\cC_{\te\om,\tilde s'_{\te\om}}\subset\tilde\cC_{\te\om,s}   
\end{equation}
where $\tilde s'_\om=\frac{s\tilde Q_{\te^{-1}\om}+\tilde H_{\te^{-1}\om}}{\tilde Q_{\te^{-1}\om}+\tilde H_{\te^{-1}\om}}<s$. 
Next, for every constant $C>0$ consider the cone $\cK_{\om,C}$ given by
$$
\cK_{\om,C}=\{g:\cE_\om\to[0,\infty): g(x)\leq Cg(x'),\, \forall x,x'\in\cE_\om\}.
$$
Then by applying \cite[Lemma 5.7.3]{HK} with $\tilde\phi$ instead of $\phi$  we see that for every $n\geq m(\om)$,
\begin{equation}\label{Inclusion 2}
L_\om^n\tilde\cC_{\om}\subset \cK_{\te^n\om, \tilde C_n(\om)}   \end{equation}
where 
$
\tilde C_n(\om)=D_{\om,n}e^{s\tilde Q_\om+2\|S_n^\om\tilde \phi\|_\infty},\,\,\,D_{\om,n}=\prod_{j=0}^{n-1}D_{\te^j\om}.
$
Next,  denote by $d_\cC(\cdot,\cdot)$ the Hilbert projective metric associated with a proper convex cone $\cC$ (see \cite{Bir}). 
Iterating \eqref{Inclusion 1} and combing the resulting inclusion with \eqref{Inclusion 2} and \cite[Eq.(5.7.18)]{HK} we conclude that when $n\geq m(\om)$ then
\begin{equation}\label{Diam est 1}
\Del_{n}(\om):=
\sup_{f,g\in\tilde\cC_\om}d_{\tilde \cC_{\te^n\om}}(L_\om^n f, L_\om^n g)\leq 
  \tilde d_n(\om)
\end{equation}
where 
$
 \tilde d_n(\om)=4\|S_n^\om \tilde \phi\|_\infty+2\ln (D_{\om,n})+2\ln (s''_{\te^{n}\om})+2s\tilde Q_\om,\,\,
\tilde s''_\om=\frac{2s}{s-1}\cdot \frac{\tilde Q_{\te^{-1}\om}}{2\tilde H_{\te^{-1}\om}}+1+\frac{s+1}{s-1}.
$
Next,
using Corollary \ref{cor h bound}, we see that 
$$
\tilde d_n(\om)\leq d_n(\om):=
4R_n(\om)+2\ln (D_{\om,n})+2\ln (\tilde s''_{\te^{n}\om})+2s\tilde Q_\om
$$
where $R_n$ is defined in Corollary \ref{cor h bound}, and when
 $\xi_\om=1$ we set $R_n=R_{0,n}$ with  $R_{J,M}$ defined like in Corollary \ref{cor h bound}.
Thus, for
$n\geq m(\om)$ we have
\begin{equation}\label{Diam est 2}
\Del_{n}(\om)\leq  d_n(\om).    
\end{equation}

Next,   let us define  $\tilde s''_{\om,r}$  similarly to $\tilde s''_\om$ but with $\bbE[\tilde Q_{\te^{-1}\om}|\cF_{-1-r,-1+r}]$ and $\bbE[\tilde H_{\te^{-1}\om}|\cF_{-1-r,-1+r}]$ instead of $\tilde Q_{\te^{-1}\om}$ and $\tilde H_{\te^{-1}\om}$, respectively.
We will also need the following result.
\begin{lemma}\label{s'' approx}
Let $k_1,k_2,k_3\geq 1$ be such that $\frac{1}{k_3}=\frac{1}{k_2}+\frac{1}{k_1}$. 
Then for all $r\in\bbN$,
$$
\|\tilde s''_{\om}-\tilde s''_{\om,r}\|_{L^{k_3}(\bbP)}\leq \frac{2s}{s-1}\left(\|\tilde Q_\om\|_{L^{k_2}}\tilde h_{k_1}(r)+\|\tilde H_\om\|_{L^{k_1}}\tilde q_{k_2}(r)\right).
$$
\end{lemma}
\begin{proof}
Denote $\tilde Q_{\om,r}=\bbE[\tilde Q_\om|\cF_r]$ and 
$\tilde H_{\om,r}=\bbE[\tilde H_\om|\cF_r]$.
Let $Y(\om)=\frac{\tilde Q_\om}{\tilde H_\om}$ and $Y_r(\om)=\frac{\tilde Q_{\om,r}}{\tilde H_{\om,r}}$. Then by the minimization property of conditional expectations and the definition of $\tilde s''_\om$,
$$
\|\tilde s''_{\om}-\tilde s''_{\om,r}\|_{L^{k_3}}=
\|(\tilde s''_{\om}-\tilde s''_{\om,r})\circ\te\|_{L^{k_3}}
\leq \frac{2s}{s-1}\|Y-Y_r\|_{L^{k_3}}.
$$
 Using that $\tilde H_\om\geq 1$ we have
$
|Y(\om)-Y_r(\om)|\leq \tilde Q_\om|\tilde H_\om-\tilde H_{\om,r}|+\tilde H_\om|\tilde Q_\om-\tilde Q_{\om,r}|
$
and the lemma follows by the H\"older inequality.
\end{proof}

Next, let us define  $d_{k,M}(\om)$ similarly to $d_{M}(\om)$ but with $R_{k,M}(\omega)$ instead of $R_M(\om)$, where when $\xi_\om=1$ we set $j_\om=m(\om)=0$. Here $R_{k,M}$ are the random variables defined in Corollary \ref{cor h bound}.
Let
$\bar d_{j,N}(\om)=\max_{k\leq j}d_{k,M}(\om)$.
Now,  let $M_0$ be large enough such that 
$
\bbP(\om: m(\om)\leq M_0)>\frac34
$
and given such $M_0$ let $J_0$ be large enough such that 
$
\bbP(\om: \max(j_\om, j_{\te^{M_0}\om})\leq J_0)>\frac 34.
$
Then 
$
\bbP(\om: m(\om)\leq M_0,\, \max(j_\om, j_{\te^{M_0}\om})\leq J_0)>\frac 12.
$
Let us also take $D_0>1$ large enough such that 
$
\bbP(\om: \bar d_{J_0,M_0}(\om)\leq D_0-1)>\frac 34.
$
Set 
$$
A_0=\{\om: m(\om)\leq M_0,\,\max(j_\om, j_{\te^{M_0}\om})\leq J_0,\,\bar d_{J_0,M_0}(\om)\leq D_0-1\}
$$
and 
$
A=\{\om: m(\om)\leq M_0, \,\max(j_\om, j_{\te^{M_0}\om})\leq J_0,\, \bar d_{J_0,M_0}(\om)\leq D_0\}.
$
Then  $A_0\subset A$ and $\bbP(A_0)\geq \frac 14$.

Next,  using that a linear map between two cones has contraction rate $\tanh(\Del)$, where $\Del$ is the projective diameter of the image of the cone (see \cite{Bir}), we see that if $\om\in A$ and $f,g\in \cC_{\om,s}$
$$
d_{\tilde\cC_{\te^M\om,s}}(L_\om^{M_0} f,L_\om^{M_0} g)\leq \tanh(D_0):=e^{-c}
$$
where $c>0$. Next, by applying the  contraction after $m(\om)$ iterates,  then applying it after $M_0$ iterates whenever $\te^{M_0 j}\om\in A$, for all $m(\om)\leq j\leq [(n-1)/M_0]$ and for the other indexes $j$ using  the one step weak contraction (which holds because of the inclusion $L_\om\tilde\cC_{\om,s}\subset\tilde\cC_{\te\om,s}$) 
we conclude that for $\bbP$-a.a. $\om$ and all $f,g\in \cC_{\om,s}$ and $n\geq M_0m(\om)$ we have
\begin{equation}\label{Random Contraction}
 d_{\tilde\cC_\om}(L_\om^n f,L_\om^n g)\leq U(\om)e^{-c\sum_{j=m(\om)}^{[(n-1)/M_0]}\bbI_A(\te^{M_0j}\om)},\,\,\,\,\,\,\,\,U(\om)=d_{m(\om)}(\om).   
\end{equation}

Next, we provide some integrability conditions for $U(\cdot)$.
\begin{lemma}\label{Mom Lemma Final}
Let $p_0,q_0,b,v,q\geq 1$ be such that $\frac 1q\leq \frac 1{p_0}+\frac{1}{q_0}$ and $\frac 1{qq_0}\leq \frac{1}{b_0}+\frac 1v$. Suppose  $\ln D_\om\in L^{qq_0}$ and   $H_\om,\tilde Q_\om, Q_\om\in L^{b_0}$. Assume also that 
\begin{equation}\label{Tails}
 \sum_{n\geq 1}n^{\frac{1}{q_0}}\left(\bbP(m(\om)\geq n)\right)^{\frac1{p_0}}<\infty\,\,   \text{ and }\,\,\sum_{n\geq 1}\left(\bbP(m(\om)\geq n)\right)^{\frac1{v}}<\infty.
\end{equation}
Then $U(\om)\in L^q$.  
\end{lemma}
\begin{proof}
First, 
$
(U(\om))^q=\sum_{n=0}^{\infty}\bbI(m(\om)=n)d_n^q(\om)
$
and so by the H\"older inequality, \begin{eqnarray}\label{1}
 &   \bbE_\bbP[U^q]\leq \sum_{n=0}^\infty\left(\bbP(m(\om)\geq n)\right)^{\frac{1}{p_0}}\|d_n\|_{L^{qq_0}}^{\frac1{q_0}}.
\end{eqnarray}
Next, we claim that there is a constant $C>0$ such that
\begin{equation}\label{2}
\|d_n\|_{L^{t}}\leq Cn,\, t=qq_0.     
\end{equation} 
Relying on \eqref{2}, the proof of the lemma is completed by \eqref{1} and \eqref{Tails}. In order to prove \eqref{2}, we first recall that
$
d_n(\om)=|d_n(\om)|=R_n(\om)+2\ln(D_{\om,n})+2\ln (s''_{\te^n\om})+2s\tilde Q_\om
$
where $R_n$ is defined in Corollary \ref{cor h bound}, and when
 $\xi_\om=1$ we set $R_n=R_{0,n}$ with  $R_{J,M}$ defined like in Corollary \ref{cor h bound}.
Now, using that $\ln(1+x)\leq x$ for all $x\geq0$ and that $\tilde H_\om\geq 1$ we see that $\ln (s''_{\te^n\om})\leq C(1+\tilde Q_{\te^{n-1}\om})$ for some constant $C>0$. Notice also that
\begin{eqnarray*}
&\|\ln(D_{\om,n})\|_{L^t}=\left\|\sum_{j=0}^{n-1}\ln D_{\te^j\om}\right\|_{L^t}\leq n\|\ln D_\om\|_{L^t}.   
\end{eqnarray*}
Thus, there are constants $C',C''>0$ such that
$$
\|d_n\|_{L^t}\leq \|R_n\|_{L^t}+C'\|\tilde Q_\om\|_{L^t}+2n\|\ln D_\om\|_{L^t}+C''.
$$
It remains to show that $\|R_n\|_{L^t}=O(n)$.
Set $J_n(\om)=\max(j_\om,j_{\te^n\om})$. Then 
\begin{eqnarray*}
&\|R_n\|_{L^t}\leq 3\left\|\sum_{j=-J_n(\om)}^{n}H_{\te^j\om}\right\|_{L^t}
+s\left\|Q_{\te^{-j_\om}\om}\right\|_{L^t}+\left\|Q_{\te^{n-j_{_{\te^n\om}}}\om}\right\|_{L^t}+ 3\left\|\sum_{k=-J_n(\om)}^{n}\ln D_{\te^k\om}\right\|_{L^t}
\\
&=I_{1,n}+I_{2,n}+I_{3,n}+I_{4,n}.
\end{eqnarray*}
To estimate $I_{1,n}$ notice that 
\begin{eqnarray*}
&\sum_{k=-J_n(\om)}^{n}H_{\te^k\om}=\sum_{k=-\infty}^{0}\bbI(J_n(\om)\geq 
|k|)H_{\te^k\om}+\sum_{k=1}^{n}H_{\te^k\om}.    
\end{eqnarray*}
Thus, using also   the triangle and the H\"older inequalities to estimate the norm of the first term on the right hand side we see that
\begin{eqnarray*}
\label{uu}
&\|I_{1,n}\|_{L^t}\leq3\left\|\sum_{k=-\infty}^{0}\bbI(J_n(\om)\geq 
|k|)H_{\te^k\om}\right\|_{L^t} +n\|H_\om\|_{L^t}\leq \sum_{k=-\infty}^{0}\left\|\bbI(J_n(\om)\geq 
|k|)H_{\te^k\om}\right\|_{L^t}
\\&+n\|H_\om\|_{L^t}
\leq \|H_\om\|_{L^{b_0}}\sum_{k=-\infty}^0\left(\bbP(J_n(\om)\geq |k|)\right)^{1/v}+n\|H_\om\|_{L^t} .\nonumber
\end{eqnarray*}
Notice next that 
\begin{eqnarray}\label{Uni Tail}
\bbP(J_n(\om)\geq \ell)\leq \bbP(j_\om\geq \ell)+\bbP(j_{\te^n\om}\geq \ell)=2\bbP(j_\om\geq \ell)\leq 2\bbP(m(\om)\geq \ell)    
\end{eqnarray}
where in the last inequality we have used Lemma \ref{tails lemma}. Taking into account \eqref{Tails} we see  that   that 
$
\|I_{1,n}\|_{L^t}=O(n).
$

Next, we estimate $I_{2,n}$ and $I_{3,n}$. Since 
$
Q_{\te^{-j_\om}\om}=\sum_{k\geq 1}^\infty \bbI(j_\om=k)Q_{\te^{-k}\om}
$
by the H\"older inequality,
\begin{equation}\label{I 2 n}
I_{2,n}=\|Q_{\te^{-j_\om}\om}\|_{L^t}\leq  
\|Q_\om\|_{L^{b_0}}\sum_{k\geq 1}\left(\bbP(j_\om\geq k)\right)^{1/v}\leq \|Q_\om\|_{L^{b_0}}\sum_{k\geq 1}\left(\bbP(m(\om)\geq k\right)^{1/v}:=C_{b_0,v}<\infty.
\end{equation}
where we have used Lemma \ref{tails lemma}.
Similarly, $I_{3,n}\leq C_{b_0,v}$
where we used that $j_{\te^n\om}$ and $j_\om$ are equally distributed. Using also  \eqref{Tails} we get
$
\max(I_{2,n},I_{3,n})\leq C_0\|Q_\om\|_{L^{b_0}}
$
for some constant $C_0$.

Next, let us bound $I_{4,n}$. First note that
$
\sum_{k=-J_n(\om)}^{n}\ln D_{\te^k\om}=\sum_{k=-\infty}^{0}\bbI(J_n\geq |k|)\ln D_{\te^k\om}+\sum_{k=1}^{n}\ln D_{\te^k\om}.
$
Thus
$
I_{4,n}\leq \left\|\sum_{k=-\infty}^{0}\bbI(J_n\geq |k|)\ln D_{\te^k\om}\right\|_{L^t}+n\|\ln D_\om\|_{L^t}.
$
Now, by the triangle and the H\"older inequalities and \eqref{Uni Tail},
\begin{eqnarray*}
&\left\|\sum_{k=-\infty}^{0}\bbI(J_n\geq |k|)\ln D_{\te^k\om}\right\|_{L^t}\leq \sum_{k=-\infty}^{0}\left\|\bbI(J_n\geq |k|)\ln D_{\te^k\om}\right\|_{L^t}
\\
&\leq \|\ln D_\om\|_{L^{b_0}}2^{1/v}\sum_{k=-\infty}^{0}\left(\bbP(m(\omega)\geq |k|)\right)^{1/v}
= \|\ln D_\om\|_{L^{b_0}}2^{1/v}\sum_{k\geq 0}\left(\bbP(m(\omega)\geq k)\right)^{1/v}.
\end{eqnarray*}
 Using the above estimates and \eqref{Tails} we see that  $\|R_n\|_{L^{q_0q}}=O(n)$. Combining this with the estimates on the rest of the $L^{q_0q}$ norms of the terms in the definition of $d_n$ we obtain \eqref{2}, and the proof of the lemma is complete. 
\end{proof}

Next, let $d_p(r),h_p(r)$ be like in \eqref{approx coef}.
The following result follows directly from the definition of $R_{J,M}$ in Corollary \ref{cor h bound} and Lemma \ref{s'' approx}.
\begin{lemma}\label{R approx}
 For all $p\geq 1$ and all $J,M\in\bbN$ we have 
 $$
\left\|R_{J,M}(\om)-\bbE[R_{J,M}(\om)|\cF_{-(r+M+J), (r+M+J)}]\right\|_{L^p}\leq 3(M+2J)(b_p(r)+d_p(r)+sq_{p}(r)).
 $$
 Consequently,  if $\frac{1}p\leq \frac1{k}+\frac 1m$ then there is an absolute constant $C>0$ such that
 \begin{eqnarray}\label{bar d est}
&\left\|\bar d_{J,M}(\om)-\bbE[\bar d_{J,M}(\om)|\cF_{-(r+M+J), (r+M+J)}]\right\|_{L^p}     
\\
&\leq Cs(J+M)\left(h_p(r)+d_p(r)+q_p(r)+\tilde q_p(r)+\|\tilde Q_\om\|_{L^k}\tilde h_{m}(r)+\|\tilde H_\om\|_{L^m}\tilde q_{k}(r)\right)
:=\del_{p,k,m}(r).\nonumber
 \end{eqnarray}
\end{lemma}

Combining Lemmata \ref{tilde H approx}, \ref{L1}, \ref{tilde Q approx} and \ref{R approx} we get the following result.
\begin{lemma}\label{delta lemma}
Let $p,u,\tilde u,\tilde p, \tilde b, b,b_2\geq 1$ be like in Lemma \ref{tilde Q approx}. Assume 
 \eqref{limsup ass} and that one of  (1)-(2) in Lemma \ref{L1} hold with $\tilde u$ instead of $u$. 
Then there is a constant $C>0$ such that for every $r\geq 1$ we have
$$
\del_{1,\tilde p,\tilde b}(2r)\leq C\left(d_1(r)+h_{p}(r)+n_{b_2}(r)+\tilde \zeta_r+\min(\beta_\infty(r),c_r\beta_{\tilde u}(r))\right)
$$   
where $c_r$ is defined in Lemma \ref{L1} and $\tilde \zeta_r$ is defined like in Lemma \ref{L1} but with $\tilde u$ instead of $u$.
\end{lemma}

\begin{proposition}\label{A prop}
 Suppose  that  $\del_{1,\tilde p,\tilde b}(r)\to 0$ as $r\to\infty$ for $\tilde p$ and $\tilde b$ like in Lemma \ref{delta lemma}.
Then  
 for every $r\in\bbN$ there are $A_r\in \cF_{-r_0-M_0-J_0, r+M_0+J_0}$ and   $B_r\in\cF$
such that 
$$
\bbP(B_r)\leq \del_{1,\tilde p,\tilde b}(r)+v_r,\, A_r\subset A\cup B_r,\,
p_0:=\lim_{r\to\infty}\bbP(A_r)=\bbP(A_0)>0
$$
where $v_r=O(r^{-M})$ and $M$ comes from Assumption \ref{Inner aprpox}.
\end{proposition}

\begin{proof}
Denote
$
 Q_{\om,r}=\bbE[Q_\om|\cF_r],\,\,
H_{\om,r}=\bbE[H_\om|\cF_r].
$
Define $\tilde Q_{\om,r}$ and $\tilde H_{\om,r}$ similarly, but with $\tilde Q$ and $\tilde H$ instead of $Q$ and $H$, respectively. 
Let
 $d_{M_0,J_0,r}(\om)$  be defined similarly to $d_{M_0,J_0}(\om)$ replacing each appearance of $\tilde Q_{\te^\ell\om}, \tilde H_{\te^\ell\om}$ 
 and $\ln D_{\te^\ell\om}$ (for some $\ell$) by
 $\tilde  Q_{\te^\ell\om,r}$, $\tilde H_{\te^\ell\om,r}$ and $\bbE[\ln D_{\te^\ell\om}|\cF_{\ell-r,\ell+r}]$, respectively.
 When $\xi_\om=1$ we set $J_0=0$, $d_{M_0,J_0}(\om)=d_{M_0}(\om)$  and define $d_{M_0,r}(\om)=d_{M_0,J_0,r}$ similarly to $d_{M_0}(\om)$ but with the above replacements. Let $\bar d_{J_0,M_0,r}=\max_{j\leq J_0} d_{j,M_0,r}$. 

 Next, let
 $$
U_1=\{\om: m(\om)\leq M_0\}, \,U_2=\{\om:j_\om\leq J_0\}, \,U_3=\{\om:j_{\te^{M_0}\om}\leq J_0\},\,
 U_4=\{\om: \bar d_{M_0, J_0}(\om)\leq D_0\}.
 $$
 Then 
\begin{eqnarray}\label{A U }
&A=U_1\cap U_2\cap U_3\cap U_4.   
\end{eqnarray}

By Assumption \ref{Inner aprpox} and Lemma \ref{Inner approx j} there are sets $ A_{i,r}\in\cF_r$ and  $B_{i,r}\in\cF$, where $i=1,2,3$, such that for $i=1,2,3$
$$
A_{i,r}\subset U_i\cup B_{i,r}, \,\,\lim_{r\to\infty}\bbP(A_{i,r})=\bbP(U_i)
$$
and $\bbP(B_{i,r})=O(v_r)$. 
 Now consider  
the property of a set $U$ that $A_r\subset U\cup B_r$ with $A_r\in\cF_r, B_r\in\cF$, $\lim_{r\to\infty}\bbP(A_r)=\bbP(U)$ and $\bbP(B_r)=O(v_r)$. Then this property is closed under finite intersections. We conclude that there are sets $\bar A_{r}\in\cF_r$  and $\bar B_r\in \cF$ such that 
$
\bar A_r\subset (U_1\cap U_2\cap U_3)\cup\bar B_r,\,\,\,\lim_{r\to\infty}\bbP(\bar A_r)=\bbP(U_1\cap U_2\cap U_3) 
$
and $\bbP(\bar B_r)=O(v_r)$.
Next, let 
$
A_{4,r}=\{\om: \bar d_{M_0,J_0,r}(\om)\leq D_0-1\}\in\cF_{-r-M_0-J_0,r+M_0+J_0}.
$
Then 
$$
\lim_{r\to \infty}\bbP(A_{4,r})=\bbP(d_{M_0,J_0}(\om)\leq D_0-1).
$$
Moreover, 
\begin{equation}\label{U claim}
A_{4,r}\subset U_4\cup B_{4,r},\,\,\,B_{4,r}=\{\om: |\bar d_{M_0,J_0}(\om)-\bar d_{M_0,J_0,r}(\om)|\geq 1\}.
\end{equation}
Now, by the Markov inequality
$
\bbP(B_{r,4})\leq \del_{1,\tilde p,\tilde b}(r).
$
Thus, all the results stated in the proposition hold with the sets
$
A_r=\bar A_r\cap A_{4,r}.
$
\end{proof}



The following result follows from Proposition \ref{A prop} and Lemma \ref{A Lemma}. 
 \begin{corollary}\label{Fin cor}
Let $c>0$ be like in \eqref{Random Contraction} and 
suppose that either \eqref{limsup ass} is in force with this $c$ or $\al(r)=O(r^{1-M})$.

 Suppose  $\del_{1,\tilde p,\tilde b}(r)=O(r^{-a})$ for some $a>3$ and 
 $\bbP(m(\om)\geq n)=O(n^{-\beta d-1-\ve_0})$ for some $d\geq1, \beta,\ve_0>0$ such that $a>\beta d+3$. Then there is a random variable $K(\om)\in L^d$ such that $\bbP$-a.s. for all $n\geq 1$ we have 
$
 e^{-c\sum_{j=m(\om)}^{[(n-1)/M_0]}\bbI_A(\te^{M_0j}\om)}   \leq K(\om)n^{-\beta}.
$

\end{corollary}

\section{Proof of  Theorem \ref{OSC1}}
\label{OSC 12} 
\subsubsection{Estimates in the $\|\cdot\|_{\infty,\al}$ norm-Proof of Theorem \ref{OSC1} (ii)}

Let $\cK_{\om,+}$ be the cone of positive functions $f:\cE_\om\to\bbR$. Let $n\geq M_0m(\om)$. Then 
by \eqref{Random Contraction} and taking into account that $L_\om^n \textbf{1}=\textbf{1}$ we see that for every $f\in\tilde\cC_{\om,s}$ we have 
$$
d_{\cK_{\te^n\om,+}}(L_\om^n f, \textbf{1})\leq d_{\tilde\cC_{\te^n\om}}(L_\om^n f, \textbf{1})\leq U(\om)e^{-c\sum_{j=Mm(\om)}^{[\frac{n-1}{M_0}]}\bbI_{A}(\te^{M_0j}\om)}
$$
where we also used  that $\cK_{\om,+}\subset \tilde\cC_\om$, which implies that $d_{\cK_{\om,+}}\leq d_{\tilde\cC_{\om}}$. By applying \cite[Lemma 3.5] {Kifer Thermo} with the measure $\mu_{\te^n\om}$ and the functions $f=\textbf{1}$ and $g=\frac{L_\om^n g}{\mu_{\te^n\om}(L_\om^n g)}=\frac{L_\om^n g}{\mu_{\om}(g)}$ we see that for all $n\geq M_0m(\om)$,
$$
\|L_\om^n g-\mu_\om(g)\|_{\infty}\leq \mu_\om(g) U(\om)e^{-c\sum_{j=M_0m(\om)}^{[\frac{n-1}{M_0}]}\bbI_{A}(\te^{M_0j}\om)}.
$$
Now, using Lemma \ref{regeneration}  and that $\tilde Q_\om\geq \tilde H_{\te^{-1}\om}\gamma_{\te^{-1}\om}^{-\al}\geq H_{\te^{-1}\om}\gamma_{\te^{-1}\om}^{-\al}$
 we see that for every H\"older continuous function $g$,
\begin{equation}\label{eeest}
 \|L_\om^n g-\mu_\om(g)\|_{\infty}\leq CU(\om)\xi_\om^{-1}\left(1+\frac{\gamma_{\te^{-1}\om}^{\al}}{H_{\te^{-1}\om}}\right)\|g\|_\al  e^{-c\sum_{j=Mm(\om)}^{[\frac{n-1}{M_0}]}\bbI_{A}(\te^{M_0j}\om)}   
\end{equation}
where $C$ is a constant.  
Now the estimates in Theorem \ref{OSC1} (ii)  when $n\geq M_0m(\om)$ follow from Corollary \ref{Fin cor} and Lemma \ref{Mom Lemma Final}. 

To prove that the estimates in Theorem \ref{OSC1} (ii) also hold when $n\leq M_0m(\om)$ we first observe that since $L_\om^n\textbf{1}=\textbf{1}$, for every bounded function $g:\cE_\om\to\bbR$ we have
\begin{equation}\label{Basic bb}
  \|L_\om^n g-\mu_\om(g)\|_\infty\leq 2\|g\|_\infty.  
\end{equation}

To obtain the estimates in Theorem \ref{OSC1} (ii) for all $n\leq M_0m(\om)$, we first note that such $n$'s, 
$$
\|L_\om^n g-\mu_\om(g)\|_\infty\leq 2\|g\|_\infty\leq 2\|g\|_\infty (M_0m(\om))^{\beta}n^{-\beta}.
$$
Thus, if the statement of Theorem \ref{OSC1} (ii) holds true with some $R(\om)$ and all $n\geq M_0m(\om)$ then it holds with all $n\geq1$ with 
$
\tilde R(\om)=\max(2(M_0m(\om))^{\beta},R(\om))
$
instead of $R(\om)$. Therefore, all that is left to do in order to complete the proof of Theorem \ref{OSC1} (ii) is to show that 
$
\om\mapsto (m(\om))^\beta\in L^{q}(\Om,\cF,\bbP).
$
To show that $(m(\om))^\beta\in L^{q}(\Om,\cF,\bbP)$ we first note that
$
\left\|(m(\om))^\beta\right\|_{L^q}=\|m(\om)\|_{q\beta}^{\beta}.
$
Now, using that in Assumption \ref{Poly Ass} we have $d\geq q$ we see that 
\begin{eqnarray*}
&    \|m(\om)\|_{q\beta}^{q\beta}=\int_{0}^\infty\bbP(m(\om)\geq t^{\frac{1}{q\beta}})dt\leq \sum_{k\geq 1}\bbP(m(\om)\geq [k^{\frac{1}{q\beta}}])<\infty
\end{eqnarray*}
and the proof of Theorem \ref{OSC1} (ii) is complete.
\qed

\subsubsection{H\"older norm estimates: proof of Theorem \ref{OSC1} (iii)}
 The proof is based on the theory of complex projective metrics developed in \cite{Dub, Rug} and the fact that the real cones are embedded in their canonical complexifications (we refer to \cite[Appendix A]{HK} for more details). 

Let $x_{\om,i}, i\leq M_\om$ be points in $\cE_\om$ such that 
$
\cE_\om=\bigcup_{j=1}^{M_\om}B_\om(x_{\om,i},\xi_\om).
$
\begin{lemma}\label{Aper Lemma}
Let $l_\om$ 
 be the linear functional given by $l_\om(g)=\sum_{j=1}^{M_\om}g(x_{\om,j})$. Then for $\bbP$-a.a. $\om$ and all $g\in\tilde\cC_{\om,s}$  we have 
$$
\|g\|_\al\leq K_\om l_\om(g),\,\,\,K_\om=3(e^{s\xi_\om^\al \tilde Q_\om}+s\tilde Q_\om e^{2s\xi_\om^\al \tilde Q_\om})\xi_\om^{-1}.
$$
\end{lemma}

\begin{proof}
   First, by \cite[Lemma 3.11]{MSU}, for every $g\in\tilde\cC_{\om,s}$ we have 
    $
v_{\al,\xi_\om}(g)\leq s\tilde Q_\om e^{s\xi_\om^\al \tilde Q_\om}\|g\|_\infty
$
and so 
$
\|g\|_{\al,\xi_\om}\leq (1+ s\tilde Q_\om e^{s\xi_\om^\al \tilde Q_\om})\|g\|_\infty .
$
Thus by Lemma \ref{Norm comp}, 
$$
\|g\|_{\al}\leq 3\xi_\om^{-1} (1+ s\tilde Q_\om e^{s\xi_\om^\al \tilde Q_\om})\|g\|_\infty.
$$
Next, we estimate $\|g\|_\infty$. Let $x_0\in\cE_\om$ be such that 
$
\|g\|_\infty=g(x_0).
$
Let $j$ be such that $d(x_0,x_{\om,j})<\xi_\om$. Then 
\begin{eqnarray*}
&g(x_0)\leq e^{s\tilde Q_\om d^\al(x_0,x_{\om,j})}g(x_{\om,j})\leq e^{s\tilde Q_\om \xi_\om^\al}\sum_{k=1}^{M_\om}g(x_{k,\om})=e^{s\tilde Q_\om \xi_\om^\al} l_\om(g).    
\end{eqnarray*}
\end{proof}

\begin{corollary}
In the circumstances of Corollary \ref{Fin cor}, for all  $n\geq M_0m(\om)$ we have
     $$
\|L_\om^n -\mu_\om\|_\al
\leq  U(\om)K(\om)V_\om (1+2K_{\te^n\om}M_{\te^n\om})n^{-\beta}
$$  
where $U(\om)$ is like in \eqref{Random Contraction}, 
$
V_\om=12\xi_\om^{-1}\left(1+\frac{4\gamma_{\te^{-1}\om}^{\al}}{H_{\te^{-1}\om}}\right)
$
 and $K(\cdot)$ is like in Corollary \ref{Fin cor}.
\end{corollary}
\begin{proof}
Let $g\in\tilde \cC_{\om,s}$. Then $L_\om^n g\in\tilde\cC_{\te^n\om,s}$. Using that the cones $\tilde\cC_{\om,s}$ contain constant functions and  $L_\om^n\textbf{1}=\textbf{1}$,  applying
  \cite[Theorem A.2.3]{HK} (ii) and taking into account that $\|l_\om\|_\al=M_\om$ we have
  $$
\left\|\frac{L_\om^n g}{M_{\te^n\om}^{-1}l_{\te^n\om}(L_\om^n g)}-\textbf{1}\right\|_\al\leq K_{\te^n\om}M_{\te^n\om} d_{\tilde\cC_{\te^n\om,s}}(L_\om^n g,\textbf{1}).
  $$
  Next, notice that 
$
l_{\te^n\om}(L_\om^n g)\leq \|g\|_\infty l_{\te^n\om}(L_\om^n \textbf{1})=M_{\te^n\om}\|g\|_\infty
$
and so 
$
M_{\te^n\om}^{-1}l_{\te^n\om}(L_\om^n g)\leq \|g\|_\infty.
$
Therefore, with $\bar l_\om=M_\om^{-1}l_\om$,
$$
\|L_\om^n g-\bar l_{\te^n\om}(L_\om^n g)\|_\al\leq  \|g\|_\infty K_{\te^n\om}M_{\te^n\om} d_{\tilde\cC_{\te^n\om,s}}(L_\om^n g,\textbf{1})\leq U(\om)K(\om)K_{\te^n\om}M_{\te^n\om}\|g\|_\infty a_n
$$
where in the second inequality we have used \eqref{Random Contraction} and Corollary \ref{Fin cor}.
Using Lemma \ref{regeneration} we conclude that for every H\"older continuous function $g:\cE_\om\to\bbR$ and $n\geq M_0m(\om)$ we have 
$$
\|L_\om^n g-\bar l_{\te^n\om}(L_\om^n g)\|_\al\leq 
\tilde U_\om U(\om)K(\om)K_{\te^n\om}M_{\te^n\om}\|g\|_\al n^{-\beta}
$$
where 
$
\tilde U_\om=12\xi_\om^{-1}\left(1+\frac{4}{s\tilde Q_\om}\right)\leq 
12\xi_\om^{-1}\left(1+\frac{4\gamma_{\te^{-1}\om}^{\al}}{s\tilde H_{\te^{-1}\om}}\right)\leq 
12\xi_\om^{-1}\left(1+\frac{4\gamma_{\te^{-1}\om}^{\al}}{H_{\te^{-1}\om}}\right)=V_\om.   $ 

Next, recall that for a linear operator $A$ we have $\|A\|_{\infty,\al}=\sup\{\|A g\|_\infty: \|g\|_\al\leq 1\}$ and that given a linear functional $\nu$,  when viewed as the operator $g\mapsto \nu(g)\textbf{1}$ we have $\|\nu\|_\al=\|\nu\|_{\infty,\al}$.
Thus, using also \eqref{eeest} and Corollary \ref{Fin cor},
$$
\|L_\om^n-\mu_\om\|_\al\leq \|L_\om^n-\bar l_{\te^n\om}\circ L_\om^n\|_\al+\|\bar l_{\te^n\om}\circ L_\om^n-\mu_\om\|_\al\leq 
\|L_\om^n-\bar l_{\te^n\om}\circ L_\om^n\|_{\al}
$$
$$
+\|\bar l_{\te^n\om}\circ L_\om^n-L_\om^n\|_{\infty,\al}+\|L_\om^n-\mu_\om\|_{\infty,\al}\leq U(\om)K(\om)V_\om(1+2K_{\te^n\om}M_{\te^n\om}) n^{-\beta}
$$
and the proof of the lemma is complete.
\end{proof}

Next, let $R_1(\om)=C_0e^{4s\tilde Q_\om}=e^{4s\tilde Q_\om+\ln C_0}$. Then for $C_0$ large enough  $K_\om\xi_\om=3(e^{s\xi_\om^\al \tilde Q_\om}+s\tilde Q_\om e^{2s\xi_\om^\al \tilde Q_\om})\leq R_1(\om)$.
The estimates in Theorem \ref{OSC1} (iii)  will follow when $n\geq M_0m(\om)$ once we show that the tails of the first visiting time to sets of the form $\{R_1(\om)\leq C\}$ decay to $0$ at the rates described in those theorems. 
Let $C>1$ be large enough so that 
$$
\bbP(\tilde Q_\om\leq C-1)>0.
$$
Set $A=A(C)=\{\om: \tilde Q_\om\leq C\}$ and let $n_1(\om)=\min\{n\geq 1: \te^n\om\in A\}$, which is well defined $\bbP$-a.s. since $\te$ is ergodic. 
By the definition of $R_1(\om)$ we see that for every $C'>0$ there exists $C>0$ such that 
$\{\om:R_1(\om)\leq C'\}\subset A(C).$
Thus, to complete the proof of Theorem \ref{OSC1} (iii) when $n\geq M_0m(\om)$ it is enough to prove the following result.

\begin{lemma}\label{R1 lem}
Suppose that either 
$\limsup_{r\to\infty}\psi_U(r)<\frac{1}{\bbP(\tilde Q_\om>C-1)}-1$   where $1/0:=\infty$, or $\al(r)=O(r^{-A}), A>1$.
 Assume also that $\tilde q_1(r)=O(r^{-A})$.
 Then for every $\varepsilon>0$,
$$
\bbP(n_1(\om)\geq k)=O(k^{-(1+A-\varepsilon)}).
$$

\end{lemma}

\begin{proof}
Denote $\tilde Q_{\om,j,r}=\bbE[\tilde Q_{\te^j\om}|\cF_{j-r,j+r}]$ and $Q_{\om,r}=Q_{\om,0,r}$.  Let us also set
$Y_\om=\tilde Q_\om$ and let $Y_{\om,j,r}=\tilde Q_{\om,j,r}$ and $Y_{\om,r}=Y_{\om,0,r}$.
Then for every integer $1\leq r\leq \frac13 k$,
\begin{eqnarray*}
&\bbP(n_1(\om)>k)\leq \bbP(Y_{\te^j\om}>C; j\leq k)\leq \bbP(Y_{\te^{jr}\om}\geq C; j\leq [\frac{k}{3r}])\leq [\frac{k}{3r}]\bbP(|Y_\om-Y_{\om,r}|\geq 1) 
\\
&+\bbP(Y_{\te^{jr}\om,jr,r}>C-1; j\leq [\frac{k}{3r}]):=I_1+I_2.
\end{eqnarray*}
Next, by the Markov inequality,
$
I_1 \leq  [\frac{k}{3r}] \tilde q_1(r).
$
On the other hand, by Lemma \ref{Basic alpha cor},\,$I_2\leq [\frac{k}{3r}]\alpha(r)+\big(\bbP(Y_{\om,r}\geq C-1)\big)^{[\frac{k}{3r}]}$  and by Lemma \ref{psi Lemm 2}, 
$
I_2\leq \left((1+\psi_U(r))\bbP(Y_{\om,r}\geq C-1)\right)^{[\frac k{3r}]}.
$
Notice that $\lim_{r\to\infty}\bbP(Y_{\om,r}\geq C-1)=\bbP(Y_\om\geq C-1)$. Thus in both cases for $r$ large enough we have $\bbP(n_1(\om)>k)\leq C(kr^{-1-A}+b^{k/r})$ for some $0<b<1$. Now the lemma follows by taking $r=k^{1-\delta}$ for a sufficiently small $\delta>0$.
\end{proof}

Finally, let us prove the corresponding estimates when $n\leq M_0m(\om)$. First, by \cite[Lemma 5.6.1]{HK} there is  a constant $C>0$ such that $\bbP$-a.s. for all $n\geq1$, 
\begin{eqnarray*}
&\|L_\om^n\|_\al\leq C(1+\prod_{j=0}^{n-1}\gamma_{\te^j\om}^{-\al}+\tilde Q_{\te^n\om}).
\end{eqnarray*}
Thus, if $n\leq M_0M(\om)$ then
\begin{eqnarray*}
&\|L_\om^n-\mu_\om\|_\al\leq 
2C(1+\prod_{j=0}^{n-1}\gamma_{\te^j\om}^{-\al}+\tilde Q_{\te^n\om})(M_0m(\om))^{\beta}n^{-\beta}.
\end{eqnarray*}
 Now, as shown at the end of the proof of Theorem \ref{OSC1} (ii),\, $(M_0m(\om))^\beta\in L^{q}(\Om,\cF,\bbP)$.
Moreover, combining \eqref{gamm prod exp} with Lemma \ref{Simple}, 
$
\prod_{j=0}^{n-1}\gamma_{\te^j\om}^{-\al}\leq K(\om)n^{-\beta}
$
with $K(\om)\in L^{t}(\Om,\cF,\bbP)$,
where   $t$ given by $1/t=1/d+1/q$. Thus, the  desired estimates follow also when $n\leq M_0m(\om)$.
Furthermore,  by Lemma \ref{R1 lem}  sets of the form $\{\tilde Q_{\om}\leq c\}$ (for $c$ large enough) have first visiting times with sufficiently fast decaying tails, and the proof of Theorem \ref{OSC1} (iii) is complete, taking also into account Lemmata \ref{tilde H approx} and \ref{tilde Q approx}.

\section{Sequences of complex transfer operators}
In order to increase readability  in this section we will provide certain estimates and block partitions using sequential notations.
\subsection{An operator theory perturbative via ``sequential inducing"}\label{Sec1}

Let $\mathcal Y_j:\mathcal B_j\to\mathcal B_{j+1}$ be  operators acting on Banach spaces $(\mathcal B_j)$.
We assume here that there are $\beta>\varepsilon>0$ and $C_1\geq 1$ such that for all $j\geq 0$,
\begin{equation}\label{strong}
\|\mathcal Y_j^n-\mu_j\|\leq C_1(1+j^\varepsilon) n^{-\beta},   
\end{equation}
where
$
\mathcal Y_j^n:=\mathcal Y_{j+n-1}\circ \ldots \circ \mathcal Y_j
$
and $\mu_j\in \mathcal B_j^*$ satisfy $C_2:=\sup_j\|\mu_j\|<\infty$.

Let us fix some $0<\varepsilon_0<1$ and let $u_0=2\varepsilon_0^{-1}+1$.
We now construct contracting blocks $B_j$. Let $n_0=0$ and define $n_1$ by $n_1=[(u_0C_1)^{1/\beta}]$. Then 
$C_1n_1^{-\beta}<\varepsilon_0$.
We define $B_0=\bbN\cap [0, n_1-1]$. Next, we take $n_2=[(u_0 C_1)^{1/\beta}n_1^{\varepsilon/\beta}]$. Then  $C_1 (1+n_1^{\varepsilon})n_2^{-\beta}<\varepsilon_0$ (where we used that $\frac{n_1^\varepsilon+1}{n_1^\varepsilon}\leq 2$). 
Set $B_1=[n_1, n_1+n_2-1]\cap\bbN$. Continuing this way, we define $n_{j}$ recursively by $n_{j+1}=[(u_0C_1)^{1/\beta}N_j^{\varepsilon/\beta}]$, where $N_j=n_1+\ldots+n_j$ and $N_0:=0$. We set $B_j=[N_{j}, N_j+n_{j+1}-1]\cap\mathbb N, j>0$. For a finite interval $B=\{a,a+1, \ldots ,b\}$ in $\mathbb Z$ we denote
$
\mathcal Y_B=\mathcal Y_{b}\circ\ldots\circ \mathcal Y_{a+1}\circ\mathcal Y_a.
$
Then for each $j$ we have 
\begin{equation}\label{Half cont}
\|\mathcal Y_{B_j}-\mu_{N_j}\|\leq \varepsilon_0.    
\end{equation}
Next, observe that  for  all $j$,
$
C_3N_j^{\varepsilon/\beta}\leq N_{j+1}-N_j=n_{j+1}\leq C_4 N_j^{\varepsilon/\beta}
$
where $C_3=(u_0C_1)^{1/\beta}-1$ and $C_4=(u_0C_1)^{1/\beta}$. 
Using this relation we get the following result. 
\begin{lemma}\label{N lemma}
For all $j$ we have
$
A_1j^{\frac{\beta}{\beta-\varepsilon}}\leq N_j\leq A_2 j^{\frac{\beta}{\beta-\varepsilon}}
$
where 
$
A_1=\frac12\left(\frac{\beta-\varepsilon}{\beta}\right)^{\frac{\beta}{\beta-\epsilon}}C_3\,\,\text{ and }\,\,
A_2=C_4^{\frac{\beta}{\beta-\varepsilon}}.
$
Consequently, 
$$
A_1'j^{\frac{\varepsilon}{\beta-\varepsilon}}\leq n_{j+1}=|B_j|\leq A_2'j^{\frac{\varepsilon}{\beta-\varepsilon}}
$$
where 
$
A_1'=C_3 A_1^{\frac{\varepsilon}{\beta}}=c_0C_3^{1+\frac{\varepsilon}{\beta}},\,\, A_2'=C_4 A_2^{\frac{\varepsilon}{\beta}}
$
and $c_0>0$ depends only on $\varepsilon$ and $\beta$. 
\end{lemma}
\begin{proof}
Write $\frac{\varepsilon}{\beta}=\frac{\gamma}{1+\gamma}$ (i.e. $\gamma=\frac{\varepsilon}{\beta-\varepsilon}$).
For $c>0$ define $M_j=M_j(c)=cj^{1+\gamma}$. Then 
$$
M_{j+1}-M_j\geq c(1+\gamma)j^{\gamma}=c^{1-\frac{\gamma}{\gamma+1}}(1+\gamma)M_j^{\frac{\gamma}{1+\gamma}}=c^{1-\frac{\gamma}{\gamma+1}}(1+\gamma)M_j^{\varepsilon/\beta}.
$$
Thus, if $c=M_1\geq N_1=[C_4]$ and $c^{1-\frac{\gamma}{\gamma+1}}(1+\gamma)\geq C_4$ then it follows by induction that $N_j\leq M_j$ for all $j\geq 1$.
Therefore we can take $A_2$ as in the statement of the lemma.
Similarly,
$$
M_{j+1}-M_j\leq c(1+\gamma)(j+1)^{\gamma}\leq 2^\gamma(1+\gamma)cj^{\gamma}=2^\gamma(1+\gamma)c^{1-\frac{\gamma}{\gamma+1}}M_j^{\frac{\gamma}{1+\gamma}}=2^\gamma(1+\gamma)c^{\frac{1}{\gamma+1}}M_j^{\varepsilon/\beta}.
$$
Thus, by taking $c$ such that $c=M_1\leq N_1=[C_4]$ and $2^\gamma(1+\gamma)c^{\frac1{1+\gamma}}\leq C_3$
it follows by induction that $M_j\leq N_j$ for all $j\geq 1$. This leads to the lower bound $A_1$ in the statement of the lemma (as $A_1$ does not exceed the minimum of the upper bounds of $c$ obtained in the above manner).
\end{proof}

\subsection{Non-uniformly expanding maps and a sequential perturbative lemma}
To simplify the notation and for future applications, 
we consider here a sequential version of the maps in Section \ref{Maps sec}.
Let $T_j:\mathcal X_j\to \mathcal X_{j+1}$ be maps from metric spaces $(\mathcal X_j,\rho_j)$ normalized in size so that $\text{diam}(X_j)\leq 1$. We assume there are sequences $\gamma_j\geq1$, $d_j\in \mathbb N$ and $\xi_j\in(0,1]$ such that for every $x,x'\in\mathcal X_{j+1}$ with $\rho_{j+1}(x,x')\leq\xi_{j+1}$,
$$
T_j^{-1}\{x\}=\{y_{j,1}, \ldots ,y_{j,d_j}\},\,\, T_j^{-1}\{x'\}=\{y_{j,1}',\ldots ,y_{j,d_j}'\}
$$
where $d_j=d_j(x)$
and for every  $k$,
$$
\rho_{j}(y_{j, k},y'_{j, k})\leq\min( (\gamma_j)^{-1}\rho_{j+1}(x,x'), \xi_j).
$$
It follows by induction that 
every $x,x'\in\mathcal X_{j+n}$ with $\rho_{j+n}(x,x')\leq\xi_{j+n}$ we have 
$$
T_j^{-n}\{x\}=\{y_{j,n,1},\ldots ,y_{j,n,d_{j,n}}\},\,\, T_j^{-n}\{x'\}=\{y_{j, n, 1}',\ldots ,y_{j, n, d_j}'\}
$$
where $d_{j,n}=d_{j,n}(x)$
and for every  $k$ and $0\leq s<n$,
$$
\rho_{j}(T_j^sy_{j,n,k},T_j^sy'_{j,n,k})\leq\min\left( \left(\gamma_{j+s}\cdot\gamma_{j+s+1}\cdots \gamma_{j+n-1}\right)^{-1}\rho_{j+n}(x,x'), \xi_{j+s}\right).
$$

Next, let us fix some $\alpha\in(0,1]$ and let us take two sequences of functions $\phi_j,f_j:\mathcal X_j\to\mathbb R$ such that $\|\phi_j\|_{\alpha}<\infty$. For a complex number $z$ and a function $g:\mathcal X_{j}\to\mathbb C$ we define 
$$
L_{j,z}g(x)=\sum_{y: T_jy=x}e^{\phi_j(y)+zf_j(y)}g(y).
$$
Let 
$
L_{j,z}^n=L_{j+n-1,z}\circ\ldots\circ L_{j+1,z}\circ  L_{j,z}
$
and denote $L_j=L_{j,0}$ and $L_j^n=L_{j,0}^n$. 

Let $\|g\|_{\alpha,r}=v_{\alpha,r}(g)+\|g\|_\infty$, where $v_{\alpha,r}$ is the H\"older constant restricted to pairs of points whose distance does not exceed $r$. We have the following version of Lemma \ref{Norm comp}.

\begin{lemma}\label{Norm comp1}
 For every function $f$ and every $0<r<1$  we have 
 $$
v_{\alpha,r}(f)\leq v_{\alpha}(f)\leq \max(2\|f\|_\infty r^{-\alpha}, v_{\alpha,r}(f)) \,\,\text{ and }\,\,\|f\|_{\alpha,r}\leq \|f\|_{\alpha}\leq 3r^{-\alpha}\|f\|_{\alpha,r}.
 $$
\end{lemma}

The following result is an important ingredient in our approach.
 
\begin{lemma}\label{lemma TO dist}
For every function $g:\mathcal X_j\to\mathbb C$ such that $\|g\|_{\alpha,\xi_j}\leq 1$ we have 
\begin{equation}\label{EstTO}
\|L_{j,z}^ng-L_{j}^ng\|_{\alpha,\xi_{j+n}}\leq |z|e^{|\text{Re}(z)|\|S_{j,n}f\|_{\infty}}\| L_j^n\textbf{1}\|_\infty\left((1+(\Gamma_{j,n})^\alpha+2Q_{j,n}(\phi))\|S_{j,n}f\|_\infty+Q_{j,n}(f)\right)    
\end{equation}
where $\Gamma_{j,n}=\prod_{s=j}^{j+n-1}\gamma_{s}^{-1}$ and
for a sequence of functions $h=(h_j)$,
$$
Q_{j,n}(h)=\sum_{k=0}^{n-1}v_{\alpha,\xi_{j+k}}(h_{j+k})(\gamma_{j+k}\cdots \gamma_{j+n-1})^{-\alpha}=\sum_{k=0}^{n-1}v_{\alpha,\xi_{j+k}}(h_{j+k})(\Gamma_{j+k,n-k})^\alpha.
$$
Consequently, by Lemma \ref{Norm comp1} we have 
\begin{equation}\label{EstTO1}
\|L_{j,z}^n-L_{j}^n\|_{\alpha}\leq 3\xi_{j+n}^{-1} R_{j,n}(z)    
\end{equation}
where $R_{j,n}(z)$ is the right hand side of \eqref{EstTO}.
\end{lemma}


\begin{proof}[Proof of Lemma \ref{lemma TO dist}]
 Let $g:\mathcal X_j\to\mathbb C$ be such that $\|g\|_{\alpha,\xi_j}\leq 1$.  Then for every $x\in\mathcal X_{j+n}$ we have 
 $$
|L_{j,z}^ng(x)-L_{j}^ng(x)|\leq L_j^n(|g||e^{zS_{j,n}f}-1|) (x)\leq \|g\|_\infty\|e^{zS_{j,n}f}-1\|_\infty L_{j}^n\textbf{1}(x)
$$
$$
\leq |z|\|S_{j,n}f\|_\infty e^{|\text{Re}(z)|\|S_{j,n}f\|_{\infty}}\|L_j^n\textbf{1}\|_\infty
 $$
 where in the last estimate we used that $|e^{zx}-1|\leq e^{|\text{Re}(z)x|}|zx|$ for all real $x$.

 Next, we estimate the H\"older constant of $F(x):= L_{j,z}^ng(x)- L_{j}^ng(x)$. Let $x,x'\in\mathcal X_{j+n}$ be such that $\rho_{j+n}(x,x')\leq\xi_{j+n}$. Let us write 
 $$
F(x)-F(x')=\sum_{k}e^{S_{j,n}\phi(y_k)}(e^{zS_{j,n}f(y_k)}-1)g(y_k)-
\sum_{k}e^{S_{j,n}\phi(y'_k)}(e^{zS_{j,n}f(y'_k)}-1)g(y'_k)=:
I_1+I_2+I_3,
 $$
 where with $y_{k}=y_{j,k,n}$ and $y'_{k}=y'_{j,k,n}$ we have
 \begin{eqnarray*}
 &I_1=\sum_{k}e^{S_{j,n}\phi(y_k)}(e^{zS_{j,n}f(y_k)}-1)[g(y_k)-g(y'_k)],    I_2=\sum_{k}e^{S_{j,n}\phi(y_k)}[e^{zS_{j,n}f(y_k)}-e^{zS_{j,n}f(y_k')}]g(y_k'),\,\,\\
 &I_3=\sum_{k}[e^{S_{j,n}\phi(y_k)}-e^{S_{j,n}\phi(y_k')}](e^{zS_{j,n}f(y_k')}-1)g(y_k').
 \end{eqnarray*}
 To estimate $I_1$, using that
 $
|g(y_k)-g(y'_k)|\leq (\rho_j(y_k,y_{k}'))^\alpha\leq \Gamma_{j,n}^{\alpha}(\rho_{j+n}(x,x'))^\alpha
 $
we see that 
$$
|I_1|\leq \Gamma_{j,n}^{\alpha}(\rho_{j+n}(x,x'))^\alpha L_j^n\textbf{1}(x)e^{|\text{Re}(z)|\|S_{j,n}f\|_\infty}|z|\|S_{j,n}f\|_\infty.
$$
To estimate $I_2$,  notice that 
$$
|S_{j,n}f(y_k)-S_{j,n}f(y_k')|\leq \sum_{s=0}^{n-1}|f_{j+s}(T_j^sy_k)-f_{j+s}(T_j^sy_k')|\leq (\rho_{j+n}(x,x'))^\alpha Q_{j,n}(f).
$$
Using also the mean value theorem to estimate the terms $|e^{zS_{j,n}f(y_k)}-e^{zS_{j,n}f(y_k')}|$, we see that
$$
|I_2|\leq L_j^n\textbf{1}(x)e^{|\text{Re}(z)|\|S_{j,n}f\|_\infty}|z|Q_{j,n}(f) (\rho_{j+n}(x,x'))^\alpha.
$$
Finally, using again that  that 
$
|e^{zS_{j,n}f(y'_k)}-1|\leq e^{|\text{Re}(z)|\|S_{j,n}f\|_\infty}|z|\|S_{j,n}f\|_\infty
$
and a similar estimate
we see that  
$$
|I_3|\leq [L_j^n\textbf{1}(x)+L_j^n\textbf{1}(x')]e^{|\text{Re}(z)|\|S_{j,n}f\|_\infty}Q_{j,n}(\phi)|z|\|S_{j,n}f\|_\infty(\rho_{j+n}(x,x'))^\alpha
$$
where we used that
$
|e^{S_{j,n}\phi(y_k)}-e^{S_{j,n}\phi(y_k')}|\leq 
\max(e^{S_{j,n}\phi(y_k)},e^{S_{j,n}\phi(y_k')})|S_{j,n}\phi(y_k)-S_{j,n}\phi(y_k')|
$
which follows from the mean value theorem.
\end{proof}

\subsection{Inducing (sequential form)}\label{Sec SDS1}
Let the maps $T_j$ be like in the previous section and suppose $L_j\textbf{1}=\textbf{1}$, $(T_j)_*\mu_j=\mu_{j+1}$ and 
\begin{equation}\label{stronger}
\|L_j^n-\mu_j\|_\alpha\leq E_0 R(j+n)(1+j^{\bar \varepsilon})n^{-\beta}
\end{equation}
where $E_0>0$ is a constant, $R(\cdot)$ has the property that for some $D>0$ large enough we have 
$$
\{m: R(m)\leq D\}=\{m_1<m_2< \ldots \}
$$
for some sequence $(m_n)_{n=1}^\infty$ such that $ m_n\leq E_2n^{1+\eta_0}$ for some $E_2,\eta_0>0$ and $m_{n+1}-m_n\leq E_3n^{\delta}$, $\delta\geq 0$ for some $E_3\geq 1$. Notice also that $m_n\geq n$.  
\begin{remark}\label{Large diff rem}
$\forall k<s,\,
 m_k-m_s=\sum_{j=0}^{k-s-1}(m_{s+j+1}-m_{s+j})\leq E_3\sum_{j=0}^{k-s-1}(j+s+1)^\delta\leq E_3(k-s)k^{\delta}.
 $
\end{remark}
Define $\mathcal D_{j}= L_{m_j}^{m_{j+1}-m_j}$, $m_0=0$. Then with $ \varepsilon=\bar \varepsilon(1+\eta_0)$ we have
$$
\|\mathcal D_j^n-\mu_{m_{j}}\|_\alpha\leq E_0'(1+j^\varepsilon)n^{-\beta}
$$
where $E_0'=E_0D(1+E_3^\varepsilon)$.
Thus, the sequence of operators $(\mathcal Y_j)_{j\geq 0}=(\mathcal D_j)_{j\geq 0}$ satisfies \eqref{strong}.

Next we take  functions $f_j$ such that $\mu_j(f_j)=0$ and $\|f_\ell\|_\alpha\leq E_4\ell^a, \ell\geq 1$  for some $a\geq 0$.
Let $S_{j,n}f=\sum_{k=0}^{n-1}f_{j+k}\circ T_{j}^k$ and 
$
S_n f=S_{0,n}f.
$
We view $S_n f$ as random variables on the probability space $(\mathcal X_0, \text{Borel}, \mu_0)$.

Let $B_j$ be the blocks of integers like in Section \ref{Sec1} with the operators $\mathcal Y_j=\cD_j$. 
Set $\ell_j=m_{N_j}$. 
We will constantly use the following result, which follows directly from the Lemma \ref{N lemma} and basic properties of sequences of the form $s_n=\max\{k: a_k\leq n\}$ where $(a_k)$ is a strictly increasing integer valued sequence.
\begin{lemma}\label{L9}
(i) Let $L_n=\max\{k: \ell_k\leq n\}$. Then  $L_{k+1}-L_k\leq 2$ for all $k$. Moreover, for all $n$,
$$
Q_1 n^{\frac{\beta-\varepsilon}{\beta(1+\eta_0)}}-1\leq L_n\leq Q_2 n^{\frac{\beta-\varepsilon}{\beta}}
$$
and $\ell_{j+1}-\ell_j\leq Q_3j^{\frac{\varepsilon+\delta\beta}{\beta-\varepsilon}}$
where 
$
Q_1=\frac12 E_2^{-\frac{1}{1+\eta_0}}(A_2)^{-\frac{\beta-\varepsilon}{\beta}},
Q_2=(A_1)^{-\frac{\beta-\varepsilon}{\beta}},
$
and $Q_3=2^{{\frac{\varepsilon+\delta\beta}{\beta-\varepsilon}}}E_3A_2'A_2^\delta$.
Consequently,
$$
n-\ell_{L_n}\leq \ell_{L_{n}+1}-\ell_{L_n}\leq D_3 n^{\frac{\varepsilon+\delta\beta}{\beta}}
$$
where $D_3=Q_3(Q_2)^{\frac{\varepsilon+\delta\beta}{\beta-\varepsilon}}$.
\vskip0.2cm

(ii) We have 
$
\|S_nf-S_{\ell_{L_n}}f\|_{\infty}\leq D_4n^{\eta}
$
where $\eta=\frac{a\beta+\varepsilon+\delta\beta}{\beta}$ and $D_4=D_3E_4$.
\end{lemma}

Define 
$
U_{j}=\sum_{\ell_j\leq k<\ell_{j+1}}f_{k}\circ T_{\ell_j}^{k-\ell_j}.
$
Using Remark \ref{Large diff rem}, Lemma \ref{N lemma} and Lemma \ref{L9} we have 
\begin{equation}\label{U sup}
 \|U_{j}\|_\infty\leq E_4 
\ell_{j+1}^a(\ell_{j+1}-\ell_j)
\leq E_4Q_3(j+1)^{\frac{\varepsilon+\delta\beta}{\beta-\varepsilon}}E_2^aN_{j+1}^{a(1+\eta_0)}\leq E_5 (j+1)^{\frac{\beta(a(1+\eta_0)+\delta)+\varepsilon}{\beta-\varepsilon}}   
\end{equation}
where $E_5=Q_3E_4E_2^aA_2^{a(1+\eta_0)}$.
Moreover, we have 
$$
Q_{\ell_j,\ell_{j+1}-\ell_j}(f)=\sum_{\ell_j\leq s<\ell_{j+1}}v_\alpha(f_s)\prod_{k=s}^{\ell_{j+1}-1}\gamma_{k}^{-\alpha}.
$$
Next, we suppose that there is a constant $E_6>0$ such that for all $s$ and $d$ we have
\begin{eqnarray}\label{gamma dec}
&\prod_{k=s}^{s+d-1}\gamma_{k}^{-\alpha}\leq E_6s^{\kappa}d^{-a_0}     
\end{eqnarray}
where $0<\kappa<1$ and $a_0>1$. Applying this and the above estimate on $Q_{\ell_j,\ell_{j+1}-\ell_j}(f)$  we get 
$$
Q_{\ell_j,\ell_{j+1}-\ell_j}(f)\leq E_4E_6\sum_{\ell_j\leq s<\ell_{j+1}}s^{\theta+a}(\ell_{j+1}-1-s)^{-a_0}\leq C'E_4E_6 
(\ell_{j+1})^{\kappa+a}\leq 
C'E_7(j+1)^{\frac{(a+\kappa)\beta(1+\eta_0)}{\beta-\varepsilon}}
$$
where $C'\geq 1 $ is a constant depending only on $\kappa,\varepsilon,\beta$ and $E_7=E_4E_6 E_2^{\kappa+a}A_2^{(\kappa+a)(1+\eta_0)}$.
Similarly, if $\|\phi_\ell\|_\alpha\leq E_4\ell^a,\ell\geq1$ then 
$
Q_{\ell_j,\ell_{j+1}-\ell_j}(\phi)\leq C'E_7(j+1)^{\frac{(a+\kappa)\beta(1+\eta_0)}{\beta-\varepsilon}}.
$


Next we suppose that $\xi_j^{-1}\leq E_8j^{a_1}, j\geq 1$.
Then we can to use the perturbative argument via Lemma \ref{lemma TO dist}.
Indeed, under \eqref{gamma dec}, applying Lemma \ref{lemma TO dist} and the above estimates and taking into account that $\gamma_j\geq 1$, with $\zeta_1=\frac{\beta(a(1+\eta_0)+\delta)+\varepsilon}{\beta-\varepsilon}$ and $\zeta_2=\frac{(a+\kappa)\beta(1+\eta_0)}{\beta-\varepsilon}$  we get 
\begin{equation}\label{PertNew}
\|\mathcal L_{\ell_j,z}^{\ell_{j+1}-\ell_j}-\mathcal L_{\ell_j}^{\ell_{j+1}-\ell_j}\|_\alpha\leq E_9|z|e^{E_5|z|j^{\zeta_1}}(j+1)^{\zeta_1+\zeta_2}    
\end{equation}
where 
$$
E_9=6\cdot 2^{\frac{a_1\beta}{\beta-\varepsilon}}E_8E_2A_2(E_5+C''E_5E_7+E_7).
$$
Using this estimate we have the following result.

Fix some $J\in \mathbb N$ and define 
$
\mathcal A_{J,j,z}=\mathcal L_{\ell_j,z}^{\ell_{j+1}-\ell_j},\, 0\leq j\leq J.
$
For $j>J$ we define $\mathcal A_{J,j,z}=\mathcal A_{j}:=\mathcal L_{\ell_j}^{\ell_{j+1}-\ell_j}$. Note that by construction we have 
$$
\left\|\mathcal A_j^n-\mu_{\ell_j}\right\|_{\alpha}\leq 
\varepsilon_0^{n}
$$
where 
$
\mathcal A_j^n=\mathcal A_{j+n-1}\circ\cdots\circ \mathcal A_{j+1}\circ\mathcal A_j.
$
Next, we need the following result.
\begin{proposition}\label{RPF PROP}
If $0<\varepsilon_0<0.17$ 
there are constants $\epsilon_0>0$ and $R_0$ such that 
for every $J\in\bbN$ on the domain
$V_J=\{z\in\mathbb C: |z|\leq \epsilon_0 C_{\varepsilon_0}^{-1}\min(E_5^{-1},E_9^{-1})(J+1)^{-\zeta_1-\zeta_2}\}$ we have the following: 
There are analytic in $z$ and uniformly bounded in $j$ and $J$ triplets $\lambda_{J,j}(z)\in\mathbb C\setminus 0, h_{J,j}^{(z)}\in \mathcal B_{\ell_j}, \nu_{J,j}^{(z)}\in \mathcal B_{\ell_j}^*$ such that: 
\,
\vskip0.1cm
(i) $\lambda_{J,j}(0)=1, h_{J,j}^{(0)}\equiv 1,\nu_{J,j}^{(0)}=\mu_{\ell_j}$  and $\nu_{J,j}^{(z)}(1)=\nu_{J,j}^{(z)}(h_{J,j}^{(z)})=1$;
 \vskip0.1cm   
(ii) $\mathcal A_{J,j,z}h_{J,j}^{(z)}=\lambda_{J,j}(z)h_{J,j+1}^{(z)}$ and $\mathcal (A_{J,j,z})^*\nu_{J,j+1}^{(z)}=\lambda_{J,j}(z)\nu_{J,j}^{(z)}$;
\vskip0.1cm
(iii) For all $j,n$ and $z\in V_J$ we have 
$$
\left\|(\lambda_{J,j,n}(z))^{-1}\mathcal A_{J,j,z}^n-\nu_{J,j}^{(z)}\otimes h_{J,j+n}^{(z)}\right\|\leq R_0(2\varepsilon_0)^{n}
$$
where $\mathcal A_{J,j,z}^n=\mathcal A_{J,j+n-1,z}\circ\ldots\circ \mathcal A_{J,j+1,z}\circ\mathcal A_{J,j,z}$, $\lambda_{J,j,n}(z)=\prod_{k=j}^{j+n-1}\lambda_{J,k}(z)$ and $\nu_{J,j}^{(z)}\otimes h_{J,j+n}^{(z)}$ is the projection operator $g\to \nu_{J,j}^{(z)}(g)h_{J,j+n}^{(z)}$.

\end{proposition}
\begin{remark}
    By \cite[Theorem D.2]{DolHaf}, taking into account \eqref{PertNew} the proposition follows with $\varepsilon_0$ that depends on the sequence $\mathcal A_j$ (which in our application to random dynamical will depend on $\omega$). The delicate part is to to show that we can choose $\varepsilon_0$ and $R_0$ independently of the sequence.
\end{remark}

\begin{proof}
For every $0<\kappa<1$ define real cones 
$$
\mathcal C_{j,\kappa}=\{g:\mathcal X_j\to [0,\infty):  v_\alpha(g)\leq \kappa \inf g\}.
$$
Let us take $0<\kappa<1$ that maximizes $f(\kappa)=\frac{\kappa}{(2\kappa+1)(\kappa+1)}$. Actually $\kappa=\frac{1}{\sqrt 2}$, but this is not very important here.
Let $\varepsilon_0<\frac{\kappa}{(2\kappa+1)(\kappa+1)}$ (any $\varepsilon_0<0.17$ works). We claim  that  for all $j$,
\begin{equation}\label{Cone inv}
\mathcal A_j\mathcal C_{j,\kappa}\subset \mathcal\cC_{j+1,\kappa\zeta}    
\end{equation}
where
$
\zeta=\zeta_\kappa=\frac{\varepsilon_0}{\kappa((2\kappa+1)^{-1}-\varepsilon_0)}\in(0,1).
$
To prove \eqref{Cone inv} we first note that for every $g\in \mathcal C_{j,\kappa}$ with $\|g\|_\alpha=1$ we have $\inf g\geq \frac{1}{2\kappa+1}$. Indeed, since $\text{diam}(\mathcal X_j)\leq 1$,
$$
1=v_\alpha(g)+\|g\|_\infty\leq \kappa \inf g+v_\alpha(g)+\inf g\leq (2\kappa+1)\inf g.
$$
Therefore,
$$
\mathcal A_j g\geq \mu_{\ell_j}(g)-\varepsilon_0\geq \inf g-\varepsilon_0\geq (2\kappa+1)^{-1}-\varepsilon_0.
$$
On the other hand,
$$
v_\alpha(\mathcal A_j g)=v_\alpha(\mathcal A_j g-\mu_{\ell_j}(g))\leq \varepsilon_0=\zeta \kappa \left((2\kappa+1)^{-1}-\varepsilon_0\right)\leq \zeta\kappa \inf \mathcal A_j g.
$$
This proves \eqref{Cone inv}.

By \cite[Propsition 5.2]{castro}  (see the proof of \cite[Propsition 4.3]{castro}), the projective diameter of $\mathcal A_j(\mathcal C_{j,\kappa})$ inside $\mathcal C_{j+1,\kappa}$ does not exceed a constant $C(\kappa,\zeta)>0$. Relying on \eqref{PertNew} 
the rest of the proof is based of the theory of complex cones and it is similar to the proof of  \cite[Theorem 3.2]{Haf 19}.
\end{proof}


We also need the following result. 
\begin{lemma}\label{L15}
For every $p>2$ there is a constant $C=C_{p,\beta}>0$ that depends only on $p$ and $\beta$ such that for every $j$ and $n$,
$$
\|S_{j,n}f-\mu_j(S_{j,n}f)\|_{L^p}\leq  C E_{12}(j+n)^{\zeta}n^{1/2}
$$
where $\zeta=\max\left(\frac{\beta(\varepsilon+a)(1+\eta_0)}{\beta-\varepsilon},\frac{\beta(a(1+\eta_0)+\delta)+\varepsilon}{\beta-\varepsilon}\right)$
and $E_{12}=E_4 Q_3Q_2^{\frac{\varepsilon+\delta\beta}{\beta-\varepsilon}}+Q_2^\zeta(E_5+E_{10})+E_{10}(E_{11}Q_2)^{\frac{\beta(\varepsilon+a)(1+\eta_0)}{\beta-\varepsilon}}$, with
$E_{10}=E_0 D E_4 E_2^{(a+\varepsilon)(1+\eta_0)}A_2^{(a+\varepsilon)(1+\eta_0)}$ 
and
$E_{11}=Q_2(E_2A_2)^{\frac{\beta-\varepsilon}{\beta}}$. 
\end{lemma}

\begin{proof}
To simplify the proof we suppose $\mu_k(f_k)=0$ for all $k$.
As before, let $U_k=\sum_{\ell_k\leq s<\ell_{k+1}}f_s\circ T_{\ell_k}^{s-\ell_k}$ and
$$
u_k=\sum_{s=1}^k\mathcal A_{k-s}^{s}U_{k-s}=\sum_{s=1}^k\sum_{\ell_{k-s}\leq m<\ell_{k-s+1}}\mathcal L_{m}^{\ell_{k}-m}f_m.
$$
Then 
$$
\|u_k\|_\alpha \leq E_0D\sum_{s=1}^k\sum_{\ell_{k-s}\leq m<\ell_{k-s+1}}\|f_m\|_\alpha(1+m^{\varepsilon})(\ell_{k}-m)^{-\beta}\leq E_0 D E_4\ell_{k}^{a+\varepsilon}\sum_{s=1}^k\sum_{\ell_{k-s}\leq m<\ell_{k-s+1}}(\ell_{k}-m)^{-\beta}
$$
$$
\leq E_0 D E_4C_\beta E_2^{a+\varepsilon}(N_k)^{(a+\varepsilon)(1+\eta_0)}=E_{10}k^{\frac{\beta(a+\varepsilon)(1+\eta_0)}{\beta-\varepsilon}}
$$
where $C_\beta=\sum_{k\geq 1}k^{-\beta}<\infty$, $E_{10}=E_0 D E_4 E_2^{a+\varepsilon}A_2^{a+\varepsilon}$ and we used Lemma \ref{N lemma} and that $\|f_k\|_\alpha \leq E_4k^a$.

Now, we can write 
$$
U_k=u_{k+1}\circ T_{\ell_k}^{\ell_{k+1}-\ell_k}-u_k+M_k
$$
with $M_k\circ T_0^{\ell_k}$ being a reverse martingale difference. Note that by the above estimate and \eqref{U sup}, 
$$
\|M_k\|_\infty\leq E_{10}k^{\frac{\beta(\varepsilon+a)(1+\eta_0)}{\beta-\varepsilon}}+2E_5(k+2)^{\frac{\beta(a(1+\eta_0)+\delta)+\varepsilon}{\beta-\varepsilon}}.
$$
By applying the Burkholder inequality we see that there is a constant $C_p>0$ which depends only on $p$ such that for all $J,N$, 
$$
\|S_{J,N} M\|_{L^p}^2\leq C_p\left\|\sum_{k=0}^{N-1}(M_{J+k}\circ T_{\ell _J}^{\ell_{k+J}-\ell_J})^2\right\|_{L^{p/2}}\leq C_p\sum_{k=0}^{N-1}\|M_{J+k}\|_{L^\infty}^2\leq C_p(E_5+E_{10})^2(J+N)^{2\zeta}N.
$$
Thus, 
$$
\|S_{J,N} M\|_{L^p}\leq \sqrt{C_p}(E_5+E_{10})(J+N)^{\zeta}N^{1/2}. 
$$
Using that $\|u_k\|_\infty \leq E_{10}k^{\frac{\beta(\varepsilon+a)(1+\eta_0)}{\beta-\varepsilon}}$ conclude that
\begin{equation}\label{U up}
\|S_{J,N} U\|_{L^p}\leq \sqrt{C_p}(E_5+E_{10})(J+N)^{\zeta}N^{1/2}+2E_{10}(J+N)^{\frac{\beta(\varepsilon+a)(1+\eta_0)}{\beta-\varepsilon}}.    
\end{equation}

Finally, using that $\|f_k\|\leq E_4k^a$ and Lemma \ref{L9} we get that for all $m$,
$$
\|S_{m}f-S_{\ell_{L_m}}f\|_\infty\leq \sum_{k=\ell_{L_m}}^{m-1}\|f_k\|_\infty\leq E_4(m-\ell_{L_m})m^{a}\leq E_4 Q_3Q_2^{\frac{\varepsilon+\delta\beta}{\beta-\varepsilon}} m^{a+\frac{\varepsilon+\delta\beta}{\beta}}.
$$
Thus, using that $S_{L_k}U=S_{\ell_{L_k}}f$,
$$
\|S_{j,n}f-S_{\ell_{L_j},\ell_{L_{n+j}}-\ell_{L_{j}}}f\|_{\infty}=\|(S_{j+n}f-S_{\ell_{L_{j+n}}}U)-(S_{j}f-S_{\ell_{L_j}}U)\|_\infty\leq 2 E_4 Q_3Q_2^{\frac{\varepsilon+\delta\beta}{\beta-\varepsilon}} (j+n)^{a+\frac{\varepsilon+\delta\beta}{\beta}}.
$$
Using again that $S_{L_k}U=S_{\ell_{L_k}}f$ for all $k$ and using 
 the above estimates and that 
 \begin{eqnarray*}
   &  \ell_{L_k}=m_{N_{L_k}}\leq E_2(N_{L_k})^{1+\eta_0}\leq E_{11}k^{1+\eta_0}, E_{11}=E_2A_2^{1+\eta_0}Q_2^{\frac{\beta(1+\eta_0)}{\beta-\varepsilon}}
 \end{eqnarray*}
 (which follows from Lemma \ref{N lemma}, Lemma \ref{L9} and the definition of $\ell_m$) and taking $J=L_j$ and $N=L_{n+j}-L_j$ we conclude that 
 \begin{eqnarray*}
   & 
\|S_{j,n}f\|_{L^p}\leq  2E_4 Q_3Q_2^{\frac{\varepsilon+\delta\beta}{\beta-\varepsilon}} (j+n)^{a+\frac{\varepsilon+\delta\beta}{\beta}}+ \sqrt{C_p}(E_5+E_{10})(L_{j+n})^{\zeta}(L_{n+j}-L_j)^{1/2}
\\
&+2E_{10}E_{11}^{\frac{\beta(\varepsilon+a)(1+\eta_0)}{\beta-\varepsilon}}(L_{n+j})^{\frac{\beta(\varepsilon+a)(1+\eta_0)}{\beta-\varepsilon}}.
\end{eqnarray*}
Now the result follows using that $L_{n+j}-L_j\leq\sum_{k=j}^{j+n-1}(L_{k+1}-L_k)\leq 2n$ and  $L_k\leq Q_2k^{\frac{\beta-\varepsilon}{\beta}}$.
\end{proof}

Next, let $q\in\mathbb N$ be the minimal index $q$ such that 
$
\|S_n f\|_{L^2}^2\geq \frac12\Sigma^2 n, \Sigma>0
$
for all $n\geq q$.

\begin{cor}\label{cor11}
Suppose that $\eta<1/2$, where $\eta$ is as in Lemma \ref{L9}. Assume also that  $\mu_k(f_k)=0$ for all $k$.
Then for every $n\geq q$ such that $D_4n^{\eta}<\frac14\Sigma \sqrt n$ we have
\begin{equation}\label{LB Y}
 \|S_{\ell_{L_n}}f\|_{L^2}^2\geq \frac1{16} \Sigma^2 n.   
\end{equation}
Moreover, for every $p>2$ there is a constant
$C=C_{\Sigma,p,\beta}$ such that 
$$
\left\|S_n f/\|S_nf\|_{L^2}-S_{\ell_{L_n}}f/\|S_{\ell_{L_n}}f\|_{L^2}\right\|_{L^p}\leq CD_4 n^{\eta+\zeta-1/2}.
$$
\end{cor}
\begin{proof}
 Denote $X=S_n f$ and $Y=S_{\ell_{L_n}} f$. Then
by Lemma \ref{L9}(ii) we have
$
\|X-Y\|_{L^\infty}\leq D_4n^{\eta}
$
and so $\|Y\|_{L^2}\geq \|X\|_{L^2}-\|X-Y\|_\infty$.
Now the lower bound  follows from the definition of $q$. 

 Next, we prove the $L^p$ bounds. Let us write
 $$
Y/\|Y\|_{L^2}-X/\|X\|_{L^2}=\frac{Y-X}{\|Y\|_{L^2}}+\frac{(\|X\|_{L^2}-\|Y\|_{L^2})X}{\|X\|_{L^2}\|Y\|_{L^2}}:=I_1+I_2.
 $$
 Using Lemma \ref{L9} (ii) we get that 
 $
\|X-Y\|_{L^2}\leq \|X-Y\|_\infty\leq D_4n^\eta
 $
 and so by \eqref{LB Y},
 $
\|I_1\|_{L^\infty}\leq 4\Sigma^{-1}D_4n^{\eta-1/2}.
 $
Next, note that 
$$
\|I_2\|_{L^p}\leq \|X-Y\|_{L^\infty}\|X\|_{L^p}(\|X\|_{L^2}\|Y\|_{L^2})^{-1}.
$$
Using the lower bounds of $\|X\|_{L^2}$ and $\|Y\|_{L^2}$  and the previous lemma with $j=0$  we see that there is a constant $C=C_{\Sigma,p,\beta}$ such that 
$
\|I_2\|_{L^p}\leq CD_4n^{\eta+\zeta-1/2}.
$ 
\end{proof}

\section{CLT rates and MDP:  proof of Theorems \ref{BE1} and \ref{MDP}}
\subsection{CLT and MDP in the sequential notations}\label{Secc}
First, by \cite[Lemma 3.3]{HK SPA} and Corollary \ref{cor11} for every $p>1$,
\begin{equation}\label{approx00}
\sup_{t\in\mathbb R}\left|\mu_0(\bar S_nf\leq \|\bar S_nf\|_{L^2}t)-\Phi(t)\right|
\leq 3\sup_{t\in\mathbb R}\left|\mu_0(S_{L_n}f\leq \|S_{\ell_{L_n}}f\|_{L^2}t)-\Phi(t)\right|+5\left(CD_4n^{\eta+\zeta-1/2}\right)^{\frac{p}{p+1}} 
\end{equation}
where  $S_nf=S_n^\om f$ and $\bar S_n f=S_nf-\mu_0(S_nf)$. Henceforth we assume that $\mu_k(f_k)=0$ for all $k$, so that $\bar S_n f=S_nf$.
Thus, it is enough to obtain CLT rates for $S_{\ell_{L_n}}f/\|S_{\ell_{L_n}}f\|_{L^2}$.
This will be done using characteristic functions. We first recall that by Esseen's inequality there exists $C>0$ such that for every random variable $\hat S$ and all $T>0$,
\begin{equation}\label{Esseen}
\sup_{t\in\mathbb R}|\mathbb P(\hat S\leq t)-\Phi(t)|\leq\int_{-T}^{T}|t|^{-1}|\mathbb E[e^{itS}]-e^{-\frac12 t^2}|dt+C/T.    
\end{equation}

Next, by Proposition \ref{RPF PROP} for every $z\in V_L=\{\zeta\in\mathbb C: |\zeta|\leq \delta_L\},\, \delta_L=\epsilon_0 C_{\varepsilon_0}^{-1}\min(E_5^{-1},E_9^{-1})L^{-\zeta_1-\zeta_2}$,
$$
\mathbb E[e^{zS_{\ell_L}f}]=\mu_{\ell_L}(\mathcal A_{j,z}^{L}\textbf{1})=\lambda_{L,1,L}^{(z)}\left(\mu_{\ell_L}(h_{L,L}^{(z)})+U_{L}(z)\right)
$$
where $|U_L(z)|\leq |z|R_0 (2\varepsilon_0)^L$ (the factor $z$ appears since $U_L(0)=0$ and $U_L$ is analytic in $z$). Note also that $|\mu_{\ell_L}(h_{L,L}^{(z)})-1|\leq C_\varepsilon|z|$. Moreover, since $\lambda_{L,j}(0)=0$ and $\lambda_{L,j}(z)$ is uniformly bounded we have $|\lambda_{L,j}(z)-1|\leq C_\varepsilon|z|$.

Now, by possibly decreasing $\epsilon_0$ if needed we can develop  branches of $\ln\lambda_{L,j}(z)$ and $\ln(\mu_{\ell_L}(h_{L,L}^{(z)})+ U_L(z))$ in the above domain we have 
$$
\Delta_{L}(z):=\ln \mathbb E[e^{zS_{\ell_L}f}]=\Pi_L(z)+R_{L}(z)
$$
where $\Pi_L(z)=$
$\sum_{j=0}^{L-1}\ln \lambda_{L,j}(z)$ and $|R_L(z)|\leq C_\varepsilon|z|$. Now, using the Cauchy integral formula there is a constants $c_\varepsilon$ such that if $|z|\leq \frac12\delta_L$ then for all $k=1,2,3$,
$$
\left|\Delta_{L}^{(k)}(z)-\Pi_{L}^{(k)}(z)\right|\leq (\delta_L)^{-k}c_{\varepsilon_0}
$$
where $g^{(k)}$ denotes the $k$-th derivative of a function $g$.

Next, since $|\Pi_L'''(z)|\leq c'_{\varepsilon_0} L$ we have
$$
\sup_{|z|\leq \frac12\delta_L}|\Delta_{L}^{'''}(z)|\leq (\delta_L)^{-3}c_{\varepsilon_0}+c'_{\varepsilon_0} L.
$$
This yields that for $|z|\leq \frac12\delta_L$,
$$
\left|\Delta_L(z)-\frac12 z^2 \hat \Sigma_L^2\right|\leq \left((\delta_L)^{-3}c_{\varepsilon_0}+c'_{\varepsilon_0} L\right)|z|^3\leq c''_{\varepsilon_0} E_{12}L^{\zeta_3}|z|^3, 
$$
 where $\hat \Sigma_L=\|S_{\ell_{L}}f\|_{L^2}=\Sigma_{\ell_L}$, $E_{12}=(1+\max(E_5^{3}, E_9^{3}))$ and $\zeta_3=\max(1, 3\zeta_1+3\zeta_2)$. Note that $|\hat \Sigma_{L_n}|\geq \Sigma_n-D_4n^{\eta}$.

Next, we assume that $\hat \Sigma_L^2\geq c_0 L^{\frac{\beta}{\beta-\varepsilon}}$ for some constant $c_0$ and that $z$ also satisfies 
$4c_{\varepsilon_0}''|z|E_{12}\leq \frac14 c_0L^{\frac{\beta}{\beta-\varepsilon}-\zeta_3}$. Then if $z=it$ for a real $t$ we have 
$$
|\Delta_L(it)|\leq e^{-c_1t^2\Sigma_L^2}
$$
for some constant $c_1=c_1(c_0)$ and so, by the mean value theorem, 
$$
\left|e^{\Delta_L(it)}-e^{-\frac12t^2\hat \Sigma_L^2}\right|\leq c_{\varepsilon_0}''E_{12}L^{\zeta_3} |t|^3 e^{-c_1t^2\Sigma_L^2}.
$$
 
Using these estimates we conclude that
$$
\int_{|t|\leq \rho_L\Sigma_L }\left|\frac{e^{\Delta_L(it/\hat \Sigma_L)}-e^{-t^2/2}}{t}\right|dt\leq A(\varepsilon_0,c_0) E_{12} L^{-v_0},\,\,v_0=\frac{3\beta}{2(\beta-\varepsilon)}-\zeta_3
$$
where $\rho_L=\min(\frac12\delta_L, \frac{c_0}{4c_{\varepsilon_0}'' E_{12}}L^{\frac{\beta}{\beta-\varepsilon}-\zeta_3})$.
Now  by taking $L=L_n$ and assuming $\eta<1/2$ we  get the CLT rates $O(\min(L_n^{-v_0}, \rho_{L_n}^{-1}\Sigma_{L_{\ell_n}}^{-1}))$ for  $S_{\ell_{L_n}}f$ by \eqref{Esseen}. Using \eqref{approx00} we get rates for $S_n=S_n^\om f$. They can be expressed by means of $n$ using Lemma \ref{L9}.

To prove  MDP, take a sequence $a_L\to\infty$ such that $\frac{1}{a_L\hat \Sigma_L}=o(\delta_L)$ and for  $L$ large enough  expand 
$$
\Lambda_L(\frac{t}{a_L\hat \Sigma_L})=\frac12t^2 a_L^{-2}+O(|t|^3L^{\zeta_3}a_L^{-3}\hat \Sigma_L^{-3}).
$$
Thus, for every real $t$ we  have
$$
\lim_{n\to\infty}\frac{1}{a_{L_n}^2}\Lambda_L(\frac{t}{a_{L_n}\hat \Sigma_{L_n}})=\frac12 t^2.
$$
Finally, using that $\|S_n-S_{\ell_{L_n}}\|_{\infty}\leq D_4n^\eta$ we have 
$$
e^{-|t|D_4n^\eta}\mathbb E[e^{tS_{\ell_{L_n}}}]\leq \mathbb E[e^{tS_n}]\leq e^{|t|D_4n^\eta}\mathbb E[e^{tS_{\ell_{L_n}}}].
$$
It follows that  if $n^\eta=o(a_{L_n}^2)$ then
$$
\frac 12 t^2=\lim_{n\to\infty}\frac{1}{a_{L_n}^2}\Lambda_L(\frac{t}{a_{L_n}\Sigma_{L_n}})=\lim_{n\to\infty}\frac{1}{a_{L_n}^2}\ln(\mathbb E[e^{tS_n}]).
$$
Thus, the MDP with speed $a_{L_n}^2$ for $(a_{L_n}\Sigma_{L_n})^{-1}S_n$ follows from the G\"artner Ellis theorem.

\subsection{Completion of the proof of Theorems \ref{BE1} and \ref{MDP}}
When taking $L_j=L_{\te^j\om}, \phi_{j}=\phi_{\te^j\om}$ and $f_{j}=f_{\te^j\om}$
the constants $A_i,D_i, E_i,Q_i$ in the sequential notations are random variables, and in what follows we will show that they all belong to $L^{u(p)}$ with $\lim_{p\to\infty}u(p)=\infty$.  This is enough to prove Theorems \ref{BE1} and \ref{MDP} given the results in Section \ref{Secc} and Lemma \ref{L9}.

To find $u(p)$ as above, first by Theorem \ref{RPF Poly} (ii) we can take $\beta=\beta(p)$ in \eqref{stronger} with $L_j=L_{\te^j\om}$ and $R(j)=R_p(\te^j\om)$.  Moreover, since $R_p\in L^{\beta(p)}$ in Theorem \ref{RPF Poly} we have $R_p(\te^j\om)=O(j^{1/p})$ and so we can take $\bar\varepsilon=1/p$ in \eqref{stronger}. 
Next, if $m_k(\omega)$ denotes the $k$-th visit to $\{\om: R_p(\om)\leq D\}$ for some $D$ large enough then $m_k=O(k)$ (a.s) and so we can take $\eta_0=0$ right after \eqref{stronger}. Now since $m_1$ is in $L^{\beta(p)}$ we have $m_1(\sigma^j\omega)=O(j^{1/p})$. Since
$
m_{k+1}(\om)-m_k(\om)=m_1(\sigma^{m_k(\om)}\om),
$
$$
m_{k+1}(\om)-m_k(\om)=O((m_k(\om))^{1/p})=o(k^{1/p})
$$
so in the non-effective case we can take $\delta=1/p$.

Now note that $\|f_{\te^j\om}\|_\al=o(j^{a}),a=1/p$ since $\|f_\om\|\in L^{p}$. Notice also that by  Lemmata \ref{v tilde phi} and \ref{tilde H approx} we have $v_{\al,\xi_\om}(\tilde\phi_{\om})\in L^{\tilde p}$. Moreover, by 
\eqref{Lam Bound} and Lemma \ref{h Bound lemm},
$$
\|\tilde\phi_{\om}\|_{\infty}\leq \|\phi_\om\|_\infty(1+\ln(D_\om))+A(\om)
$$
where $A(\om)=Q_{\te^{-j_\om}\om}+\sum_{k=1}^{j_\om}(\ln D_{\te^{-k}\om}+\|\phi_{\te^{-k}\om}\|_\infty)$.
Arguing like in the proof of Lemma \ref{Mom Lemma Final}, we have
$
\|A(\om)\|_{L^t}<\infty.
$
Using also Lemma \ref{Norm comp} we get that $\|\tilde\phi_{\om}\|_{\al}\in L^{v(p)}$ with $\lim_{p\to\infty}v(p)=\infty$. Note also that for both $C^2$ maps and random SFT we have $\xi_\om^{-1}\in L^{e(p)}$ with $\lim_{p\to\infty}e(p)=\infty$.

Finally, to get \eqref{gamma dec} with $\gamma_j=\gamma_{\te^j\om}$ we define $\Gamma_{\om,d}=\prod_{k=0}^{d-1}\gamma_{\te^k\om}^{-\alpha}$. Then by either Lemma \ref{alpha exp lemm} or Lemma \ref{g exp lemm} for every $s$ we have $\bbE[\Gamma_{\om,d}^s]\leq C_sd^{-1-\frac12 p}$. Taking $s=\sqrt p$ and using Lemma \ref{Simple} we see that $\Gamma_{\om,d}\leq E(\om)d^{-\sqrt p/2}$ with $E\in L^{\sqrt p/2}$. Now, as a consequence of the mean ergodic theorem $E(\te^j\om)=o(j^{2/\sqrt p})$ and so in \eqref{gamma dec} we can take $\kappa=2/\sqrt p$ and $a_0=\sqrt p/2$. Using the definitions of the random variables $A_i,D_i, E_i,Q_i$ and the above estimates we get that they all belong to $L^{u(p)}$ as above, and the proof of Theorems \ref{BE1} and \ref{MDP} is completed.
\qed

\section{Effective CLT rates: proof of Theorem \ref{BE2}}
\subsection{Effective convergence rates towards the asymptotic variance}
First, we note the  arguments in proof of Theorem \ref{OSC1} show that for $\bbP$-a.a. $\om$ and all 
$n\geq n(\om):=M_0m(\om)$,
\begin{equation}\label{RPF 1}
 \|\mathcal L_{\om}^n\textbf{1}-h_{\sigma^n\om}\|_\infty\leq V_\om R(\om)n^{-\beta}   
 \end{equation}
with $R\in L^{t}$ and $h_\om=d\mu_\om/dm$ (recall that $\nu_\om=m=\text{Vol}$). Indeed,  the projective diameter estimates hold also for the non-normalized potential $\phi_\om$, and the functions $\textbf{1}$ and $h_\om$ belong to the cone $\cC_\om$, so we can apply \cite[Lemma 3.5]{Kifer Thermo} with these functions and the measure $m$.   Since Assumption \ref{Poly Ass} holds $m(\om)\in L^{q_0}$, where $q_0=\beta d+1$. 

We assume here that $R_3(\om)=V_\om R(\om)\in L^{p_0}$ for some $p_0>1$ and that $T_\om$ (and so $\mathcal L_\om$) and $f_\om$ depend only on $\om_0$. Moreover,  suppose  $\|f_\om\|_\alpha\in L^{p_1}, p_1>4$  and for some $1<v<p_1/4$,
\begin{eqnarray}\label{mix 1}
 &   \sum_{m\geq 0}(\al(m))^{1/v-4/p_1}<\infty. 
\end{eqnarray}
\begin{remark}\label{Remm}
Let
$
\varpi_{q,p}(m)=\sup\{\|\bbE[h|X_0,X_{-1},...]-\bbE[h]\|_{L^p}: h\in L^q(\cF_{m,\infty}), \|h\|_{L^q}=1\}
$
where $\cF_{m,\infty}=\sigma\{X_j: j\geq m\}$.
By \cite[(2.25)]{KV}, \, $\varpi_{q,p}(m)\leq 2(\alpha(m))^{1/p-1/q}, q\geq p$. Thus  $\sum_{m\geq 0}\varpi_{p_1/4,v}(m)<\infty$. This is the actual mixing condition needed in what follows.
\end{remark}


Next, let $v'$ be given by $\frac1{v'}=\frac{1}{2v}+\frac1{p_1}$ and
$u^*$ be given by $\frac{1}{u^*}=\frac{1}{v'}+\frac{3q_0-2}{2(q_0^2-q_0)}$. Set let  $u=\min(u^*, a_1^*)$, where $a_1^*$ is given by $\frac1{a_1}= \frac{1}q+\frac2{p_1}+\frac1{p_0}$.
Set also $V=\min(\frac{\beta}{q_0}, \frac12-\frac1{2q_0})$.
Let $p^*=\min(u,v',u_5,u_6,b_1)$ where, 
$u_5$ is defined by $\frac{1}{u_5}=\frac{1}{p_0}+\frac{1}{v'}$, $u_6$ is defined by $\frac{1}{u_6}=\frac{1}{p_0}+\frac{2}{p_1}$ and $b_1$ is defined by $\frac{1}{b_1}=\frac{2}{p_1}+\frac{q_0-2}{2q_0(q_0-1)}$. Let 
$\iota=\min(\frac{V(\beta-1)-1/\bar u}{\beta}, V-\frac{1}u)$, where $\bar u=\min(v',u_5,u_6,b_1).$ Note that if $p_0$, $q_0$ and $\beta$ can be chosen to be arbitrarily large then $p^*$ can be chosen arbitrarily close to $2v$ and $\iota$ can be chosen to be arbitrarily close to $\frac12-\frac1{2v}$.

\begin{theorem}\label{VarRate}
Let
$
\Sigma_{\om,n}^2=\text{Var}_{\mu_\om}(S_n^\om f).
$
If $\iota>0$, then  for all $0<\kappa<\iota$ there is $\bar B_\kappa\in L^{p^*-\kappa}$ such that $\mathbb P$-a.s. 
$$
\left|\frac 1n\Sigma_{\om,n}^2-\Sigma^2\right|\leq \bar B_\kappa(\om)n^{-(\iota-\kappa)}
$$
for all $n\geq 1$,
where with $\bar f_\om=f_\om-\mu_\om(f_\om)$,
$$
\Sigma^2 =\int_\Omega \mu_\om(\bar f_{\om}^2)\, d\mathbb P+2\sum_{n\geq1}\int_\Omega\mu_\om(\bar f_{\om} (\bar f_{\sigma^n \omega}\circ T_{\om}^n))\,  d\mathbb P=\lim_{n\to\infty}\frac 1n \Sigma_{\om,n}^2.
$$
Moreover, when $\beta>2$ then             $\Sigma^2=0$ if and only if $\bar f_\om(x)=u(\om,x)-u(\sigma\om,T_\om x)$ for $\mu$-a.e. $(\om,x)$ and some $u\in L^2(\mu)$. 
\end{theorem}


\begin{corollary}\label{q corr}
Let $A=\iota-\kappa$. Then
in Corollary \ref{cor11}  we can take $q=q(\om)=\Sigma^{-\frac{2}{1-A}}(B_\kappa(\om))^{\frac1{1-A}}\in L^{(p^*-\kappa)(1-A)}$. 
\end{corollary}

We will first  need the following simple result.
\begin{lemma}\label{Growth lemma}
Let $(W_k)$  be a sequence of random variables such that $\|W_k\|_{L^p}\leq Ck^{a}$ for some $c,p,a>0$. Then for every $\varepsilon>1/p$ there exists $A=A_{p,\varepsilon}\in L^p$ such that for every $k\geq1$,
$$
|W_k(\cdot)|\leq A(\cdot) k^{a+\varepsilon}.
$$
\end{lemma}
\begin{proof}
 Define 
 $
A(\cdot)=\sup_{k\geq 1}\left(k^{-a-\varepsilon}|W_k(\cdot)|\right).
 $
 Then $
|W_k|\leq A k^{a+\varepsilon}$. Next, $|A|^p\leq \sum_{n\geq 1}k^{-ap-\varepsilon p}|W_k|^p$ and so 
$
\mathbb E[|A|^p]\leq \sum_{n\geq 1}k^{-ap-\varepsilon p}\|W_k\|_{L^p}^p\leq C^p\sum_{n\geq 1}k^{-\varepsilon p}<\infty.
$
\end{proof}

%
%
%
%

\begin{proof}[Proof of Theorem \ref{VarRate}]
First
\begin{equation}\label{decom}
\Sigma_{\om,n}^2=
\|S_n^\om \bar f\|_{L^2(\mu_\om)}^2=\sum_{j=0}^{n-1}\mu_{\te^j\om}(\bar f_{\te^j\om}^2)+2\sum_{i=0}^{n-1} \sum_{j=i+1}^{n-1}\mu_{\te^i\om}(\bar f_{\te^i\om}(\bar f_{\te^j\om}\circ T_{\te^i\om}^{j-i})):=I_n(\om)+2J_n(\om).    
\end{equation}

Next, since $n(\om)\in L^{q_0}$ by Lemma \ref{Growth lemma} applied with $W_k(\om)=n(\te^k\om)$, $p=q_0$ and $a=0$ we have $n(\te^k\om)\leq C(\om)|k|^{1/q_0+\delta}, k\not=0$ with $C\in L^{q_0}$ and $\delta$ arbitrarily small (one can also assume that $C$ is integer valued and that $C(\om)\geq n(\om)$). Thus, for every $j\geq0$ if $r\geq \max(C(\om)j^{1/q+\delta}, C(\om)^{\frac{1}{1-1/q_0-\delta}})$ then  by \eqref{RPF 1},
\begin{equation}\label{Approx cond}
\|\mathcal L_{\te^{j-r}\om}^{r}\textbf{1}-h_{\te^j\om}\|_\infty\leq R_3(\te^{j-r}\om)r^{-\beta} 
\end{equation}
where for $j>r$ we used that $|j-r|^{1/q_0+\delta}\leq j^{1/q_0+\delta}$ and when $j\leq r$ we used that $r\geq C(\om)r^{1/q_0+\delta}\geq C(\om)|j-r|^{1/q_0+\delta}$.
Let us also set  $\bar f_{\om,r}=f_\om-m( f_\om \mathcal L_{\te^{-r}\omega}^r\textbf{1})$.
To estimate $I_{n}(\om)$  notice that by \eqref{Approx cond} if $r\geq\max( C(\om)n^{1/q_0+\delta}, C(\om)^{\frac1{1-1/q_0-\delta}})$ then for all $j<n$,
$$
|m(\bar f_{\te^j\om}^2h_{\te^j\om})-m(\bar f_{\te^j\om,r}^2\mathcal L_{\te^{j-r}\om}^r\textbf{1})|\leq 
m(\bar f_{\te^j\om}^2|h_{\te^j\om}-\mathcal L_{\te^{j-r}\om}^r\textbf{1}|)+m(|\bar f_{\te^j\om}^2-\bar f_{\te^j\om,r}^2|\mathcal L_{\te^{j-r}\om}^r\textbf{1})
$$
$$
\leq 
2d(\om) R_3(\te^{j-r}\om)r^{-\beta}+\|\bar f_{\te^j\om}^2-\bar f_{\te^j\om,r}^{2}\|_\infty
\leq 6d(\om) R_3(\te^{j-r}\om)r^{-\beta}
$$
where
$d(\om)=\|f_\om\|_\infty^2\in  L^{p_1/2}$.
Thus,
\begin{eqnarray*}
  &  |I_{n}(\om)- I_{n,r}(\om)|\leq \sum_{j=0}^{n-1}d(\te^j\om)R_3(\te^{j-r}\om)r^{-\beta}
,\,\,\text{ where }\,\, I_{n,r}(\om)=\sum_{j=0}^{n-1}m(f_{\te^j\om,r}^2\mathcal L_{\te^{j-r}\om}^r\textbf{1}).
\end{eqnarray*}
Next since $C(\om)\geq 1$ we have $r^{-\beta}\leq n^{-\beta(1/q_0+\delta)}$ and so with $r_n(\om)=[C_1(\om)n^{1/q_0+\delta}]+1$, $C_1(\om)=(C(\om))^{\frac{1}{1-1/q_0-\delta}}$ 
we have
\begin{eqnarray*}
  &
|I_{n}(\om)-I_{n,r_n(\om)}(\om)|
\leq n^{-\beta(1/q+\delta)}\sum_{s\geq 1}\mathbb I(C(\om)=s)\sum_{j=0}^{n-1}d(\te^j\om)R_3(\te^{j-b_{s,n}}\om)
\end{eqnarray*}
where $b_{s,n}=[s^{\frac{1}{1-1/q_0-\delta}}n^{1/q_0+\delta}]+1$.
Now, by the Markov inequality, 
$
\mathbb P(C(\om)=s)\leq As^{-q_0}, \,\,\,A=\mathbb E[C^{q_0}]<\infty.
$
Hence, if $a_1,a_2,a_3$ satisfy $1/a_1=1/a_2+1/a_3$ then by the H\"older inequality,
\begin{eqnarray*}
  &
\|I_{n}(\cdot)-I_{n,r_n(\cdot)}(\cdot)\|_{L^{a_1}}
\leq A'n^{1-\beta(1/q_0+\delta)}\sum_{s\geq 1}s^{-q_0/a_2}\leq A'' n^{1-\beta(1/q_0+\delta)}
\end{eqnarray*}
if $d(\om)R_3(\om)\in L^{a_3}$ and $a_2<q_0$. Thus, we can define $a_3$ according to $1/a_3=2/p_1+1/p_0$.  Since $d(\om)\in L^{p_1/2}$ and $R_3\in L^{p_0}$  taking $a_2$ arbitrarily close to $q$ we see that $1/a_1$ can be chosen arbitrarily close to $\frac{1}{q_0}+\frac2{p_1}+\frac1{p_0}$ (but larger).

Next, we estimate $I_{n,r_n(\om)}(\om)$. Let  $v'$ be given by $2/v'=1/v+2/p_1$.
First, for every $u_0,u_1>0$ such that $1/u_0=1/u_1+1/v'$ we have 
\begin{eqnarray*}
  &
\|I_{n,r_n(\cdot)}(\cdot)-\mathbb E_{\mathbb P}[I_{n,r_n(\cdot)}(\cdot)]\|_{L^{u_0}}\leq\sum_{s=1}^\infty\|\mathbb I(C(\om)=s)(I_{n,b_{s,n}}(\om)-\mathbb E_{\mathbb P}[I_{n,b_{s,n}}])\|_{L^{u_0}}\\
&\leq \sum_{s=1}^{\infty}\left(\mathbb P(C(\cdot)=s)\right)^{1/u_1}\|I_{n,b_{s,n}}-\mathbb E_{\mathbb P}[I_{n,b_{s,n}}]\|_{L^{v'}}
\leq A'\sum_{s=1}^\infty s^{-q_0/u_1}\|I_{n,b_{s,n}}-\mathbb E_{\mathbb P}[I_{n,b_{s,n}}]\|_{L^{v'}}.
\end{eqnarray*}
Now, let $Y_{j,r}(\om)=m(\bar f_{\te^j\om,r}^2\mathcal L_{\te^{j-r}\om}^r\textbf{1})$. Then $|Y_{j,r}(\om)|\leq 4d(\te^j\om)$ and so  $c_0=\sup_{j,r}\|Y_{j,r}\|_{L^{p_1/2}}<\infty$. Moreover, 
since $Y_{j,r}$ is measurable with respect to $\mathcal F_{j-r,r}$ for every $m\geq 0$,
$$
\left\|Y_{j,r}-\mathbb E[Y_{j,r}|\mathcal F_{-\infty,j-r-m}]\right\|_{L^{v}}\leq c_0\varpi_{p_1/2,v}(m).
$$
Therefore, taking into account \eqref{mix 1} and Remark \ref{Remm}, 
by  \cite[Proposition 7]{MPU06},
$$
\left\|I_{n,b_{s,n}}-\mathbb E_{\mathbb P}[I_{n,b_{s,n}}]\right\|_{L^{v'}}\leq C_{v'}(s^{\frac{1}{1-1/q_0-\delta}}n^{1/q_0+\delta})^{1/2}n^{1/2}.
$$
Taking $u_1<\frac{2q_0-2-\delta q_0}{3-2/q_0-2\delta}$  we have 
$
\sum_{s\geq 1}s^{-q_0/u_1+\frac{1}{2(1-1/q_0-\delta)}}<\infty
$
and so for some $C>0$,
$$
\|I_{n,r_n(\cdot)}(\cdot)-\mathbb E_{\mathbb P}[I_{n,r_n(\cdot)}(\cdot)]\|_{L^{u_0}}\leq Cn^{\frac12(1+1/q_0+\delta)}.
$$
By taking $\delta$ small enough we see that we can take $u_0$   arbitrarily close to $u^*$ defined by $\frac{1}{u^*}=\frac{1}{v'}+\frac{3q_0-2}{2(q_0^2-q_0)}$ (but smaller).
Combining the above estimates 
we conclude that with $u_4=\min(u_0,a_1)$ there is a constant $C>0$ such that
$$
\|I_{n}-\mathbb E_{\mathbb P}[I_n]\|_{L^{u_4}}\leq C\big(n^{\frac12(1+1/q_0+\delta)}+n^{1-\beta(1/q_0+\delta)}\big).
$$
Therefore, by Lemma \ref{Growth lemma} for every $\varepsilon'>1/u_4$ there exists $B(\cdot)\in L^{u_4}$ such that
$$
\left|I_{n}-\mathbb E_{\mathbb P}[I_n]\right|\leq B(\om)n^{\zeta+\varepsilon'}
$$
where
$
\zeta=\max\left(\frac12(1+1/q_0+\delta),1-\beta(1/q_0+\delta)\right).
 $
That is, we have
\begin{eqnarray}\label{I n}
  &
\left|\frac1n\sum_{j=0}^{n-1}\mu_{\te^j\om}(\bar f_{\te^j\om}^2) 
-\int_\Omega \mu_\om(\bar f_{\om}^2)\right|\leq B(\om)n^{-(1-\zeta-\varepsilon')}.
\end{eqnarray}

Next we estimate $J_n$ in \eqref{decom}. 
First, notice that with 
$
b(k)=\int_{\Omega}\mu_\om(\bar f_\om\cdot(\bar f_{\te^k\om})\circ T_\om^k)d\mathbb P(\om)
$
we have
$
\mathbb E_{\mathbb P}[J_n]=\sum_{k=1}^{n-1}(n-k)b(k).
$
Thus,
\begin{eqnarray*}
&
\left|\frac1n\mathbb E_{\mathbb P}[J_n]-\sum_{k=1}^{\infty}b(k)\right|\leq n^{-1}\sum_{k\geq 1}k|b(k)|+\sum_{k\geq n}|b(k)|.
\end{eqnarray*}
Since 
\begin{equation}\label{repres}
\mu_{\te^i\om}(\bar f_{\te^i\om}(\bar f_{\te^j\om}\circ T_{\te^i\om}^{j-i}))=\mu_{\te^j\om}((L_{\te^i\om}^{j-i}\bar f_{\te^i\om})\bar f_{\te^j\om})    
\end{equation}
using \eqref{RPF2} and that $\|f_\om\|_\al\in L^{p_1}$ we see that there exists $C>0$ such that
\begin{equation}\label{estt}
 \left\|\mu_{\te^i\om}(\bar f_{\te^i\om}(\bar f_{\te^j\om}\circ T_{\te^i\om}^{j-i}))
\right\|_{L^{u_5}}\leq C(j-i)^{-\beta}
\end{equation}
where $u_5$ is given by $1/u_5=1/p_0+2/p_1$.
Since $u_5\geq 1$ we see that
$
|b(k)|\leq Ck^{-\beta}
$
and so 
\begin{eqnarray}\label{expec est}
&\left|\frac1n\mathbb E_{\mathbb P}[J_n]-\sum_{k=1}^{\infty}b(k)\right|\leq C/n+Cn^{1-\beta}.
\end{eqnarray}
Therefore, it is enough to estimate $\|J_n-\mathbb E[J_n]\|_{L^{u_5}}$.

Next, using \eqref{estt}, we see that 
$
\|J_n-\tilde J_{\varepsilon,n}\|_{L^{u_5}}
\leq Cn^{1-\varepsilon(\beta-1)}
$
for every $\varepsilon>0$, where
$$
\tilde J_{\varepsilon,n}=\sum_{i=0}^{n-1}\sum_{|j-i|\leq n^{\varepsilon}}\mu_{\te^i\om}(\bar f_{\te^i\om}(\bar f_{\te^j\om}\circ T_{\te^i\om}^{j-i}))
$$
and
we used that $\sum_{k\geq n^\varepsilon}k^{-\beta}=O (n^{\varepsilon(1-\beta)})$.
Now, since
$$
\mu_{\te^i\om}(\bar f_{\te^i\om}(\bar f_{\te^j\om}\circ T_{\te^i\om}^{j-i}))=m(h_{\te^i\om}\bar f_{\te^i\om}(\bar f_{\te^j\om}\circ T_{\te^i\om}^{j-i}))
$$
using \eqref{Approx cond} we see that for every $i<n$ we have
$$
\left|\mu_{\te^i\om}(\bar f_{\te^i\om}(\bar f_{\te^j\om}\circ T_{\te^i\om}^{j-i}))-m((\mathcal L_{i-r_n(\om)}^{r_n(\om)}\textbf{1})\bar f_{\te^i\om}(\bar f_{\te^j\om}\circ T_{\te^i\om}^{j-i}))\right|
$$
$$
\leq \|\bar f_{\te^i\om}\|_\infty  \|\bar f_{\te^j\om}\|_\infty R_3(\te^{i-r_n(\om)}\om)n^{-\beta(1/q_0+\delta)}\leq 4\|f_{\te^i\om}\|_\infty  \|f_{\te^j\om}\|_\infty R_3(\te^{i-r_n(\om)}\om)n^{-\beta(1/q_0+\delta)}.
$$
Similarly, recalling that $\bar f_{\om,r}=f_\om-m( f_\om \mathcal L_{\te^{-r}\omega}^r\textbf{1})$,
\begin{eqnarray*}
  &
\left|m((\mathcal L_{i-r_n(\om)}^{r_n(\om)}\textbf{1})\bar f_{\te^i\om}(\bar f_{\te^j\om}\circ T_{\te^i\om}^{j-i}))-m((\mathcal L_{i-r_n(\om)}^{r_n(\om)}\textbf{1})\bar f_{\te^i\om,r_n(\om)}(\bar f_{\te^j\om,r_n(\om)}\circ T_{\te^i\om}^{j-i}))\right|
\\
&\leq 4\|f_{\te^i\om}\|_\infty  \|f_{\te^j\om}\|_\infty R_3(\te^{i-r_n(\om)}\om)n^{-\beta(1/q_0+\delta)}.
\end{eqnarray*}
Thus, if we define 
$
\bar J_{\varepsilon,n}(\om)=
\sum_{i=0}^{n-1}\sum_{|j-i|\leq n^{\varepsilon}}m(\mathcal L_{i-r_n(\om)}^{r_n(\om)}\bar f_{\te^i\om,r_n(\om)}(\bar f_{\te^j\om,r_n(\om)}\circ T_{\te^i\om}^{j-i}))
$
then 
\begin{equation}\label{est J s}
\left\|\tilde J_{\varepsilon,n}(\om)-\bar J_{\varepsilon,n}(\om)\right\|_{L^{u_6}}\leq Cn^{1+\varepsilon-\beta(1/q_0+\delta)}     
\end{equation}
where $u_6$ is given by $1/u_6=2/p_1+1/p_0$.

Next, with 
$
A_{i,j,r}(\om)=m(\mathcal L_{\te^{i-r}\om}^{r}\textbf{1}\bar f_{\te^i\om,r}(\bar f_{\te^j\om,r}\circ T_{\te^i\om}^{j-i}))
$
we have 
$$
\bar J_{\varepsilon,n}(\om)=\sum_{s=1}^\infty\mathbb I(C(\om)=s)\sum_{i=0}^{n-1}\sum_{|j-i|\leq n^\varepsilon} A_{i,j,b_{s,n}}(\om)
$$
where we recall that $b_{s,n}=[s^{\frac{1}{1-1/q_0-\delta}}n^{1/q_0+\delta}]+1$. Note that $\sup_{j,i,r}\|A_{i,j,r}\|_{L^{p_1/2}}<\infty$ since $|A_{i,j,r}|\leq 4\|f_{\te^i\om}\|_\infty \|f_{\te^j\om}\|_\infty$.
By applying \cite[Proposition 7]{MPU06} with $X_i=\sum_{|j-i|\leq n^\varepsilon} A_{i,j,b_{s,n}}$ (for fixed $s$ and $n$) and taking into account \eqref{mix 1} and Remark \ref{Remm} we see that
$$
\left\|\sum_{i=0}^{n-1}\sum_{|j-i|\leq n^\varepsilon}(A_{i,j,b_{s,n}}-\bbE[(A_{i,j,b_{s,n}}])\right\|_{L^{v'}}\leq Cn^{\varepsilon/2}\left(s^{\frac{1}{1-1/q_0-\delta}}n^{1/q_0+\delta}\right)^{1/2}n^{1/2}.
$$
Here we used
 that $X_i$ is measurable with respect to $\mathcal F_{i-n^\varepsilon-r,i+n^\varepsilon}$.
Thus, if $1/b_1=1/v'+1/b_2$ for some $b_2$ such that $q_0/b_2>1+\frac{1}{2(1-1/q_0-\delta)}$ then
$$
\|\bar J_{\varepsilon,n}\|_{L^{b_1}}\leq Cn^{\frac12(1+1/q_0+\delta)}\sum_{s\geq 1}s^{-q_0/b_2+\frac{1}{2(1-1/q_0-\delta)}}\leq C'n^{\frac12(1+1/q_0+\delta+\varepsilon)}.
$$
By taking $\delta$ arbitrarily small and $b_2$ arbitrarily close to $q_0\left(1+\frac{1}{2(1-1/q_0-\delta)}\right)^{-1}$  we can ensure that $1/b_1$ is arbitrarily close to  $\frac{2}{p_1}+\frac{q_0-2}{2q_0(q_0-1)}$.
We conclude that with $\bar u=\min(v',b_1,u_6,u_5)$,
\begin{eqnarray*}
&    \left\|\frac 1n J_n(\om)-\sum_{k\geq 1}b(k)\right\|_{L^{\bar u}}\leq C(n^{-1}+n^{1-\beta}+n^{\varepsilon-\beta(1/q_0+\delta)}+n^{\frac12(1/q_0+\delta+\varepsilon-1)}+n^{\varepsilon(\beta-1)}).
\end{eqnarray*}
Thus, by Lemma \ref{Growth lemma} for every $\varepsilon''>1/\bar u$ there is a random variable $B_1\in L^{\bar u}$ such that 
\begin{eqnarray}\label{J n}
&
\left|\frac 1n J_n(\om)-\sum_{k\geq 1}b(k)\right|\leq B_1(\om)n^{-\gamma_\varepsilon+\varepsilon''}
\end{eqnarray}
where 
$
\gamma_\varepsilon=\min\left(\beta(1/q_0+\delta)-\varepsilon, \frac12(1-1/q_0-\delta-\varepsilon), \varepsilon(\beta-1)\right).
$

Combining \eqref{decom}, \eqref{I n} and \eqref{J n}
we conclude that
\begin{eqnarray*}
  &
\left|\frac 1n\Sigma_{\om,n}-\Sigma^2\right|\leq B(\om)n^{-(1-\zeta-\varepsilon')}+B_1(\om)n^{-\gamma_\varepsilon}.
\end{eqnarray*}
To complete the proof of the theorem we take $\varepsilon=V/\beta$, $V=\min(\beta/q_0, 1/2-1/2q_0)$ and then $\delta$ small enough, $\varepsilon'$ arbitrarily close to $1/u_4$ and $\varepsilon''$ close enough to $\bar u$.

Finally the characterization of $\Sigma^2=0$ follows by applying a standard result (e.g. \cite{BrMix}) with the stationary process $F_n(\om,x)=\bar f_{\te^n\om}(T_\om^nx)$.
\end{proof}

\subsection{Effective growth rates of the $k$-th visiting times to level sets}
Let $A$ be a set with $\bbP(A)>0$ and let $m_k(\om)$ be the $k$-th visiting time to $A$ by the $\te$ orbit of $\om$. 

\begin{lemma}
If $m_1\in L^p$, $p>0$ then for every $0<\delta<p$ there exists $B\in L^{p/2-\delta}$ such that for $\bbP$-a.e. $\om$, 
\begin{equation}\label{m k as}
m_k(\om)\leq B(\om)k^{1+1/p+\delta}.    
\end{equation}
 Hence the conditions of Section \ref{Sec SDS1} hold with $\eta_0=1/p+\delta$ and $E_2=B(\om)$.
\end{lemma}
\begin{proof}
 Let $\hat \te=\te^{m_1(\om)}\om$ and $\te_A=\hat\te|_{A}$. Then for $\om\in A$ we have $m_k(\om)=\sum_{j=0}^{k-1}m_1(\te_A^j\om)$. Since $\te_A$ preserves $\bbP_A=\bbP(\cdot|A)$ we see that $\|m_k\|_{L^p(\bbP_A)}\leq k\|m_1\|_{L^p(\bbP_A)}\leq Ck$ with $C=C_A=\|m_1\|_{L^p}/\bbP(A)$. By Lemma \ref{Growth lemma} applied with the map $\te_A$ and the measure $\bbP_A=\bbP(\cdot|A)$ for every $\epsilon>0$ there exists $H\in L^p(\bbP_A)$ such that $m_k(\omega)\leq H(\omega)k^{1+1/p+\epsilon}$ for $\bbP_A$-a.e. $\om$. Now, since $m_k(\omega)=m_{k}(\hat\te \omega)$ we conclude that $m_k(\om)\leq H(\hat\te\omega)k^{1+1/p+\delta}$. Now, let us take $d=p-\epsilon$ and let $q$ be defined by $1/q=1/p+1/d$. Set also $H_A=H\cdot \bbI_A$. Then by the H\"older inequality, 
$$
\|H(\hat\te\omega)\|_{L^q}=\left\|\sum_{s}\bbI(m_1=s)H_A(\te^s\omega)\right\|_{L^q}\leq \sum_{s}(\bbP(m_1=s))^{1/d}\|H_A\|_{L^p}<\infty
$$
where we used that by the Markov inequality $\bbP(m_1=s)\leq \bbP(m_1\geq s)=O(s^{-p})$ and that $d<p$ and $H_A\in L^p$. Now the lemma follows by taking $\epsilon$ small enough.
\end{proof}

\begin{corollary}\label{Cor a1} Let $B(\om)$ like in \eqref{m k as}. Then for every $e>1/a$, $a=p_2-\delta$ there is $B_1\in L^{a}$ with $$m_{k+1}(\om)-m_k(\om)\leq B_1(\om)(B(\om))^{1/a+e}k^{(1/a+e)(1+1/p+\delta)}.$$   
Hence the conditions of Section \ref{Sec SDS1} hold with $\delta=(1/a_1+e)(1+\epsilon_1)$ and $E_3=B_1(\om)(B(\om))^{1/a_1+e}$.
\end{corollary}
\begin{proof} We have $m_{k+1}(\om)-m_k(\om)=m_1(\te^{m_k(\om)}\om)\leq B(\te^{m_k(\om)}\om).$
 By Lemma \ref{Growth lemma} for every $e>1/a_1$ there is $B_1\in L^{a}, $ with $B(\te^j\omega)\leq B_1(\om)j^{1/a_1+e}, j>0$. Now the result follows by taking $j=m_k(\om)$ in \eqref{m k as}.
\end{proof}

\subsection{Proof of Theorem \ref{BE2}}
First,  by Lemma \ref{Growth lemma}  for $Y(\cdot)\in L^s$   for all $r>0$ we have $Y(\te^j\om)\leq E(\om)j^{1/s+r}$ with $E\in L^{s}$. 
Note also that in Lemma \ref{L9} we can take $\eta=\eta(p)$ with $\lim_{p\to\infty}\eta(p)=0$, and that we can ensure that $D_4=D_4(\om)\in L^{a(p)}$ with $\lim_{p\to\infty}a(p)=\infty$.
Using this together with Lemma \ref{a1} and Corollary \ref{Cor a1}  the arguments in the proof of Theorem \ref{BE1}  for $n\geq N(\om)=\max(q(\om),(4D_4(\om)/\Sigma)^{\frac2{1-2\eta(p)}})$ yield rates of the form $B_p(\om)n^{-1/2+\delta(p)}$ with $B_p\in L^{1/\delta(p)}$ and $\delta(p)\to 0$ as $p\to\infty$. Notice that 
 $N(\om)\in L^{u(p)}$ with $\lim_{p\to\infty}u(p)=\infty$. Thus, for $n\leq N(\om)$,
$$
\Delta_{\om,n}=:\sup_{t\in\mathbb R}\left|\mathbb P(\bar S_n^\om f\leq \Sigma_{\om,n} t)-\Phi(t)\right|\leq 1\leq n^{-1/2}(N(\om))^{1/2}.
$$
We thus conclude that for all $n\geq 1$ we have $\Delta_{\om,n}\leq B_p(\om)n^{-1/2+\delta(p)}$ with $\delta(p)\to 0$ as $p\to\infty$ and $B_p\in L^{1/\delta(p)}$.
Next by \cite[Lemma 3.3]{HK SPA}, 
$$
\sup_{t\in\mathbb R}\left|\mathbb P(\bar S_n^\om f\leq \sqrt {\Sigma^2 n} t)-\Phi(t)\right|
\leq 3\Delta_{\om,n}
+10|\Sigma_{\om,n}^{-1}-\Sigma^{-1}n^{-1/2}|^{\frac{p}{p+1}}\|\bar S_{n}^\om f\|_{L^p}^{\frac{p}{p+1}}.
$$
Using that $x^{-1/2}-y^{-1/2}=\frac{x-y}{\sqrt{xy}(x+y)}$ for all $x,y>0$ and Theorem \ref{VarRate} we see that for $n\geq q(\om)$,
$$
|\Sigma_{\om,n}^{-1}-\Sigma^{-1}n^{-1/2}|\leq Cn^{-1/2-\iota+\kappa} B_\kappa(\om)
$$
for some constant $C>0$ and all $0<\kappa<\iota$. Notice  that in Theorem \ref{VarRate} we can take $p^*\to \infty$ as $p\to\infty$ and $\iota\to 1/2$ as $p\to\infty$. We also take $\ka<1/p$. Observe also that in Lemma \ref{L15} we can ensure that $\zeta\to 0$ as $p\to \infty$ and that   $E_{12}\in L^{b(p)}$ with $b(p)\to \infty$ as $p\to\infty$. Applying  Lemma \ref{L15} we conclude that
 the  estimates in Theorem \ref{BE2} hold for all $n\geq q(\om)$. The estimates for $n<q(\om)$ follow like in the self normalized case.  
 The proof of Theorem \ref{BE2} is completed. \qed

\begin{acknowledgement*}
I would like to thank Davor Dragi\v{c}evi\'c for several discussions and suggestions and for carefully reading several parts of the manuscript.
\end{acknowledgement*}

\end{document}